\documentclass[11pt]{article}
\usepackage{amsthm}
\usepackage{amsmath}
\usepackage{amssymb}

\theoremstyle{definition}
\newtheorem{definition}{Definition}[section]

\newtheorem{example}[definition]{Example}

\newtheorem{remark}[definition]{Remark}

\theoremstyle{plain}
\newtheorem{theorem}[definition]{Theorem}
\newtheorem{lemma}[definition]{Lemma}
\newtheorem{proposition}[definition]{Proposition}
\newtheorem{corollary}[definition]{Corollary}

\newcommand{\dotvee}{\stackrel{\bullet }{\vee }}

\numberwithin{equation}{section}

\def\N{{\mathbb N}}

\begin{document}
\title{\vspace*{-35pt}Stone Commutator Lattices and Baer Rings}
\author{Claudia MURE\c SAN\\ {\small University of Cagliari, University of Bucharest}\\ {\small c.muresan@yahoo.com, cmuresan@fmi.unibuc.ro}}
\date{\today }
\maketitle

\begin{abstract} In this paper, we transfer Davey`s characterization for $\kappa $--Stone bounded distributive lattices to lattices with certain kinds of quotients, in particular to commutator lattices with certain properties, and obtain related results on prime, radical, complemented and compact elements, annihilators and congruences of these lattices. We then apply these results to certain congruence lattices, in particular to those of semiprime members of semi--degenerate congruence--modular varieties, and use this particular case to transfer Davey`s Theorem to commutative unitary rings.


\end{abstract}

\section{Introduction}
\label{introduction}

We shall refer to \cite[Theorem $1$]{dav} as {\em Davey`s Theorem}. Given an arbitrary infinite cardinality $\kappa $, Davey`s Theorem provides a characterization for $\kappa $--Stone bounded distributive lattices: those bounded distributive lattices with the property that the annihilators of their subsets of cardinality at most $\kappa $ are principal ideals generated by elements from their Boolean center.

It turns out that commutator lattices with certain properties, in particular congruence lattices of semiprime algebras from semi--degenerate con\-gru\-ence--modular varieties, satisfy the equivalences from Davey`s Theorem; moreover, changing the cardinalities in those equivalent conditions to any nonzero value produces more properties equivalent to those conditions; furthermore, by restricting the values of these cardinalities, we obtain a generalization of this equivalence result to a generalization of commutator lattices.

To prove this, we first transfer Davey`s Theorem to bounded lattices from certain quotients of theirs which are distributive, or frames, or a generalization of frames, and then apply this transfer to commutator lattices satisfying certain conditions and a certain quotient of such commutator lattices.

Then we apply the previous result to the ideal lattices of commutative unitary rings, from which we transfer this result to the elements of these rings.

\section{Definitions and Notations}

We shall denote by $\N $ the set of the natural numbers and by $\N ^*=\N \setminus \{0\}$. For any set $S$, $|S|$ will denote the cardinality of $S$.

Throughout this paper, all algebras shall be nonempty and, unless there is danger of confusion, they will be designated by their underlying sets. By {\em trivial algebra} we mean one--element algebra. Recall that a variety ${\cal V}$ is said to be {\em semi--degenerate} iff no nontrivial algebra in ${\cal V}$ has trivial subalgebras. For any algebra $A$, $({\rm Con}(A),\vee ,\cap ,\Delta _A,\nabla _A)$ shall be the bounded lattice of the congruences of $A$, with the exception of the case when $A$ is a commutator lattice, case in which ${\rm Con}(A)$ will denote the lattice of congruences of the lattice reduct of $A$; for any $X\subseteq A^2$ and any $a,b\in A$, $Cg_A(X)$ shall be the congruence of $A$ generated by $X$ and we will denote $Cg_A(a,b)=Cg_A(\{(a,b)\})$; the set of the principal congruences of $A$ will be denoted by ${\rm PCon}(A)$. Recall that the compact congruences of $A$, that is the compact elements of the lattice ${\rm Con}(A)$, are exactly the finitely generated congruences of $A$. For any $\theta \in {\rm Con}(A)$, $p_{\theta }:A\rightarrow A/\theta $ will be the canonical surjection. For any algebra $L$ having a lattice reduct (in particular for any commutator lattice $L$), ${\rm Con}(L)$ will denote the set of the congruences of the lattice reduct of $L$.

Now let $L$ be an arbitrary lattice. We denote by ${\rm Cp}(L)$, ${\rm Mi}(L)$ and ${\rm Smi}(L)$ the sets of the compact, the meet--irreducible and the strictly meet--irreducible elements of $L$, respectively. Recall that $L$ is said to be {\em compact} iff ${\rm Cp}(L)=L$ and $L$ is said to be {\em algebraic} iff each of its elements is a join of compact elements. Note that, if $L$ is compact, then the join of any nonempty $U\subseteq L$ equals the join of a finite subset of $U$, and that, if $L$ has finite length, then $L$ is compact, thus $L$ is algebraic. Note that, if $L$ has a $1$, then $1\notin {\rm Smi}(L)$, because $\displaystyle 1=\bigwedge \emptyset =\bigwedge \{x\in L\ |\ 1<x\}$. For each $a\in {\rm Smi}(L)$, we shall denote by $\displaystyle a^{+}=\bigwedge \{x\in L\ |\ a<x\}$ the unique successor of $a$ in $L$. If $L$ has a $1$, then we shall denote by ${\rm Max}_L$ the set of the maximal elements of the ordered set $(L\setminus \{1\},\leq )$. For any algebra $A$, ${\rm Max}_{{\rm Con}(A)}$ will simply be denoted by ${\rm Max}(A)$. ${\rm Filt}(L)$ and ${\rm Id}(L)$ shall be the bounded lattices of the filters and ideals of $L$, respectively, and ${\rm PId}(L)$ will be the bounded sublattice of ${\rm Id}(L)$ of the principal ideals of $L$. Note that a filter of $L$ is principal iff it has a minimum and an ideal of $L$ is principal iff it has a maximum. Recall, also, that any class of a congruence of $L$ is a convex sublattice of $L$, so it has a unique writing as the intersection between a filter and an ideal of $L$.

Let $U\subseteq L$ and $a,b\in L$, arbitrary. We denote by $(U]_L$ and $[U)_L$ the ideal and the filter of $L$ generated by $U$, respectively, by $(a]_L=(\{a\}]_L$ and $[a)_L=[\{a\})_L$ and, to avoid overlapping with the classical notation for the commutator operation in commutator lattices (see Section \ref{theth}), $\lfloor a,b\rceil _L=[a)_L\cap (b]_L$ will be the notation for intervals; if $L$ is the chain of natural numbers with the natural order, then we denote $\overline{a,b}=\lfloor a,b\rceil _L$. Note that, for any lattice $M$ and any surjective lattice morphism $f:L\rightarrow M$, the map $I\mapsto f(I)$ is a complete lattice morphism from ${\rm Id}(L)$ to ${\rm Id}(M)$ that satisfies $f((U]_L)=(f(U)]_M$; in particular, for any $\theta \in {\rm Con}(L)$, we have $(U]_L/\theta =(U/\theta ]_{L/\theta }$, so $(a]_L/\theta =(a/\theta ]_{L/\theta }$.

${\rm Ann}_L(a)$ and ${\rm Ann}_L(U)$ shall be the {\em annihilator} of $a$ and $U$ in $L$, respectively: ${\rm Ann}_L(a)\linebreak =\{x\in L\ |\ x\wedge a=0\}$ and $\displaystyle {\rm Ann}_L(U)=\bigcap _{u\in U}{\rm Ann}_L(u)$. We will denote by: ${\cal A}nn(L)=\{{\rm Ann}_L(U)\ |\ U\subseteq L\}$, ${\rm PAnn}(L)=\{{\rm Ann}_L(a)\ |\ a\in L\}$, ${\rm P2Ann}(L)=\{{\rm Ann}_L({\rm Ann})_L(a))\ |\ a\linebreak \in L\}$ and ${\rm 2Ann}(L)=\{{\rm Ann}_L({\rm Ann}_L(U))\ |\ U\subseteq L\}$. The following notations will also be useful: let ${\cal A}nn_{<\infty }(L)=\{{\rm Ann}_L(U)\ |\ U\subseteq L,|U|<\aleph _0\}$ and ${\rm 2Ann}_{<\infty }(L)=\{{\rm Ann}_L({\rm Ann}_L(U))\ |\ U\subseteq L,|U|<\aleph _0\}$ and, if $\kappa $ is a cardinality, let ${\cal A}nn_{\kappa }(L)=\{{\rm Ann}_L(U)\ |\linebreak U\subseteq L,|U|\leq \kappa \}$ and ${\rm 2Ann}_{\kappa }(L)=\{{\rm Ann}_L({\rm Ann}_L(U))\ |\ U\subseteq L,|U|\leq \kappa \}$.

${\cal B}(L)$ will denote the set of the complemented elements of the bounded lattice $L$, regardless of whether $L$ is distributive. Unless mentioned otherwise, we shall denote by $\neg \, $ the complementation in every Boolean algebra.

Recall that the bounded lattice $L$ is said to be {\em Stone}, respectively {\em strongly Stone}, iff, for all $a\in L$, respectively all $U\subseteq L$, there exists an $e\in {\cal B}(L)$ such that ${\rm Ann}_L(a)=(e]_L$, respectively ${\rm Ann}_L(U)=(e]_L$. Also, for any cardinality $\kappa $, $L$ is said to be {\em $\kappa $--Stone} iff, for all $U\subseteq L$ with $|U|\leq \kappa $, there exists an $e\in {\cal B}(L)$ such that ${\rm Ann}_L(U)=(e]_L$.

Remember that $L$ is called a {\em frame} iff $L$ is complete and the meet in $L$ is completely distributive w.r.t. the join.

\section{The Theorem We Are Going to Transfer to Commutator Lattices, then to Commutative Unitary Rings}
\label{theth}

Throughout this section, $L$ will be a bounded lattice. We shall use the following notations for these conditions on $L$, where $\kappa $ is an arbitrary cardinality:

\begin{flushleft}\begin{tabular}{ll}
$(1)_{\kappa ,L}$ & $L$ is a $\kappa $--Stone lattice;\\ 
$(1)_{<\infty ,L}$ & ${\cal A}nn_{<\infty }(L)\subseteq \{(e]_L\ |\ e\in {\cal B}(L)\}$;\\
$(1)_L$ & $L$ is a strongly Stone lattice;\end{tabular}

\begin{tabular}{ll}
$(2)_{\kappa ,L}$ & $L$ is a Stone lattice and ${\cal B}(L)$ is a $\kappa $--complete Boolean sublattice of $L$;\\
$(2)_{<\infty ,L}$ & $L$ is a Stone lattice and ${\cal B}(L)$ is a Boolean sublattice of $L$;\\ 
$(2)_L$ & $L$ is a Stone lattice and ${\cal B}(L)$ is a complete Boolean sublattice of $L$;\end{tabular}

\begin{tabular}{ll}
$(3)_{\kappa ,L}$ & ${\rm P2Ann}(L)$ is a $\kappa $--complete Boolean sublattice of ${\rm Id}(L)$ such that\\ 
& $a\mapsto {\rm Ann}_L({\rm Ann}_L(a))$ is a lattice morphism from $L$ to ${\rm P2Ann}(L)$;\\ 
$(3)_{<\infty ,L}$ & ${\rm P2Ann}(L)$ is a Boolean sublattice of ${\rm Id}(L)$ such that\\ 
& $a\mapsto {\rm Ann}_L({\rm Ann}_L(a))$ is a lattice morphism from $L$ to ${\rm P2Ann}(L)$;\\ 
$(3)_L$ & ${\rm P2Ann}(L)$ is a complete Boolean sublattice of ${\rm Id}(L)$ such that\\ 
& $a\mapsto {\rm Ann}_L({\rm Ann}_L(a))$ is a lattice morphism from $L$ to ${\rm P2Ann}(L)$;\end{tabular}

\begin{tabular}{ll}
$(4)_{\kappa ,L}$ & for all $a,b\in L$, ${\rm Ann}_L(a\wedge b)=({\rm Ann}_L(a)\cup {\rm Ann}_L(b)]_L$, and\\ 
& ${\rm 2Ann}_{\kappa }(L)\subseteq {\rm PAnn}(L)$;\\ 
$(4)_{<\infty ,L}$ & for all $a,b\in L$, ${\rm Ann}_L(a\wedge b)=({\rm Ann}_L(a)\cup {\rm Ann}_L(b)]_L$, and\\ 
& ${\rm 2Ann}_{<\infty }(L)\subseteq {\rm PAnn}(L)$;\\ 
$(4)_L$ & for all $a,b\in L$, ${\rm Ann}_L(a\wedge b)=({\rm Ann}_L(a)\cup {\rm Ann}_L(b)]_L$, and\\ 
& ${\rm 2Ann}(L)\subseteq {\rm PAnn}(L)$;\\ 
$(iv)_L$ & for all $a,b\in L$, ${\rm Ann}_L(a\wedge b)=({\rm Ann}_L(a)\cup {\rm Ann}_L(b)]_L$;\end{tabular}

\begin{tabular}{ll}
$(5)_{\kappa ,L}$ & for each $U\subseteq L$ with $|U|\leq \kappa $, $({\rm Ann}_L(U)\cup {\rm Ann}_L({\rm Ann}_L(U))]_L=L$;\\ 
$(5)_{<\infty ,L}$ & for each finite $U\subseteq L$, $({\rm Ann}_L(U)\cup {\rm Ann}_L({\rm Ann}_L(U))]_L=L$;\\ 
$(5)_L$ & for each $U\subseteq L$, $({\rm Ann}_L(U)\cup {\rm Ann}_L({\rm Ann}_L(U))]_L=L$.\end{tabular}\end{flushleft}

Of course, annihilators are nonempty, since each of them contains $0$. Note also that, for any $U\subseteq V\subseteq L$, we have ${\rm Ann}_L(V)\subseteq {\rm Ann}_L(U)$, hence ${\rm Ann}_L({\rm Ann}_L(U))\subseteq {\rm Ann}_L({\rm Ann}_L(V))$.

Since $0,1\in {\cal B}(L)$, we have ${\rm Ann}_L(\emptyset )=L=(1]_L={\rm Ann}_L(0)\in {\rm PAnn}(L)\cap \{(e]_L\ |\ e\in {\cal B}(L)\}$ and ${\rm Ann}_L({\rm Ann}_L(\emptyset ))={\rm Ann}_L(L)=\{0\}=(0]_L={\rm Ann}_L(1)\in {\rm PAnn}(L)\cap \{(e]_L\ |\ e\in {\cal B}(L)\}$, hence conditions $(1)_{0,L}$ and $(5)_{0,L}$ are trivially satisfied; we may also note that ${\rm PAnn}(L)={\cal A}nn_1(L)$ and ${\rm P2Ann}(L)={\rm 2Ann}_1(L)$.

Clearly, if $L$ is distributive, then every annihilator of $L$ is an ideal of $L$.

\begin{remark} If $L$ is a frame, then every annihilator of $L$ is a principal ideal of $L$. Indeed, if $L$ is a frame, then, clearly, for all $U\subseteq L$, $\bigvee {\rm Ann}_L(U)\in {\rm Ann}_L(U)$, hence the ideal ${\rm Ann}_L(U)$ is principal.\label{frameannpid}\end{remark}

Of course, if ${\cal B}(L)$ is a distributive sublattice of $L$, in particular if $L$ is distributive, then ${\cal B}(L)$ is a Boolean sublattice of $L$.

Note that, for any bounded lattice $L$ and any cardinalities $\kappa \leq \mu $ and any $i\in \overline{1,5}$:\begin{itemize}
\item $(4)_{\kappa ,L}$ implies $(iv)_L$;
\item $(i)_{\mu ,L}$ implies $(i)_{\kappa ,L}$, hence, if the converse implication holds, as well, then $(i)_{\kappa ,L}$ is equivalent to $(i)_{\nu ,L}$ for any cardinality $\nu $ with $\kappa \leq \nu \leq \mu $;
\item $(i)_{<\infty ,L}$ is equivalent to $(i)_{\nu ,L}$ being valid for all finite cardinalities $\nu $;
\item $(i)_L$ is equivalent to $(i)_{\nu ,L}$ being valid for all cardinalities $\nu $.\end{itemize}

For any nonempty family $(U_i)_{i\in I}$ of subsets of $L$, clearly $\displaystyle {\rm Ann}_L(\bigcup _{i\in I}U_i)=\linebreak \bigcap _{i\in I}{\rm Ann}_L(U_i)$. For any family $(a_i)_{i\in I}\subseteq L$ having a meet in $L$, we have $\displaystyle \bigcap _{i\in I}(a_i]_L=(\bigwedge _{i\in I}a_i]_L$. Trivially, if $L$ is strongly Stone, then $L$ is Stone, and, by the above, the converse holds if ${\cal A}nn(L)={\rm PAnn}(L)$ and ${\cal B}(L)$ or ${\cal B}(L)$ is closed w.r.t. arbitrary meets; in particular, $(2)_L$ implies $(1)_L$.

If $L$ is distributive, then, for any $n\in \N ^*$ and any $u_1,\ldots ,u_n\in L$, ${\rm Ann}_L(\{u_1,\ldots ,u_n\})={\rm Ann}_L(u_1\vee \ldots \vee u_n)\in {\rm PAnn}(L)$, so ${\cal A}nn_{<\infty }(L)={\rm PAnn}(L)$, hence $(1)_{1,L}$ is equivalent to $(1)_{<\infty ,L}$, that is $L$ is a Stone lattice iff it satisfies $(1)_{<\infty ,L}$, and it immediately follows that ${\rm Ann}_L(U)={\rm Ann}_L((U]_L)$ for all $U\subseteq L$ and thus, for any family $(I_k)_{k\in K}$ of ideals of $L$, $\displaystyle {\rm Ann}_L(\bigvee _{k\in K}I_k)={\rm Ann}_L((\bigcup _{k\in K}I_k]_L)={\rm Ann}_L(\bigcup _{k\in K}I_k)=\bigcap _{k\in K}{\rm Ann}_L(I_k)$.

\begin{remark} Let $\kappa $ be a nonzero cardinality.

If $(x_i)_{i\in I}\subseteq L$ such that $\displaystyle \bigvee _{i\in I}x_i\in L$ and $\displaystyle x\wedge (\bigvee _{i\in I}x_i)=\bigvee _{i\in I}(x\wedge x_i)$ for all $x\in L$, then clearly $\displaystyle {\rm Ann}_L(\{x_i\ |\ i\in I\})={\rm Ann}_L(\bigvee _{i\in I}x_i)\in {\rm PAnn}(L)$.

Thus, if $L$ is closed w.r.t. the joins of all families of elements of cardinality at most $\kappa $ and has the meet distributive w.r.t. the joins of families of cardinalities at most $\kappa $, then ${\rm Ann}_L(U)={\rm Ann}_L(\bigvee U)$ for all $U\subseteq L$ with $|U|\leq \kappa $, thus ${\cal A}nn_{\kappa }(L)={\rm PAnn}(L)$.

Also, if $L$ is a frame, then ${\rm Ann}_L(U)={\rm Ann}_L(\bigvee U)$ for all $U\subseteq L$, thus ${\cal A}nn(L)={\rm PAnn}(L)$.

If ${\cal A}nn_{\kappa }(L)={\rm PAnn}(L)$, in particular if $L$ is closed w.r.t. the joins of families of elements of cardinality at most $\kappa $ and has the meet distributive w.r.t. such joins, then $(1)_{1,L}$ is equivalent to $(1)_{\kappa ,L}$ (thus to $(1)_{\lambda ,L}$ for any nonzero cardinality $\lambda \leq \kappa $), and ${\rm 2Ann}_{\kappa }(L)={\rm P2Ann}(L)$, hence ${\rm P2Ann}(L)\subseteq {\rm PAnn}(L)$ iff ${\rm 2Ann}_{\kappa }(L)\subseteq {\rm PAnn}(L)$, thus $(4)_{1,L}$ is equivalent to $(4)_{\kappa ,L}$ (thus to $(4)_{\lambda ,L}$ for any nonzero cardinality $\lambda \leq \kappa $).

If ${\cal A}nn(L)={\rm PAnn}(L)$, in particular if $L$ is a frame, then $(1)_{1,L}$ is equivalent to $(1)_L$ (thus to $(1)_{\lambda ,L}$ for any nonzero cardinality $\lambda $), and ${\rm 2Ann}(L)={\rm P2Ann}(L)\subseteq {\cal A}nn(L)={\rm PAnn}(L)$, thus the second part of condition $(4)_L$ is satisfied, which means that $(4)_L$ is equivalent to $(iv)_L$ and thus to $(4)_{\lambda ,L}$ for any nonzero cardinality $\lambda $.\label{mcplt}\end{remark}

As an example, note that any Boolean lattice is Stone, because, if $L$ is Boolean, then ${\rm Ann}_L(e)=(\neg \, e]_L$ for all $e\in L$, thus, by the above, any complete Boolean lattice is strongly Stone.

\begin{theorem}\begin{enumerate}
\item\label{davey1} If $L$ is a bounded distributive lattice, then the conditions $(1)_{\kappa ,L}$, $(2)_{\kappa ,L}$, $(3)_{\kappa ,L}$, $(4)_{\kappa ,L}$ and $(5)_{\kappa ,L}$ are equivalent for any nonzero cardinality $\kappa $.
\item\label{davey3} If $L$ is a bounded distributive lattice, then conditions $(1)_L$, $(2)_L$, $(3)_L$, $(4)_L$ and $(5)_L$ are equivalent.
\item\label{davey0} Let $m$ be a nonzero cardinality. If $L$ is a bounded distributive lattice having ${\cal A}nn_m(L)={\rm PAnn}(L)$, in particular if the bounded lattice $L$ is closed w.r.t. the joins of all families of elements of cardinality at most $m$ and has the meet distributive w.r.t. such joins then, for any $h,i\in \overline{1,5}$ and any nonzero cardinality $\kappa \leq m$, conditions $(h)_{\kappa ,L}$ and $(i)_{<\infty ,L}$ are equivalent, in particular the Boolean center of $L$ is $m$--complete and $L$ is Stone iff it is $m$--Stone.
\item\label{davey2} If $L$ is a bounded distributive lattice with ${\cal A}nn(L)={\rm PAnn}(L)$, in particular if $L$ is a frame, then, for any $h,i,j\in \overline{1,5}$ and any nonzero cardinality $\kappa $, conditions $(iv)_L$, $(h)_{\kappa ,L}$, $(i)_{<\infty ,L}$ and $(j)_L$ are equivalent, in particular the Boolean center of $L$ is complete and $L$ is Stone iff it is strongly Stone.\end{enumerate}\label{davey}\end{theorem}

\begin{proof} (\ref{davey1}) is \cite[Theorem $1$]{dav} for $\kappa $ infinite and part of Remark \ref{mcplt} for $\kappa $ finite, and it clearly implies (\ref{davey3}).

\noindent (\ref{davey0}),(\ref{davey2}) By (\ref{davey1}) and Remark \ref{mcplt}.\end{proof}

\begin{definition}{\rm \cite{erhard},\cite{cze2}} Let $[\cdot ,\cdot ]$ be a binary operation on $L$. The algebra $(L,\vee ,\wedge ,[\cdot ,\cdot ],0,1)$ (which we shall also denote, simply, by $(L,[\cdot ,\cdot ])$) is called a {\em commutator lattice} and the operation $[\cdot ,\cdot ]$ is called {\em commutator} iff $(L,\vee ,\wedge ,0,1)$ is a complete lattice with lattice bounds $0$ and $1$ and, for all $x,y\in L$ and any family $(y_i)_{i\in I}\subseteq L$:\begin{itemize}
\item $[x,y]=[y,x]\leq x\wedge y$ ($[\cdot ,\cdot ]$ is commutative and smaller than its arguments);
\item $\displaystyle [x,\bigvee _{i\in I}y_i]=\bigvee _{i\in I}[x,y_i]$ ($[\cdot ,\cdot ]$ is completely distributive w.r.t. the join).\end{itemize}\label{defcommlat}\end{definition}

\begin{remark} For any complete lattice $L$, we have the equivalence: $(L,[\cdot ,\cdot ])$ is a commutator lattice with $[\cdot ,\cdot ]=\wedge $ iff $L$ is a frame.\label{framecommlat}\end{remark}

Let $(L,\vee ,\wedge ,[\cdot ,\cdot ],0,1)$ be a commutator lattice. We call $p$ a {\em prime element} of $L$ iff $p\in L\setminus \{1\}$ and, for all $a,b\in L$, if $[a,b]\leq p$, then $a\leq p$ or $b\leq p$. We denote by ${\rm Spec}_L$ the set of the prime elements of $L$. Note that, if $[\cdot ,\cdot ]=\wedge $, then ${\rm Spec}_L$ is the set of the meet--prime elements of $L$ and $L$ is distributive (actually a frame, by Remark \ref{framecommlat}), hence ${\rm Spec}_L={\rm Mi}(L)\setminus \{1\}\supseteq {\rm Smi}(L)$.

For any $x\in L$, we denote by $V(x)=[x)_L\cap {\rm Spec}_L$, by $\rho (x)=\bigwedge V(x)=\bigwedge \{p\in {\rm Spec}_L\ |\ x\leq p\}$ and by $R(L)=\{\rho (x)\ |\ x\in L\}$. We call $\rho (x)$ the {\em radical of $x$}, and the elements of $R(L)$ {\em radical elements} of $L$. Clearly, ${\rm Spec}_L\subseteq R(L)=\{x\in L\ |\ \rho (x)=x\}$.

Recall that all elements of an algebraic lattice are meets of strictly meet--irreducible elements, thus, if $L$ is algebraic and $[\cdot ,\cdot ]=\wedge $, then $R(L)=L$; see also Remarks \ref{commmeet} and \ref{cgrad} and Proposition \ref{allradical}, (\ref{allradical2}), below.

\begin{example}{\rm \cite{agl},\cite{fremck}} If ${\cal V}$ is a congruence--modular variety, $A$ is a member of ${\cal V}$ and $[\cdot ,\cdot ]_A$ is the (modular) commutator of $A$, then $({\rm Con}(A),\vee ,\cap ,[\cdot ,\cdot ]_A,\Delta _A,\nabla _A)$ is a commutator lattice.\end{example} 

Let $A$ be an arbitrary member of a congruence--modular variety ${\cal V}$. Then we will denote the modular commutator of $A$ as above and the set ${\rm Spec}_{{\rm Con}(A)}$ of the prime elements of the commutator lattice $({\rm Con}(A),\vee ,\cap ,[\cdot ,\cdot ]_A,\Delta _A,\nabla _A)$, called {\em prime congruences} of $A$, by ${\rm Spec}(A)$. The elements of $R({\rm Con}(A))$ are called {\em radical congruences} of $A$. Recall that $A$ is said to be {\em semiprime} iff $\Delta _A$ is a radical congruence of $A$.

Recall that, if ${\cal V}$ is congruence--distributive, then ${\cal V}$ has no skew congruences and the commutator $[\cdot ,\cdot ]_A$ coincides to the intersection of congruences. If $[\cdot ,\cdot ]_A$ equals the intersection, then $[\theta ,\nabla _A]_A=\theta \cap \nabla _A=\theta $ for all $\theta \in {\rm Con}(A)$, and, by the above, $A$ is congruence--distributive and, moreover, ${\rm Con}(A)$ is a frame, and we have ${\rm Smi}({\rm Con}(A))\subseteq {\rm Mi}({\rm Con}(A))\setminus \{\nabla _A\}={\rm Spec}(A)$, so that $R({\rm Con}(A))={\rm Con}(A)$ since the lattice ${\rm Con}(A)$ is algebraic, in particular $A$ is semiprime.

If ${\cal V}$ is semi--degenerate, then ${\cal C}$ has no skew congruences \cite[Theorem 8.5, p. 85]{fremck}, $\nabla _A$ is a compact congruence of $A$, $[\theta ,\nabla _A]_A=\theta $ for all $\theta \in {\rm Con}(A)$, any maximal congruence of $A$ is prime and each proper congruence of $A$ is included in a prime congruence \cite[Theorem $5.3$]{agl}; see also Lemma \ref{smimaxspec} below.

\section{Transferring Conditions $(i)_{\kappa ,\cdot }$ between Bounded Lattices and Their Quotients, and Related Results}
\label{preservation}

Throughout this section, $M$ shall be an arbitrary bounded lattice and $\theta \in {\rm Con}(M)$, unless mentioned otherwise.

\begin{remark} Let $x\in M$ and $U\subseteq M$. Then, clearly, $x\in {\rm Ann}_M(U)$ implies $x/\theta \in {\rm Ann}_{M/\theta }(U/\theta )$, hence ${\rm Ann}_M(U)/\theta \subseteq {\rm Ann}_{M/\theta }(U/\theta )$.\label{inclann}\end{remark}

\begin{lemma} If $M$ is a bounded lattice and a $\theta \in {\rm Con}(M)$ has $0/\theta =\{0\}$, then, for all $x\in M$ and all $U,V\subseteq M$:\begin{enumerate}
\item\label{anntheta0} $x/\theta \in {\rm Ann}_{M/\theta }(U/\theta )$ iff $x\in {\rm Ann}_M(U)$ iff $x/\theta \subseteq {\rm Ann}_M(U)$, and $x/\theta \in {\rm Ann}_{M/\theta }(\linebreak {\rm Ann}_{M/\theta }(U/\theta ))$ iff $x\in {\rm Ann}_M({\rm Ann}_M(U))$ iff $x/\theta \subseteq {\rm Ann}_M({\rm Ann}_M(U))$;
\item\label{anntheta1} ${\rm Ann}_M(U)/\theta ={\rm Ann}_{M/\theta }(U/\theta )$ and ${\rm Ann}_M({\rm Ann}_M(U))/\theta ={\rm Ann}_{M/\theta }({\rm Ann}_{M/\theta }(U/\theta ))$;
\item\label{anntheta2} $U/\theta \subseteq {\rm Ann}_M(V)/\theta $ iff $U\subseteq {\rm Ann}_M(V)$, and ${\rm Ann}_M(U)/\theta \!=\!{\rm Ann}_M(V)/\theta $ iff\linebreak  ${\rm Ann}_M(U)={\rm Ann}_M(V)$.\end{enumerate}\label{anntheta}\end{lemma}

\begin{proof} (\ref{anntheta0}) and (\ref{anntheta1}) If $x/\theta \in {\rm Ann}_{M/\theta }(U/\theta )$, then, for all $u\in U$, we have $x\wedge u\in 0/\theta =\{0\}$, thus $x\in {\rm Ann}_M(U)$, so $x/\theta \in {\rm Ann}_M(U)/\theta $, hence ${\rm Ann}_{M/\theta }(U/\theta )\subseteq {\rm Ann}_M(U)/\theta $. We have the converse implication and inclusion from Remark \ref{inclann}, therefore ${\rm Ann}_{M/\theta }(U/\theta )={\rm Ann}_M(U)/\theta $ and: $x\in {\rm Ann}_M(U)$ iff $x/\theta \in {\rm Ann}_{M/\theta }(U/\theta )$, so that: $x\in {\rm Ann}_M({\rm Ann}_M(U))$ iff $x/\theta \in {\rm Ann}_{M/\theta }({\rm Ann}_M(U)/\theta )={\rm Ann}_{M/\theta }({\rm Ann}_{M/\theta }(U/\theta ))$, thus ${\rm Ann}_M({\rm Ann}_M(U))/\theta $\linebreak $={\rm Ann}_{M/\theta }({\rm Ann}_{M/\theta }(U/\theta ))$.

Clearly, if $x/\theta \subseteq {\rm Ann}_M(U)$, then $x\in {\rm Ann}_M(U)$, while, if $x/\theta \subseteq {\rm Ann}_M({\rm Ann}_M(U))$, then $x\in {\rm Ann}_M({\rm Ann}_M(U))$. By the above, for any $y\in x/\theta $, we have: $x\in {\rm Ann}_M(U)$ iff $x/\theta \in {\rm Ann}_{M/\theta }(U/\theta )$ iff $y/\theta \in {\rm Ann}_{M/\theta }(U/\theta )$ iff $y\in {\rm Ann}_M(U)$, and, similarly, $x\in {\rm Ann}_M({\rm Ann}_M(U))$ iff $x/\theta \in {\rm Ann}_{M/\theta }({\rm Ann}_{M/\theta }(U/\theta ))$ iff $y/\theta \in {\rm Ann}_{M/\theta }({\rm Ann}_{M/\theta }(U/\theta ))$ iff $y\in {\rm Ann}_M({\rm Ann}_M(U))$. Therefore $x\in {\rm Ann}_M(U)$ implies $x/\theta \subseteq {\rm Ann}_M(U)$, while $x\in {\rm Ann}_M({\rm Ann}_M(U))$ implies $x/\theta \subseteq {\rm Ann}_M({\rm Ann}_M(U))$.

\noindent (\ref{anntheta2}) By (\ref{anntheta1}), for all $u\in U$, we have: $u/\theta \in {\rm Ann}_M(V)/\theta ={\rm Ann}_{M/\theta }(V/\theta )$ iff $u\in {\rm Ann}_M(V)$, hence the first equivalence, therefore: ${\rm Ann}_M(U)/\theta ={\rm Ann}_M(V)/\theta $ iff ${\rm Ann}_M(U)/\theta \subseteq {\rm Ann}_M(V)/\theta $ and ${\rm Ann}_M(V)/\theta \subseteq {\rm Ann}_M(U)/\theta $ iff ${\rm Ann}_M(U)\subseteq {\rm Ann}_M(V)$ and ${\rm Ann}_M(V)\subseteq {\rm Ann}_M(U)$ iff ${\rm Ann}_M(U)={\rm Ann}_M(V)$.\end{proof}

\begin{lemma} Let $M$ be a bounded lattice, $\theta \in {\rm Con}(M)$ such that $0/\theta =\{0\}$ and $\kappa $ a cardinality. Then:\begin{enumerate}
\item\label{panntheta1} the maps $P\mapsto P/\theta $ from: ${\cal A}nn(M)$ to ${\cal A}nn(M/\theta )$, ${\cal A}nn_{\kappa }(M)$ to ${\cal A}nn_{\kappa }(M/\theta )$, ${\rm PAnn}(M)$ to ${\rm PAnn}(M/\theta )$, ${\rm 2Ann}(M)$ to ${\rm 2Ann}(M/\theta )$, ${\rm 2Ann}_{\kappa }(M)$ to ${\rm 2Ann}_{\kappa }(M/\theta )$, respectively ${\rm P2Ann}(M)$ to ${\rm P2Ann}(M/\theta )$, are order isomorphisms;
\item\label{annkmth2} ${\cal A}nn_{\kappa }(M/\theta )={\rm PAnn}(M/\theta )$ iff ${\cal A}nn_{\kappa }(M)={\rm PAnn}(M)$; ${\cal A}nn(M/\theta )={\rm PAnn}(M/\theta )$ iff ${\cal A}nn(M)={\rm PAnn}(M)$; ${\rm 2Ann}_{\kappa }(M/\theta )\subseteq {\rm PAnn}(M/\theta )$ iff ${\rm 2Ann}_{\kappa }(M)\subseteq {\rm PAnn}(M)$; ${\rm 2Ann}(M/\theta )\subseteq {\rm PAnn}(M/\theta )$ iff ${\rm 2Ann}(M)\subseteq {\rm PAnn}(M)$;
\item\label{panntheta0} for all $U\subseteq M$: ${\rm Ann}
_M(U)\in {\rm Id}(M)$ iff ${\rm Ann}_{M/\theta }(U/\theta )\in {\rm Id}(M/\theta )$, and\linebreak ${\rm Ann}_M({\rm Ann}_M(U))\in {\rm Id}(M)$ iff ${\rm Ann}_{M/\theta }({\rm Ann}_{M/\theta }(U/\theta ))\in {\rm Id}(M/\theta )$;
\item\label{panntheta2} ${\cal A}nn(M)\subseteq {\rm Id}(M)$ iff ${\cal A}nn(M/\theta )\subseteq {\rm Id}(M/\theta )$; ${\rm PAnn}(M)\subseteq {\rm Id}(M)$ iff ${\rm PAnn}(M/\theta )\subseteq {\rm Id}(M/\theta )$; ${\rm P2Ann}(M)\subseteq {\rm Id}(M)$ iff ${\rm P2Ann}(M/\theta )\subseteq {\rm Id}(M/\theta )$;
\item\label{panntheta5} for all $U,V\subseteq M$ such that ${\rm Ann}_M(U),{\rm Ann}_M(V)\in {\rm Id}(M)$, we have, in ${\rm Id}(M)$ and ${\rm Id}(M/\theta )$: ${\rm Ann}_M(U\cap V)={\rm Ann}_M(U)\vee {\rm Ann}_M(V)\in {\rm Id}(M)$ iff ${\rm Ann}_{M/\theta }(U/\theta \cap V/\theta )={\rm Ann}_{M/\theta }(U/\theta )\vee {\rm Ann}_{M/\theta }(V/\theta )\in {\rm Id}(M/\theta )$;
\item\label{panntheta8} if ${\cal A}nn(M)\subseteq {\rm Id}(M)$ and ${\rm Ann}_M(U\cap V)={\rm Ann}_M(U)\vee {\rm Ann}_M(V)$ for all $U,V\subseteq M$, then ${\cal A}nn(M)$ and ${\cal A}nn(M/\theta )$ are sublattices of ${\rm Id}(M)$ and ${\rm Id}(M/\theta )$, respectively, and the map $P\mapsto P/\theta $ from ${\cal A}nn(M)$ to ${\cal A}nn(M/\theta )$ is a lattice isomorphism;
\item\label{panntheta6} for all $a,b\in M$ such that ${\rm Ann}_M(a),{\rm Ann}_M(b)\in {\rm Id}(M)$, we have, in ${\rm Id}(M)$ and ${\rm Id}(M/\theta )$: ${\rm Ann}_M(a\vee b)={\rm Ann}_M(a)\cap {\rm Ann}_M(b)\in {\rm Id}(M)$ iff ${\rm Ann}_{M/\theta }(a/\theta \vee b/\theta )={\rm Ann}_{M/\theta }(a/\theta )\cap {\rm Ann}_{M/\theta }(b/\theta )\in {\rm Id}(M/\theta )$, and ${\rm Ann}_M(a\wedge b)={\rm Ann}_M(a)\vee {\rm Ann}_M(b)\in {\rm Id}(M)$ iff ${\rm Ann}_{M/\theta }(a/\theta \wedge b/\theta )={\rm Ann}_{M/\theta }(a/\theta )\vee {\rm Ann}_{M/\theta }(b/\theta )\in {\rm Id}(M/\theta )$;
\item\label{panntheta7} for all $a,b\in M$ such that ${\rm Ann}_M({\rm Ann}_M(a)),{\rm Ann}_M({\rm Ann}_M(b))\in {\rm Id}(M)$, we have, in ${\rm Id}(M)$ and ${\rm Id}(M/\theta )$: ${\rm Ann}_M({\rm Ann}_M(a\vee b))={\rm Ann}_M({\rm Ann}_M(a))\vee {\rm Ann}_M({\rm Ann}_M(b))\in {\rm Id}(M)$ iff ${\rm Ann}_{M/\theta }({\rm Ann}_{M/\theta }(a/\theta \vee b/\theta ))={\rm Ann}_{M/\theta }({\rm Ann}_{M/\theta }(a/\theta ))\vee {\rm Ann}_{M/\theta }({\rm Ann}_{M/\theta }\linebreak (b/\theta ))\in {\rm Id}(M/\theta )$, and ${\rm Ann}_M({\rm Ann}_M(a\wedge b))={\rm Ann}_M({\rm Ann}_M(a))\cap {\rm Ann}_M({\rm Ann}_M(b))\linebreak \in {\rm Id}(M)$ iff ${\rm Ann}_{M/\theta }({\rm Ann}_{M/\theta }(a/\theta \wedge b/\theta ))={\rm Ann}_{M/\theta }({\rm Ann}_{M/\theta }(a/\theta ))\cap {\rm Ann}_{M/\theta }(\linebreak {\rm Ann}_{M/\theta }(b/\theta ))\in {\rm Id}(M/\theta )$;
\item\label{panntheta3} ${\rm PAnn}(M)$ is a sublattice of ${\rm Id}(M)$ such that the map $x\mapsto {\rm Ann}_M(x)$ is a lattice anti--morphism from $M$ to ${\rm PAnn}(M)$ iff ${\rm PAnn}(M/\theta )$ is a sublattice of ${\rm Id}(M/\theta )$ such that the map $y\mapsto {\rm Ann}_{M/\theta }(y)$ is a lattice anti--morphism from $M/\theta $ to\linebreak ${\rm PAnn}(M/\theta )$, and, if so, then the map $P\mapsto P/\theta $ from  ${\rm PAnn}(M)$ to ${\rm PAnn}(M/\theta )$ is a lattice isomorphism;
\item\label{panntheta4} ${\rm P2Ann}(M)$ is a sublattice of ${\rm Id}(M)$ such that the map $x\mapsto {\rm Ann}_M({\rm Ann}_M(x))$ is a lattice morphism from $M$ to ${\rm P2Ann}(M)$ iff ${\rm P2Ann}(M/\theta )$ is a sublattice of ${\rm Id}(M/\theta )$ such that the map $y\mapsto {\rm Ann}_{M/\theta }({\rm Ann}_{M/\theta }(y))$ is a lattice morphism from $M/\theta $ to ${\rm P2Ann}(M/\theta )$, and, if so, then the map $P\mapsto P/\theta $ from  ${\rm P2Ann}(M)$ to ${\rm P2Ann}(M/\theta )$ is a lattice isomorphism.\end{enumerate}\label{panntheta}\end{lemma}

\begin{proof} (\ref{panntheta1}) By Lemma \ref{anntheta}, (\ref{anntheta1}), these maps are well defined and surjective; by Lemma \ref{anntheta}, (\ref{anntheta2}), they are also injective, hence they are bijective. By Lemma \ref{anntheta}, (\ref{anntheta1}), these maps, as well as their inverses, preserve inclusion. Therefore they are order isomorphisms.

\noindent (\ref{annkmth2}) By (\ref{panntheta1}) and Lemma \ref{anntheta}, (\ref{anntheta2}).

\noindent (\ref{panntheta0}) From (\ref{annkmth2}) and the clear fact that $I/\theta \in {\rm Id}(M/\theta )$ for any $I\in {\rm Id}(M)$, we get the direct implications.

Now assume that ${\rm Ann}_{M/\theta }(U/\theta )\in {\rm Id}(M/\theta )$, and let $x,y,z\in M$ such that $x,y\in {\rm Ann}_M(U)$ and $x\geq z$, so that $x/\theta ,y/\theta \in {\rm Ann}_{M/\theta }(U/\theta )$ and $x/\theta \geq z/\theta $, thus $(x\vee y)/\theta ,z/\theta \in {\rm Ann}_{M/\theta }(U/\theta )$, hence $x\vee y,z\in {\rm Ann}_M(U)$ by Lemma \ref{anntheta}, (\ref{anntheta0}), therefore ${\rm Ann}_M(U)\in {\rm Id}(M)$.

Thus ${\rm Ann}_M(U)\in {\rm Id}(M)$ iff ${\rm Ann}_{M/\theta }(U/\theta )\in {\rm Id}(M/\theta )$. By Lemma \ref{anntheta}, (\ref{anntheta1}), from this we also get that ${\rm Ann}_M({\rm Ann}_M(U))\in {\rm Id}(M)$ iff ${\rm Ann}_{M/\theta }({\rm Ann}_{M/\theta }(U/\theta ))={\rm Ann}_{M/\theta }(\linebreak {\rm Ann}_M(U)/\theta )\in {\rm Id}(M/\theta )$.

\noindent (\ref{panntheta2}) By (\ref{panntheta0}).

\noindent (\ref{panntheta5}) If ${\rm Ann}_M(U),{\rm Ann}_M(V)\in {\rm Id}(M)$, then ${\rm Ann}_{M/\theta }(U/\theta ),{\rm Ann}_{M/\theta }(V/\theta )\in {\rm Id}(M/\theta )$ by (\ref{panntheta0}), so the equivalences in the enunciation follow from Lemma \ref{anntheta}, (\ref{anntheta1}), and the fact that the map $I\mapsto I/\theta $ is a lattice morphism from ${\rm Id}(M)$ to ${\rm Id}(M/\theta )$.

\noindent (\ref{panntheta8}) By (\ref{panntheta1}), (\ref{panntheta0}) and the fact that, for all $U,V\subseteq M$, ${\rm Ann}_M(U\cup V)={\rm Ann}_M(U)\cap {\rm Ann}_M(V)$ and the same goes for $U/\theta ,V/\theta $ in $M/\theta $.

\noindent (\ref{panntheta6}) and (\ref{panntheta7}) Similar to the proof of (\ref{panntheta5}).

\noindent (\ref{panntheta3}) By (\ref{panntheta1}), (\ref{panntheta0}), (\ref{panntheta2}) and (\ref{panntheta6}).

\noindent (\ref{panntheta4}) By (\ref{panntheta1}), (\ref{panntheta0}), (\ref{panntheta2}) and (\ref{panntheta7}).\end{proof}

\begin{proposition} Let $M$ be a bounded lattice and $\theta \in {\rm Con}(M)$ such that $0/\theta =\{0\}$.\begin{enumerate}
\item\label{quodist1} If $M/\theta $ is distributive, then ${\cal A}nn(M)\subseteq {\rm Id}(M)$ and ${\rm Ann}_M(U)={\rm Ann}_M((U]_M)$, so $\displaystyle {\rm Ann}_M(\bigvee _{k\in K}I_k)=\bigcap _{k\in K}{\rm Ann}_M(I_k)$ for any $(I_k)_{k\in K}\subseteq {\rm Id}(M)$.
\item\label{quodist2} Let $\kappa $ be a nonzero cardinality. If ${\cal A}nn_{\kappa }(M/\theta )={\rm PAnn}(M/\theta )$, in particular if $M/\theta $ is closed w.r.t. the joins of families of elements of cardinality at most $\kappa $ and has the meet distributive w.r.t. the joins of families of elements of cardinality at most $\kappa $, then ${\cal A}nn_{\kappa }(M)={\rm PAnn}(M)$, so $M$ is Stone iff $M$ is $\kappa $--Stone.

If $M$ and $M/\theta $ are closed w.r.t. the joins of families of elements of cardinality at most $\kappa $, $M/\theta $ has the meet distributive w.r.t. such joins and $\theta $ preserves such joins, then ${\rm Ann}_M(U)={\rm Ann}_M(\bigvee U)$ for any $U\subseteq M$ with $|U|\leq \kappa $.
\item\label{quodist3} If ${\cal A}nn(M/\theta )={\rm PAnn}(M/\theta )$, in particular if $M/\theta $ is a frame, then ${\cal A}nn(M)={\rm PAnn}(M)\subseteq {\rm PId}(M)$, so $M$ is Stone iff $M$ is strongly Stone.

If $M$ is complete, $M/\theta $ is a frame and $\theta $ preserves arbitrary joins, then ${\rm Ann}_M(U)={\rm Ann}_M(\bigvee U)$ for any $U\subseteq M$.\end{enumerate}\label{quodist}\end{proposition}

\begin{proof} (\ref{quodist1}) By Lemma \ref{anntheta}, (\ref{anntheta1}) and (\ref{anntheta2}), ${\rm Ann}_M(U)/\theta ={\rm Ann}_{M/\theta }(U/\theta )=\linebreak {\rm Ann}_{M/\theta }((U/\theta ]_{M/\theta })={\rm Ann}_{M/\theta }((U]_M/\theta )={\rm Ann}_M((U]_M)/\theta $, thus ${\rm Ann}_M(U)=\linebreak {\rm Ann}_M((U]_M)$, hence the equality for the family of ideals of $M$. 

Also, ${\rm Ann}_M(U)/\theta ={\rm Ann}_{M/\theta }(U/\theta )\in {\rm Id}(M/\theta )$, so that ${\rm Ann}_{M/\theta }(U/\theta )=\linebreak ({\rm Ann}_{M/\theta }(U/\theta )]_{M/\theta }=({\rm Ann}_M(U)/\theta ]_{M/\theta }=({\rm Ann}_M(U)]_M/\theta $, thus $({\rm Ann}_M(U)]_M/\theta \subseteq {\rm Ann}_M(U)/\theta $, hence $({\rm Ann}_M(U)]_M\subseteq {\rm Ann}_M(U)$, therefore ${\rm Ann}_M(U)=({\rm Ann}_M(U)]_M\in {\rm Id}(M)$.

\noindent (\ref{quodist2}) By Remark \ref{mcplt} and Lemma \ref{panntheta}, (\ref{annkmth2}), ${\cal A}nn_{\kappa }(M)={\rm PAnn}(M)$.

If an $U\subseteq M$ has $|U|\leq \kappa $, then ${\rm Ann}_M(U)/\theta ={\rm Ann}_{M/\theta }(\bigvee (U/\theta ))={\rm Ann}_{M/\theta }((\bigvee U)/\theta )\linebreak ={\rm Ann}_M(\bigvee U)/\theta $, hence ${\rm Ann}_M(U)={\rm Ann}_M(\bigvee U)$ by Lemma \ref{anntheta}, (\ref{anntheta2}). 

\noindent (\ref{quodist3}) By Remark \ref{mcplt} and Lemma \ref{panntheta}, (\ref{annkmth2}), ${\cal A}nn(M)={\rm PAnn}(M)$.

Additionally, $(\bigvee {\rm Ann}_M(U))/\theta =\bigvee ({\rm Ann}_M(U)/\theta )=\bigvee {\rm Ann}_{M/\theta }(U/\theta )\in {\rm Ann}_{M/\theta }(U/\theta )\linebreak ={\rm Ann}_M(U)/\theta $ by Remark \ref{frameannpid}, thus $\bigvee {\rm Ann}_M(U)\in {\rm Ann}_M(U)$, hence the ideal ${\rm Ann}_M(U)$ of $M$ is principal.

As in the proof in (\ref{quodist2}), here we obtain that, for any $U\subseteq M$, ${\rm Ann}_M(U)={\rm Ann}_M(\bigvee U)$.\end{proof}

\begin{remark} Let $e\in M$. Then $e=\max (e/\theta )$ iff, for all $x\in M$, we have the equivalence: $x/\theta \leq e/\theta $ iff $x\leq e$.

Indeed, the latter equivalence and the fact that $e\in e/\theta $ imply that $e=\max (e/\theta )$, while, if the latter equality holds and $x/\theta \leq e/\theta $, then $(x\vee e)/\theta =e/\theta $, that is $x\vee e\in e/\theta $, so that $x\vee e\leq \max (e/\theta )=e$, thus $x\leq e$.

Hence, if $e=\max (e/\theta )$, then, for all $U\subseteq M$, we have: $U/\theta \subseteq (e]_M/\theta =(e/\theta ]_{M/\theta }$ iff $U\subseteq (e]_M$.\label{emaxcls}\end{remark}

Note that Theorem \ref{davey}, (\ref{davey1}), relies on the fact that ${\cal B}({\rm Id}(D))=\{(e]_D\ |\ e\in {\cal B}(D)\}$ for any bounded distributive lattice $D$. Let us see that we can transfer this property from $M/\theta $ to $M$.

\begin{remark} Clearly, ${\cal B}(M)/\theta \subseteq {\cal B}(M/\theta )$, thus the map $p_{\theta }\mid _{{\cal B}(M)}:{\cal B}(M)\rightarrow {\cal B}(M/\theta )$ is well defined.\label{boolquo}\end{remark}

Recall from \cite{cblp} that, by definition, $\theta $ has the {\em Boolean Lifting Property (BLP)} iff ${\cal B}(M)/\theta ={\cal B}(M/\theta )$, that is iff the map above is surjective.

\begin{remark} If $0/\theta =\{0\}$ and $1/\theta =\{1\}$, then, clearly, for any $e,f\in M$: $e$ is a complement of $f$ iff $e/\theta $ is a complement of $f/\theta $, thus $e\in {\cal B}(M)$ iff $e/\theta \in {\cal B}(M/\theta )$, hence ${\cal B}(M/\theta )={\cal B}(M)/\theta $ (that is $\theta $ has the BLP).\label{0n1sgl}\end{remark}

\begin{remark} If $e=\max (e/\theta )$ for all $e\in {\cal B}(M)$, then, by Remark \ref{emaxcls}, for all $e,f\in {\cal B}(M)$, we have: $e/\theta =f/\theta $ iff $e/\theta \leq f/\theta $ and $f/\theta \leq e/\theta $ iff $e\leq f$ and $f\leq e$ iff $e=f$, hence the map $p_{\theta }\mid _{{\cal B}(M)}:{\cal B}(M)\rightarrow {\cal B}(M/\theta )$ is injective.\label{booltheta6}\end{remark}

\begin{remark} Clearly, if ${\cal B}(M)$ is a sublattice, respectively a Boolean sublattice of $M$, then ${\cal B}(M)/\theta $ is a sublattice, respectively a Boolean sublattice of $M/\theta $.

Since $p_{\theta }:M\rightarrow M/\theta $ is a bounded lattice morphism, it follows that, if ${\cal B}(M)$ and ${\cal B}(M/\theta )$ are sublattices, thus bounded sublattices, of $M$ and $M/\theta $, respectively, then $p_{\theta }\mid _{{\cal B}(M)}:{\cal B}(M)\rightarrow {\cal B}(M/\theta )$ is a bounded lattice morphism, hence, if ${\cal B}(M)$ and ${\cal B}(M/\theta )$ are Boolean sublattices of $M$ and $M/\theta $, respectively, then $p_{\theta }\mid _{{\cal B}(M)}:{\cal B}(M)\rightarrow {\cal B}(M/\theta )$ is a Boolean morphism, which is surjective iff ${\cal B}(M)/\theta ={\cal B}(M/\theta )$ and is injective iff $0/\theta \cap {\cal B}(M)=(p_{\theta }\mid _{{\cal B}(M)})^{-1}(\{0/\theta \})=\{0\}$ iff $1/\theta \cap {\cal B}(M)=(p_{\theta }\mid _{{\cal B}(M)})^{-1}(\{1/\theta \})=\{1\}$.

Therefore, if ${\cal B}(M)$ is a Boolean sublattice of $M$ and ${\cal B}(M)/\theta ={\cal B}(M/\theta )$, then ${\cal B}(M/\theta )$ is a Boolean sublattice of $M/\theta $ and $p_{\theta }\mid _{{\cal B}(M)}:{\cal B}(M)\rightarrow {\cal B}(M/\theta )$ is a surjective Boolean morphism.\label{booltheta23}\end{remark}

\begin{remark} If the map $p_{\theta }\mid _{{\cal B}(M)}:{\cal B}(M)\rightarrow {\cal B}(M/\theta )$ is injective and ${\cal B}(M/\theta )$ is a Boolean sublattice of $M/\theta $, then ${\cal B}(M)$ is a Boolean sublattice of $M$ and $p_{\theta }\mid _{{\cal B}(M)}:{\cal B}(M)\rightarrow {\cal B}(M/\theta )$ is a injective Boolean morphism.

Indeed, if this restriction of the bounded lattice morphism $p_{\theta }:M\rightarrow M/\theta $ is injective and its codomain ${\cal B}(M/\theta )$ is a distributive sublattice of $M/\theta $, then its domain ${\cal B}(M)$ is a distributive and thus a Boolean sublattice of $M$ and hence $p_{\theta }\mid _{{\cal B}(M)}:{\cal B}(M)\rightarrow {\cal B}(M/\theta )$ is a Boolean embedding.\label{booltheta4}\end{remark}

\begin{lemma} Let $M$ be a bounded lattice, $\theta \in {\rm Con}(M)$ such that $0/\theta =\{0\}$ and $e\in M$.\begin{enumerate}
\item\label{bmaxth2justnot1} If $(e]_M\in {\cal A}nn(M)$ or $e\in {\cal B}(M)$ is the unique complement of an $f\in {\cal B}(M)$, then $e=\max (e/\theta )$.
\item\label{justnoticed2} If $\{(g]_M\ |\ g\in {\cal B}(M)\}\subseteq {\cal A}nn(M)$ or $M$ is uniquely complemented, in particular if ${\cal B}(M)$ is a Boolean sublattice of $M$, then $g=\max (g/\theta )$ for all $g\in {\cal B}(M)$.
\item\label{bmaxth13} If $e=\max (e/\theta )$, in particular if $(e]_M\in {\cal A}nn(M)$ or $e\in {\cal B}(M)$ is the unique complement of an $f\in {\cal B}(M)$, then, for all $U\subseteq M$: $(e]_M/\theta ={\rm Ann}_M(U)/\theta $ iff $(e]_M={\rm Ann}_M(U)$.
\item\label{maxclsannjustnot3} If $M/\theta $ is distributive, $e\in {\cal B}(M)$ and $f\in {\cal B}(M)$ is a complement of $e$, then: $(e]_M\in {\cal A}nn(M)$ iff $(e]_M\in {\rm PAnn}(M)$ iff $(e]_M={\rm Ann}_M(f)$ iff $e=\max (e/\theta )$.
\item\label{cpctata} Assume that ${\cal B}(M/\theta )$ is a Boolean sublattice of $M/\theta $. If $g=\max (g/\theta )$ for all $g\in {\cal B}(M)$, in particular if $\{(g]_M\ |\ g\in {\cal B}(M)\}\subseteq {\cal A}nn(M)$ or $M$ is uniquely complemented, then ${\cal B}(M)$ is a Boolean sublattice of $M$ and $p_{\theta }\mid _{{\cal B}(M)}:{\cal B}(M)\rightarrow 
{\cal B}(M/\theta )$ is a Boolean embedding.
\item\label{justnot45} If $M/\theta $ is distributive and $M$ is uniquely complemented, then ${\cal B}(M)$ is a Boolean sublattice of $M$, $p_{\theta }\mid _{{\cal B}(M)}:{\cal B}(M)\rightarrow {\cal B}(M/\theta )$ is a Boolean embedding and, for all $g\in {\cal B}(M)$, $g=\max (g/\theta )$ and $(g]_M={\rm Ann}_M(\neg \, g)$.\end{enumerate}\label{compactata}\end{lemma}

\begin{proof} (\ref{bmaxth2justnot1}) If $(e]_M={\rm Ann}_M(U)$ for some $U\subseteq M$, then, by Lemma \ref{anntheta}, (\ref{anntheta0}), $e\in e/\theta \subseteq {\rm Ann}_M(U)=(e]_M$, thus $e=\max (e/\theta )$.

Now assume that $e\in {\cal B}(M)$ is the unique complement of an $f\in {\cal B}(M)$, and assume by absurdum that there exists an $x\in e/\theta $ such that $x\nleq e$. Then, if we denote by $a=x\vee e$, it follows that $a>e$ and thus $a\vee f=1$. But $a/\theta =x/\theta \vee e/\theta =e/\theta \vee e/\theta =e/\theta $, so $(a\wedge \neg \, e)/\theta =a/\theta \wedge \neg \, e/\theta =(e\wedge \neg \, e)/\theta =0/\theta \wedge \neg \, e/\theta =0/\theta =\{0\}$, hence $a\wedge \neg \, e=0$. Since $a\neq e$, we have a contradiction to the uniqueness of the complement of $f$ in $M$. Therefore $e=\max (e/\theta )$.

\noindent (\ref{justnoticed2}) By (\ref{bmaxth2justnot1}).

\noindent (\ref{bmaxth13}) Trivially, if $(e]_M={\rm Ann}_M(U)$, then $(e]_M/\theta ={\rm Ann}_M(U)/\theta $.

If $e=\max (e/\theta )$ and $(e]_M/\theta ={\rm Ann}_M(U)/\theta $, that is $(e]_M/\theta \subseteq {\rm Ann}_M(U)/\theta $ and ${\rm Ann}_M(U)/\theta \subseteq (e]_M/\theta $, then $(e]_M\subseteq {\rm Ann}_M(U)$ by Lemma \ref{anntheta}, (\ref{anntheta2}), and ${\rm Ann}_M(U)\subseteq (e]_M$ by Remark \ref{emaxcls}, so the converse of the implication above also holds.

We get the particular cases from (\ref{bmaxth2justnot1}).

\noindent (\ref{maxclsannjustnot3}) By (\ref{bmaxth2justnot1}), if $(e]_M\in {\cal A}nn(M)$, then $e=\max (e/\theta )$.

Of course, $(e]_M={\rm Ann}_M(f)$ implies $(e]_M\in {\rm PAnn}(M)$, which in turn implies $(e]_M\in {\cal A}nn(M)$.

Now assume that $e=\max (e/\theta )$. Since $e$ is a complement of $f$ in $M$ and $M/\theta $ is distributive, it follows that $e/\theta $ is the unique complement of $f/\theta $ in $M/\theta $. We have $e\in {\rm Ann}_M(f)$, thus $(e]_M\subseteq {\rm Ann}_M(f)$. Assume by absurdum that there exists an $x\in {\rm Ann}_M(f)$ such that $x\nleq e$, and let $a=x\vee e$. Then $a>e$, thus $a\vee f=1$, so $a/\theta \vee f/\theta =1/\theta $, and $a/\theta \neq e/\theta $ since $e=\max (e/\theta )$. Since $M/\theta $ is distributive, we have: $a/\theta \wedge f/\theta =(x/\theta \vee e/\theta )\wedge f/\theta =(x/\theta \wedge f/\theta )\vee (e/\theta \wedge f/\theta )=(x\wedge f)/\theta \vee (e\wedge f)/\theta =0/\theta \vee 0/\theta =0/\theta $, which gives us a contradiction to the uniqueness of the complement of $f/\theta $. Therefore $(e]_M={\rm Ann}_M(f)$.

\noindent (\ref{cpctata}) By Remarks \ref{booltheta6} and \ref{booltheta4}, with the particular cases given by (\ref{justnoticed2}).

\noindent (\ref{justnot45}) By (\ref{justnoticed2}), (\ref{cpctata}) and (\ref{maxclsannjustnot3}).\end{proof}

\begin{proposition} Let $M$ be a bounded lattice. Then:\begin{enumerate}
\item\label{boolidtheta1} $\{(e]_M\ |\ e\in {\cal B}(M)\}\subseteq {\cal B}({\rm Id}(M))$;
\item\label{boolidtheta2} let $\kappa $ be a nonzero cardinality;

if, for all $e,f\in {\cal B}(M)$ such that $e$ is a complement of $f$ in $M$, we have $(e]_M={\rm Ann}_M(f)$, in particular if, for some $\theta \in {\rm Con}(M)$ such that $0/\theta =\{0\}$ and $M/\theta $ is distributive, we have $e=\max (e/\theta )$ for all $e\in {\cal B}(M)$, or, equivalently, $\{(e]_M\ |\ e\in {\cal B}(M)\}\subseteq {\cal A}nn(M)$, in particular if $M$ is uniquely complemented and has a $\theta \in {\rm Con}(M)$ with $0/\theta =\{0\}$ such that $M/\theta $ is distributive, then:\begin{itemize}
\item ${\cal B}({\rm Id}(M))=\{(e]_M\ |\ e\in {\cal B}(M)\}\subseteq {\rm PAnn}(M)\subseteq {\cal A}nn_{\kappa }(M)\subseteq {\cal A}nn(M)$;
\item $M$ is a Stone lattice iff ${\rm PAnn}(M)={\cal B}({\rm Id}(M))$;
\item $M$ is a $\kappa $--Stone lattice iff ${\cal A}nn_{\kappa }(M)={\cal B}({\rm Id}(M))$;
\item $M$ is a strongly Stone lattice iff ${\cal A}nn(M)={\cal B}({\rm Id}(M))$.\end{itemize}
\end{enumerate}\label{boolidtheta}\end{proposition}

\begin{proof} (\ref{boolidtheta1}) For all $e\in {\cal B}(M)$, if $f$ is a complement of $e$ in $M$, then $(e]_M\vee (f]_M=(e\vee f]_M=(1]_M=M$ and $(e]_M\cap (f]_M=(e\wedge f]_M=(0]_M=\{0\}$, so $(e]_M\in {\cal B}({\rm Id}(M))$.

\noindent (\ref{boolidtheta2}) Let $I\in {\cal B}({\rm Id}(M))$, so that $I\cap J=\{0\}$ and $I\vee J=M=(1]_M$ for some $J\in {\rm Id}(M)$, hence $e\vee f=1$ for some $e\in I$ and $f\in J$, thus $e\wedge f\in I\cap J=\{0\}$, so $e\wedge f=0$, hence $e,f\in {\cal B}(M)$ and $f$ is a complement of $e$, thus $(e]_M={\rm Ann}_M(f)$ by the hypothesis. Since $e\in I$, we have $(e]_M\subseteq I$. For all $x\in I$ and all $y\in J$, we have $x\wedge y\in I\cap J=\{0\}$, so $x\wedge y=0$, hence $I\subseteq {\rm Ann}_M(J)\subseteq {\rm Ann}_M(f)=(e]_M$. Therefore $I=(e]_M={\rm Ann}_M(f)\in {\rm PAnn}(M)$. Hence the converse of the inclusion in (\ref{boolidtheta1}), thus ${\cal B}({\rm Id}(M))=\{(g]_M\ |\ g\in {\cal B}(M)\}\subseteq {\rm PAnn}(M)\subseteq {\cal A}nn(M)$, hence the last two statements.

We get the particular cases from Lemma \ref{compactata}, (\ref{justnoticed2}) and (\ref{maxclsannjustnot3}).\end{proof}

\begin{lemma} Let $M$ be a bounded lattice and $\theta \in {\rm Con}(M)$ such that $0/\theta =\{0\}$ and $1/\theta =\{1\}$.\begin{enumerate}
\item\label{booltheta5} If ${\cal B}(M)$ is a Boolean sublattice of $M$, then ${\cal B}(M/\theta )$ is a Boolean sublattice of $M/\theta $ and $p_{\theta }\mid _{{\cal B}(M)}:{\cal B}(M)\rightarrow {\cal B}(M/\theta )$ is a Boolean isomorphism.
\item\label{booltheta7} If $e=\max (e/\theta )$ for all $e\in {\cal B}(M)$, in particular if $\{(e]_M\ |\ e\in {\cal B}(M)\}\subseteq {\cal A}nn(M)$ or $M$ is uniquely complemented, then: ${\cal B}(M)$ is a Boolean sublattice of $M$ iff ${\cal B}(M/\theta )$ is a Boolean sublattice of $M/\theta $, and, if so, then $p_{\theta }\mid _{{\cal B}(M)}:{\cal B}(M)\rightarrow {\cal B}(M/\theta )$ is a Boolean isomorphism.\end{enumerate}\label{booltheta}\end{lemma}

\begin{proof} (\ref{booltheta5}) By Remark \ref{0n1sgl}, ${\cal B}(M)/\theta ={\cal B}(M/\theta )$, thus, according to Remark \ref{booltheta23}, ${\cal B}(M/\theta )$ is a Boolean sublattice of $M/\theta $ and $p_{\theta }\mid _{{\cal B}(M)}:{\cal B}(M)\rightarrow {\cal B}(M/\theta )$ is a surjective Boolean morphism. Since $0/\theta \cap {\cal B}(M)=\{0\}\cap {\cal B}(M)=\{0\}$, Remark \ref{booltheta23} ensures us that this Boolean morphism is also injective, so it is a Boolean isomorphism.

\noindent (\ref{booltheta7}) By (\ref{booltheta5}) and Remarks \ref{booltheta4} and \ref{booltheta6}, with Lemma \ref{compactata}, (\ref{justnoticed2}), for the particular cases.\end{proof}

\begin{proposition} If $M$ is a bounded lattice and $\theta \in {\rm Con}(M)$ such that $0/\theta =\{0\}$, then, for any cardinality $\kappa $:\begin{enumerate}
\item\label{1theta1} $(1)_{\kappa ,M}$ implies $(1)_{\kappa ,M/\theta }$, thus: if $M$ is Stone, respectively $\kappa $--Stone, respectively strongly Stone, then $M/\theta $ is Stone, respectively $\kappa $--Stone, respectively strongly Stone;
\item\label{1theta2} if $1/\theta =\{1\}$, then: if $e=\max (e/\theta )$ for all $e\in {\cal B}(M)$, in particular if $(e]_M\in {\cal A}nn(M)$ for all $e\in {\cal B}(M)$, in particular if $M$ is uniquely complemented, in particular if ${\cal B}(M)$ is a Boolean sublattice of $M$, then $(1)_{\kappa ,M}$ is equivalent to $(1)_{\kappa ,M/\theta }$, thus: $M$ is Stone, respectively $\kappa $--Stone, respectively strongly Stone iff $M/\theta $ is Stone, respectively $\kappa $--Stone, respectively strongly Stone.\end{enumerate}\label{1theta}\end{proposition}

\begin{proof} (\ref{1theta1}) Let $V\subseteq M/\theta $ such that $|V|\leq \kappa $, so that $V=U/\theta $ for some $U\subseteq M$ with $|U|=|V|\leq \kappa $. If $(1)_{\kappa ,M}$ is fulfilled, then there exists an $e\in {\cal B}(M)$ such that ${\rm Ann}_M(U)=(e]_M$, so that $e/\theta \in {\cal B}(M)/\theta \subseteq {\cal B}(M/\theta )$, and $(e/\theta ]_{M/\theta }=(e]_M/\theta ={\rm Ann}_M(U)/\theta ={\rm Ann}_{M/\theta }(U/\theta )={\rm Ann}_{M/\theta }(V)$ by Lemma \ref{anntheta}, (\ref{anntheta1}), hence $(1)_{\kappa ,M/\theta }$ is fulfilled.

\noindent (\ref{1theta2}) By (\ref{1theta1}), we have the direct implication. For the converse, let $U\subseteq M$ such that $|U|\leq \kappa $, so that $|U/\theta |\leq |U|\leq \kappa $, and thus, if $(1)_{\kappa ,M/\theta }$ is fulfilled, then, for some $e\in M$ such that $e/\theta \in {\cal B}(M/\theta )$, so that $e\in {\cal B}(M)$ by Remark \ref{0n1sgl}, and thus $e=\max (e/\theta )$ by the hypothesis, we have $(e]_M/\theta =(e/\theta ]_{M/\theta }={\rm Ann}_{M/\theta }(U/\theta )={\rm Ann}_M(U)/\theta $, hence $(e]_M={\rm Ann}_M(U)$ by Lemma \ref{compactata}, (\ref{bmaxth13}). For the particular cases, see Lemma \ref{compactata}, (\ref{bmaxth2justnot1}) and (\ref{justnoticed2}).\end{proof}

\begin{proposition} If $M$ is a bounded lattice and $\theta \in {\rm Con}(M)$ such that $0/\theta =\{0\}$ and $1/\theta =\{1\}$, then, for any cardinality $\kappa $:\begin{enumerate}
\item\label{2theta1} $(2)_{\kappa ,M}$ implies $(2)_{\kappa ,M/\theta }$;
\item\label{2theta2} if $e=\max (e/\theta )$ for all $e\in {\cal B}(M)$, in particular if $(e]_M\in {\cal A}nn(M)$ for all $e\in {\cal B}(M)$, in particular if $M$ is uniquely complemented, in particular if ${\cal B}(M)$ is a Boolean sublattice of $M$, then $(2)_{\kappa ,M}$ is equivalent to $(2)_{\kappa ,M/\theta }$.\end{enumerate}\label{2theta}\end{proposition}

\begin{proof} By Proposition \ref{1theta} and Lemma \ref{booltheta}, (\ref{booltheta7}), with Lemma \ref{compactata}, (\ref{bmaxth2justnot1}) and (\ref{justnoticed2}), for the particular cases in (\ref{2theta2}).\end{proof}

\begin{proposition} If $M$ is a bounded lattice and $\theta \in {\rm Con}(M)$ such that $0/\theta =\{0\}$, then, for any cardinality $\kappa $, $(3)_{\kappa ,M}$ is equivalent to $(3)_{\kappa ,M/\theta }$.\label{3theta}\end{proposition}

\begin{proof} By Lemma \ref{anntheta}, (\ref{anntheta1}), Lemma \ref{panntheta}, (\ref{panntheta4}), and the fact that the following diagram is commutative:

\begin{center}\hspace*{-35pt}\begin{picture}(400,40)(0,0)
\put(50,35){$M$}
\put(43,5){$M/\theta $}
\put(208,35){${\rm P2Ann}(M)$}
\put(205,5){${\rm P2Ann}(M/\theta )$}
\put(57,24){$x\mapsto x/\theta $}
\put(55,33){\vector(0,-1){19}}
\put(228,26){$\forall \, x\in M$, ${\rm Ann}_M({\rm Ann}_M(x))\mapsto $}
\put(275,15){${\rm Ann}_M({\rm Ann}_M(x))/\theta =$}
\put(275,4){${\rm Ann}_{M/\theta }({\rm Ann}_{M/\theta }(x/\theta ))$}
\put(225,33){\vector(0,-1){19}}
\put(80,41){$x\mapsto {\rm Ann}_M({\rm Ann}_M(x))$}
\put(61,38){\vector(1,0){145}}
\put(65,12){$x/\theta \!\mapsto \!{\rm Ann}_{M/\theta }({\rm Ann}_{M/\theta }(x/\theta ))$}
\put(65,8){\vector(1,0){138}}\end{picture}\end{center}\vspace*{-15pt}\end{proof}

\begin{proposition} If $M$ is a bounded lattice and $\theta \in {\rm Con}(M)$ such that $0/\theta =\{0\}$, then, for any cardinality $\kappa $, $(4)_{\kappa ,M}$ is equivalent to $(4)_{\kappa ,M/\theta }$.\label{4theta}\end{proposition}

\begin{proof} By Lemma \ref{anntheta}, (\ref{anntheta1}), Lemma \ref{panntheta}, (\ref{panntheta4}), and the surjectivity of the map $x\mapsto x/\theta $ from $M$ to $M/\theta $.\end{proof}

\begin{proposition} For any bounded lattice $M$, any $\theta \in {\rm Con}(M)$ and any cardinality $\kappa $:\begin{enumerate}
\item\label{5theta1} $(5)_{\kappa ,M}$ implies $(5)_{\kappa ,M/\theta }$;
\item\label{5theta2} if $0/\theta =\{0\}$ and $1/\theta =\{1\}$, then $(5)_{\kappa ,M}$ is equivalent to $(5)_{\kappa ,M/\theta }$.\end{enumerate}\label{5theta}\end{proposition}

\begin{proof} (\ref{5theta1}) If $(5)_{\kappa ,M}$ if fulfilled and $V\subseteq M/\theta $ with $|V|\leq \kappa $, then $V=U/\theta $ for some $U\subseteq M$ with $|U|\leq \kappa $, so that $({\rm Ann}_M(U)\cup {\rm Ann}_M({\rm Ann}_M(U))]_M=M$, thus $1=a_1\vee \ldots \vee a_n\vee b_1\vee \ldots \vee b_k$ for some $n,k\in \N ^*$, $a_1,\ldots ,a_n\in {\rm Ann}_M(U)$ and $b_1,\ldots ,b_k\in {\rm Ann}_M({\rm Ann}_M(U))$. Then $a_1/\theta ,\ldots ,a_n/\theta \in {\rm Ann}_{M/\theta }(U/\theta )={\rm Ann}_{M/\theta }(V)$ and $b_1/\theta ,\ldots ,b_k/\theta \in {\rm Ann}_{M/\theta }(\linebreak {\rm Ann}_{M/\theta }(U/\theta ))={\rm Ann}_{M/\theta }({\rm Ann}_{M/\theta }(V)$, therefore $1/\theta =a_1/\theta \vee \ldots \vee a_n/\theta \vee b_1/\theta \vee \ldots \vee b_k/\theta \in ({\rm Ann}_{M/\theta }(V)\cup {\rm Ann}_{M/\theta }({\rm Ann}_{M/\theta }(V))]_{M/\theta }$, hence $({\rm Ann}_{M/\theta }(V)\cup {\rm Ann}_{M/\theta }(\linebreak {\rm Ann}_{M/\theta }(V))]_{M/\theta }=M/\theta $.

\noindent (\ref{5theta2}) Assume that $0/\theta =\{0\}$ and $1/\theta =\{1\}$. If $(5)_{\kappa ,M/\theta }$ is fulfilled and $U\subseteq M$ with $|U|\leq \kappa $, then $|U/\theta |\leq |U|\leq \kappa $, so that $({\rm Ann}_{M/\theta }(U/\theta )\cup {\rm Ann}_{M/\theta }({\rm Ann}_{M/\theta }(U/\theta ))]_{M/\theta }=M/\theta $, thus $1/\theta =a_1/\theta \vee \ldots \vee a_n/\theta \vee b_1/\theta \vee \ldots \vee b_k/\theta $ for some $n,k\in \N ^*$ and $a_1,\ldots ,a_n,b_1,\ldots ,b_k\in M$ such that $a_1/\theta ,\ldots ,a_n/\theta \in {\rm Ann}_{M/\theta }(U/\theta )$ and $b_1/\theta ,\ldots ,b_k/\theta \in {\rm Ann}_{M/\theta }({\rm Ann}_{M/\theta }(U/\theta ))$. But then $a_1,\ldots ,a_n\in {\rm Ann}_M(U)$ and $b_1,\ldots ,b_k\in {\rm Ann}_M({\rm Ann}_M(U))$ by Lemma \ref{anntheta}, (\ref{anntheta0}), and $\{1\}=1/\theta =(a_1\vee \ldots \vee a_n\vee b_1\vee \ldots \vee b_k)/\theta $, thus $1=a_1\vee \ldots \vee a_n\vee b_1\vee \ldots \vee b_k\in ({\rm Ann}_M(U)\cup {\rm Ann}_M({\rm Ann}_M(U))]_M$, hence $({\rm Ann}_M(U)\cup {\rm Ann}_M({\rm Ann}_M(U))]_M=M$. We have the converse from (\ref{5theta1}), so the equivalence holds.\end{proof}

\begin{theorem} Let $M$ be a bounded lattice, $\theta \in {\rm Con}(M)$ such that $0/\theta =\{0\}$ and $1/\theta =\{1\}$ and $m$ be a nonzero cardinality. If $e=\max (e/\theta )$ for all $e\in {\cal B}(M)$, in particular if $(e]_M\in {\cal A}nn(M)$ for all $e\in {\cal B}(M)$, in particular if $M$ is uniquely complemented, in particular if ${\cal B}(M)$ is a Boolean sublattice of $M$, then:\begin{enumerate}
\item\label{transferdavey1} if $M/\theta $ is distributive, then, for any nonzero cardinality $\kappa $, conditions $(1)_{\kappa ,M}$, $(2)_{\kappa ,M}$, $(3)_{\kappa ,M}$, $(4)_{\kappa ,M}$ and $(5)_{\kappa ,M}$ are equivalent;
\item\label{transferdavey0} if $M/\theta $ is distributive and ${\cal A}nn_m(M)={\rm PAnn}(M)$, in particular if $M/\theta $ is closed w.r.t. the joins of all families of elements of cardinality at most $m$ and has the meet distributive w.r.t. the joins of families of cardinalities at most $m$, then, for any $h,i\in \overline{1,5}$ and any nonzero cardinality $\kappa \leq m$, conditions $(h)_{\kappa ,M}$ and $(i)_{<\infty ,M}$ are equivalent;
\item\label{transferdavey2} if $M/\theta $ is distributive and ${\cal A}nn(M)={\rm PAnn}(M)$, in particular if $M/\theta $ is a frame, then, for any $h,i,j\in \overline{1,5}$ and any nonzero cardinality $\kappa $, conditions $(iv)_M$, $(h)_{\kappa ,M}$, $(i)_{<\infty ,M}$ and $(j)_M$ are equivalent.\end{enumerate}\label{transferdavey}\end{theorem}

\begin{proof} (\ref{transferdavey1}) By Theorem \ref{davey}, (\ref{davey1}), Proposition \ref{1theta}, (\ref{1theta2}), Proposition \ref{2theta}, (\ref{2theta2}), Propositions \ref{3theta} and \ref{4theta} and Proposition \ref{5theta},  (\ref{5theta2}).

\noindent (\ref{transferdavey0}) By Theorem \ref{davey}, (\ref{davey0}), Proposition \ref{1theta}, (\ref{1theta2}), Proposition \ref{2theta}, (\ref{2theta2}), Propositions \ref{3theta} and \ref{4theta} and Proposition \ref{5theta},  (\ref{5theta2}).

\noindent (\ref{transferdavey2}) By Theorem \ref{davey}, (\ref{davey2}), Proposition \ref{1theta}, (\ref{1theta2}), Proposition \ref{2theta}, (\ref{2theta2}), Propositions \ref{3theta} and \ref{4theta} and Proposition \ref{5theta},  (\ref{5theta2}).\end{proof}

Let us also note:

\begin{proposition} Let $M$ be a bounded lattice, $\theta \in {\rm Con}(M)$ such that $0/\theta =\{0\}$ and $m$ be a nonzero cardinality.\begin{itemize}
\item If $M/\theta $ is distributive, then, for any nonzero cardinality $\kappa $, condition $(3)_{\kappa ,M}$ is equivalent to $(4)_{\kappa ,M}$.
\item If $1/\theta =\{1\}$ and $M/\theta $ is distributive, then, for any nonzero cardinality $\kappa $, conditions $(3)_{\kappa ,M}$, $(4)_{\kappa ,M}$ and $(5)_{\kappa ,M}$ are equivalent.
\item If $M/\theta $ is distributive and ${\cal A}nn_m(M/\theta )={\rm PAnn}(M/\theta )$, in particular if $M/\theta $ is closed w.r.t. the joins of all families of elements of cardinality at most $m$ and has the meet distributive w.r.t. such joins, then, for any nonzero cardinalities $\kappa \leq m$ and $\mu \leq m$, condition $(3)_{\kappa ,M}$ is equivalent to $(4)_{\mu ,M}$.
\item If $1/\theta =\{1\}$, then: if $M/\theta $ is distributive and ${\cal A}nn_m(M/\theta )={\rm PAnn}(M/\theta )$, in particular if $M/\theta $ is closed w.r.t. the joins of all families of elements of cardinality at most $\kappa $ and has the meet distributive w.r.t. the joins of families of cardinalities at most $\kappa $, then, for any nonzero cardinalities $\kappa \leq m$, $\lambda \leq m$ and $\mu \leq m$, conditions $(3)_{\kappa ,M}$, $(4)_{\lambda ,M}$ and $(5)_{\mu ,M}$ are equivalent.
\item If $M/\theta $ is distributive and ${\cal A}nn(M/\theta )={\rm PAnn}(M/\theta )$, in particular if $M/\theta $ is a frame, then, for any nonzero cardinalities $\kappa $ and $\mu $, conditions $(3)_{\kappa ,M}$, $(4)_{\mu ,M}$ and $(iv)_M$ are equivalent.
\item If $1/\theta =\{1\}$, then: if $M/\theta $ is distributive and ${\cal A}nn(M/\theta )={\rm PAnn}(M/\theta )$, in particular if $M/\theta $ is a frame, then, for any nonzero cardinalities $\kappa $, $\lambda $ and $\mu $, conditions $(3)_{\kappa ,M}$, $(4)_{\lambda ,M}$, $(iv)_M$ and $(5)_{\mu ,M}$ are equivalent.
\end{itemize}\label{echiv345}\end{proposition}

\begin{proof} By Theorem \ref{davey} and Propositions \ref{3theta}, \ref{4theta} and \ref{5theta}.\end{proof}

Recall from Proposition \ref{quodist} that, if $0/\theta =\{0\}$ and $M/\theta $ is distributive, then ${\cal A}nn(M)\linebreak \subseteq {\rm Id}(M)$, while, if $0/\theta =\{0\}$ and $M/\theta $ is a frame, then ${\cal A}nn(M)={\rm PAnn}(M)\subseteq {\rm Id}(M)$.

If we eliminate the nontrivial implications from Theorem \ref{transferdavey}, (\ref{transferdavey0}) and (\ref{transferdavey2}), along with those that immediately follow from Proposition \ref{quodist}, (\ref{quodist2}) and (\ref{quodist3}), and Remark \ref{mcplt}, then we obtain the following:

\begin{corollary} Let $M$ be a bounded lattice and $\theta \in {\rm Con}(M)$ such that $0/\theta =\{0\}$ and $1/\theta =\{1\}$ and $m$ be a nonzero cardinality. If $e=\max (e/\theta )$ for all $e\in {\cal B}(M)$, in particular if $(e]_M\in {\cal A}nn(M)$ for all $e\in {\cal B}(M)$, in particular if $M$ is uniquely complemented, in particular if ${\cal B}(M)$ is a Boolean sublattice of $M$, then:\begin{enumerate}
\item\label{clearer2} if $M/\theta $ is distributive and ${\cal A}nn_m(M)={\rm PAnn}(M)$, in particular if $M/\theta $ is closed w.r.t. the joins of all families of elements of cardinality at most $m$ and has the meet distributive w.r.t. such joins, then the following are equivalent:\begin{itemize}
\item $M$ is Stone;
\item $M$ is $m$--Stone and ${\cal B}(M)$ is an $m$--complete Boolean sublattice of $M$;
\item ${\rm P2Ann}(M)$ is a Boolean sublattice of ${\rm Id}(M)$ such that $a\mapsto \linebreak {\rm Ann}_M({\rm Ann}_M(a))$ is a lattice morphism from $M$ to ${\rm P2Ann}(M)$;
\item ${\rm P2Ann}(M)$ is an $m$--complete Boolean sublattice of ${\rm Id}(M)$ such that $a\mapsto {\rm Ann}_M({\rm Ann}_M(a))$ is a lattice morphism from $M$ to ${\rm P2Ann}(M)$;
\item for all $a,b\in M$, ${\rm Ann}_M(a\wedge b)={\rm Ann}_M(a)\vee {\rm Ann}_M(b)$ and ${\rm P2Ann}(M)\subseteq {\rm PAnn}(M)$;
\item for all $a\in M$, ${\rm Ann}_M(a)\vee {\rm Ann}_M({\rm Ann}_M(a))=M$;\end{itemize}
\item\label{clearer1} if $M/\theta $ is distributive and ${\cal A}nn(M)={\rm PAnn}(M)$, in particular if $M/\theta $ is a frame, then the following are equivalent:\begin{itemize}
\item $M$ is Stone;
\item $M$ is strongly Stone and ${\cal B}(M)$ is a complete Boolean sublattice of $M$;
\item ${\rm P2Ann}(M)$ is a Boolean sublattice of ${\rm Id}(M)$ such that $a\mapsto \linebreak {\rm Ann}_M({\rm Ann}_M(a))$ is a lattice morphism from $M$ to ${\rm P2Ann}(M)$;
\item ${\rm P2Ann}(M)$ is a complete Boolean sublattice of ${\rm Id}(M)$ such that $a\mapsto \linebreak {\rm Ann}_M({\rm Ann}_M(a))$ is a lattice morphism from $M$ to ${\rm P2Ann}(M)$;
\item for all $a,b\in M$, ${\rm Ann}_M(a\wedge b)={\rm Ann}_M(a)\vee {\rm Ann}_M(b)$;
\item for all $a\in M$, ${\rm Ann}_M(a)\vee {\rm Ann}_M({\rm Ann}_M(a))=M$.\end{itemize}\end{enumerate}\label{clearer}\end{corollary}

\section{A Certain Congruence of a Commutator Lattice}
\label{thecongruence}

Throughout this section, unless mentioned otherwise, $(L,\vee ,\wedge ,[\cdot ,\cdot ],0,1)$ will be a commutator lattice. Some of the results that follow in this paper generalize properties obtained in \cite{retic} for the particular case of congruence lattices endowed with the term--condition commutator, under the condition that this commutator operation is commutative and distributive in both arguments w.r.t. arbitrary joins, in particular for congruence lattices of members of congruence--modular varieties, endowed with the modular commutator.

\begin{remark} Since $[x,y]\leq x\wedge y$ for all $x,y\in L$, we clearly have ${\rm Spec}_L\subseteq {\rm Mi}(L)\setminus \{1\}$. If $L$ has finite length, then ${\rm Mi}(L)\setminus \{1\}={\rm Smi}(L)$, so that ${\rm Spec}_L\subseteq {\rm Smi}(L)$.\end{remark}

\begin{lemma} Let $(L,[\cdot ,\cdot ])$ be a commutator lattice. Then:\begin{enumerate}
\item\label{smimaxspec1} $\{x\in {\rm Smi}(L)\ |\ [x^{+},x^{+}]\nleq x\}\subseteq {\rm Spec}_L$; if $L$ has finite length, then ${\rm Spec}_L=\{x\in {\rm Smi}(L)\ |\ [x^{+},x^{+}]\nleq x\}$.
\item\label{smimaxspec0} ${\rm Spec}_L=\{x\in {\rm Mi}(L)\setminus \{1\}\ |\ (\forall \, a\in L)\, ([a,a]\leq x\Rightarrow a\leq x)\}$;
\item\label{smimaxspec2} if $[1,1]=1$, then ${\rm Max}_L\subseteq {\rm Spec}_L$;
\item\label{smimaxspec3} if $1\in {\rm Cp}(L)$, then, for each $x\in L\setminus \{1\}$, there exists a $p\in {\rm Max}_L$ such that $x\leq p$;
\item\label{smimaxspec4} if $1\in {\rm Cp}(L)$ and $[1,1]=1$, then, for each $x\in L\setminus \{1\}$, there exists a $p\in {\rm Spec}_L$ such that $x\leq p$.\end{enumerate}\label{smimaxspec}\end{lemma}

\begin{proof} (\ref{smimaxspec1}) Take $x\in {\rm Smi}(L)\subseteq L\setminus \{1\}$ with $[x^{+},x^{+}]\nleq x$, so that $[x^{+},x^{+}]=x^{+}$ since $[x^{+},x^{+}]\leq x^{+}$, and let $a,b\in L$ such that $[a,b]\leq x$. Assume by absurdum that $a\nleq x$ and $b\nleq x$, which means that $a\vee x\neq x$ and $b\vee x\neq x$, hence $a\vee x>x$ and $b\vee x>x$, so that $a\vee x\geq x^{+}$ and $b\vee x\geq x^{+}$. Then $x<x^{+}=[x^{+},x^{+}]\leq [a\vee x,b\vee x]=[a,b]\vee [a,x]\vee [x,b]\vee [x,x]\leq x$, so we have a contradiction. Thus $x\in {\rm Spec}_L$.

By the definition of ${\rm Spec}_L$, if $x\in {\rm Smi}(L)$ is such that $[x^{+},x^{+}]\leq x$, then $x\notin {\rm Spec}_L$. 
If $L$ has finite length, then ${\rm Spec}_L\subseteq {\rm Smi}(L)$, hence the equality for this case.

\noindent (\ref{smimaxspec0}) By the proof of \cite[Proposition $1.2$]{agl}, which we reproduce here for the sake of completeness. The left--to--right inclusion is clear. Now let $x\in {\rm Mi}(L)\setminus \{1\}$ such that, for all $a\in L$, $[a,a]\leq x$ implies $a\leq x$. Let $a,b\in L$ such that $[a,b]\leq x$ and assume by absurdum that $a\nleq x$ and $b\nleq x$, so that $a\vee x>x$ and $b\vee x>x$ and thus $(a\vee x)\wedge (b\vee x)>x$ since $x$ is meet--irreducible. Then $[(a\vee x)\wedge (b\vee x),(a\vee x)\wedge (b\vee x)]\leq [a\vee x,b\vee x]=[a,b]\vee [a,x]\vee [x,b]\vee [x,x]\leq x$, hence $(a\vee x)\wedge (b\vee x)\leq x$ by the choice of $x$, and we have a contradiction. Hence $x\in {\rm Spec}_L$.

\noindent (\ref{smimaxspec2}) Clearly, each $x\in {\rm Max}_L$ is strictly meet--irreducible, with $x^{+}=1$. Now apply (\ref{smimaxspec1}).

\noindent (\ref{smimaxspec3}) Assume that $1\in {\rm Cp}(L)$ and let $x\in L\setminus \{1\}$. We prove that the set $([x)_L\setminus \{1\},\leq )$ is inductively ordered. Let $C\subseteq [x)_L\setminus \{1\}$ such that $(C,\leq )$ is a chain, and let $t=\bigvee C$. We can not have $t=1$, because then, since $1\in {\rm Cp}(L)$, there would exist an $n\in \N ^*$ and elements $c_1,\ldots ,c_n\in C$ such that $\displaystyle 1=\bigvee _{i=1}^nc_i=\max \{c_1,\ldots ,c_n\}\in \{c_1,\ldots ,c_n\}\subseteq C\subseteq L\setminus \{1\}$, which gives us a contradiction. Hence $t\in L\setminus \{1\}$. But $x\leq t$, thus $t\in [x)_L\setminus \{1\}$, so indeed $([x)_L\setminus \{1\},\leq )$ is inductively ordered, therefore it has maximal elements by Zorn`s Lemma, and clearly its maximal elements are also maximal elements of $L\setminus \{1\}$, that is they belong to ${\rm Max}_L$, and they are greater than $x$.

\noindent (\ref{smimaxspec4}) By (\ref{smimaxspec2}) and (\ref{smimaxspec3}).\end{proof}

\begin{remark} By Lemma \ref{smimaxspec}, (\ref{smimaxspec0}), if $[\cdot ,\cdot ]=\wedge $, then ${\rm Spec}_L={\rm Mi}(L)\setminus \{1\}$.\label{commmeet}\end{remark}

\begin{remark} Let $x,y\in L$, $M=\{a\in L\ |\ [a,x]\leq y\}$ and $N=\{b\in L\ |\ [b,x]=y\}$.

Then $\displaystyle \bigvee \emptyset =0\in M$, in particular $M$ is nonempty, and, for any nonempty family $(a_i)_{i\in I}\subseteq M$, we have $[a_i,x]\leq y$ for all $i\in I$ and thus $\displaystyle [\bigvee _{i\in I}a_i,x]=\bigvee _{i\in I}[a_i,x]\leq y$, hence $\displaystyle \bigvee _{i\in I}a_i\in M$. Therefore $\displaystyle \bigvee _{i\in I}a_i\in M$ for any family $(a_i)_{i\in I}\subseteq M$, hence the set $M$ has a maximum, namely $\max (M)=\bigvee M$.

If $N$ is nonempty, then, for any nonempty family $(b_j)_{j\in J}\subseteq N$, we have $[b_j,x]=y$ for all $j\in J$ and thus $\displaystyle [\bigvee _{j\in J}b_j,x]=\bigvee _{j\in J}[b_j,x]=y$, hence $\displaystyle \bigvee _{j\in I}b_j\in N$, and thus $N$ has a maximum, namely $\max (N)=\bigvee N$.\label{maxcomm}\end{remark}

\begin{remark} The radical elements of $L$ are exactly the meets of the families of prime elements of $L$, hence ${\rm Spec}_L\subseteq R(L)$ and the map $x\mapsto \rho (x)$ is a closure operator on $L$ with associated closure system $R(L)=\{\rho (x)\ |\ x\in L\}=\{x\in L\ |\ \rho (x)=x\}$, so that $1\in R(L)$ since $\rho (1)=\bigwedge \emptyset =1$ and the following hold for all $a,b\in L$, $p\in {\rm Spec}_L$ and $r\in R(L)$:\begin{itemize}
\item $a\leq \rho (a)$;
\item $a\leq \rho (b)$ iff $\rho (a)\leq \rho (b)$, so that $a\leq r$ iff $\rho (a)\leq r$; in particular, $a\leq p$ iff $\rho (a)\leq p$, thus $V(a)=V(\rho (a))$;
\item $a\leq b$ implies $V(b)\subseteq V(a)$, which implies $\rho (a)\leq \rho (b)$, which in turn implies $V(b)=V(\rho (b))\subseteq V(\rho (a))=V(a)$, hence: $\rho (a)\leq \rho (b)$ iff $V(b)\subseteq V(a)$, and thus: $\rho (a)=\rho (b)$ iff $V(a)=V(b)$.\end{itemize}\label{vrho}\end{remark}

Throughout the rest of this section, unless mentiones otherwise, $\equiv $ will be the following equivalence on the set $L$: $\equiv \ =\{(a,b)\in L^2\ |\ \rho (a)=\rho (b)\}$. By Remark \ref{vrho}, $\equiv \ =\{(a,b)\in L^2\ |\ V(a)=V(b)\}$.

\begin{remark} $1\in 1/\!\!\equiv $, thus the equality $1/\!\!\equiv \ =\{1\}$ means that, for all $a\in L$: $a=1$ iff $\rho (a)=\rho (1)=1$ iff $V(a)=V(1)=\emptyset $, which in turn is equivalent to $\{a\in L\ |\ \rho (a)=1\}=\{1\}$ and to $V(a)\neq \emptyset $ for each $a\in L\setminus \{1\}$.\label{cls1cp}\end{remark}

\begin{lemma} If $(L,[\cdot ,\cdot ])$ is a commutator lattice such that $1\in {\rm Cp}(L)$ and $[1,1]=1$ and $\equiv \ =\{(a,b)\in L^2\ |\ \rho (a)=\rho (b)\}$, then $1/\!\!\equiv \ =\{1\}$.\label{cls1cpct}\end{lemma}

\begin{proof} By Lemma \ref{smimaxspec}, (\ref{smimaxspec4}), and Remark \ref{cls1cp}.\end{proof}

\begin{remark} We use Remark \ref{vrho} in what follows. Let $a,b\in L$ and $(a_i)_{i\in I}\subseteq L$. Then $V(a)=V(\rho (a))$, $V(a\wedge b)=V([a,b])=V(a)\cup V(b)=V(\rho (a)\wedge \rho (b))=V([\rho (a),\rho (b)])=V(\rho (a))\cup V(\rho (b))$ and $\displaystyle V(\bigvee _{i\in I}a_i)=\bigcap _{i\in I}V(a_i)=V(\bigvee _{i\in I}\rho (a_i))$, so that $\rho (a\wedge b)=\rho ([a,b])=\rho (a)\wedge \rho (b)=\rho (\rho (a)\wedge \rho (b))=\rho ([\rho (a),\rho (b)])$ and $\displaystyle \rho (\bigvee _{i\in I}a_i)=\rho (\bigvee _{i\in I}\rho (a_i))$, otherwise written: $a\wedge b\equiv [a,b]\equiv \rho (a)\wedge \rho (b)\equiv [\rho (a),\rho (b)]$ and $\displaystyle \bigvee _{i\in I}a_i\equiv \bigvee _{i\in I}\rho (a_i)$.

Indeed, $\rho (a)=\rho (\rho (a))$, hence the first equality. Since $[a,b]\leq a\wedge b\leq a,b$, we have $V(a)\cup V(b)\subseteq V(a\wedge b)\subseteq V([a,b])$. By the definition of prime elements, $V([a,b])\subseteq V(a)\cup V(b)$. Thus, by also using the first equality, we get the second set of equalities. Finally, $\displaystyle V(\bigvee _{i\in I}a_i)=[\bigvee _{i\in I}a_i)\cap {\rm Spec}_L=\bigcap _{i\in I}[a_i)\cap {\rm Spec}_L=\bigcap _{i\in I}([a_i)\cap {\rm Spec}_L)=\bigcap _{i\in I}V(a_i)=\bigcap _{i\in I}V(\rho (a_i))=V(\bigvee _{i\in I}\rho (a_i))$.\label{onthecg}\end{remark}

For any $x\in L$ and any $n\in \N ^*$, we shall denote: $x^1=x$ and $x^{n+1}=[x,x^n]$.

\begin{proposition} Let $(L,[\cdot ,\cdot ])$ be a commutator lattice and $\equiv \ =\{(a,b)\in L^2\ |\ \rho (a)=\rho (b)\}$. Then:\begin{enumerate}
\item\label{thecong1} $\equiv $ is a lattice congruence of $L$ which preserves arbitrary joins and the commutator operation $[\cdot ,\cdot ]$ and satisfies $[a,b]\equiv a\wedge b$ for all $a,b\in L$, $R(L)=\{\max (x/\!\!\equiv )\ |\ x\in L\}=\{x\in L\ |\ x=\max (x/\!\!\equiv )\}$, $0/\!\!\equiv \ =(\rho (0)]_L$, and, for all $x\in L$, $\rho (x)=\max (x/\!\!\equiv )=\max (\rho (x)/\!\!\equiv )=\min ([x)_L\cap R(L))$;
\item\label{thecong2} $\equiv \ =Cg _L(\{(x,\rho (x))\ |\ x\in L\})\supseteq Cg_L(\{(x\wedge y,[x,y])\ |\ x,y\in L\})\supseteq Cg_L(\{(x,[x,x])\ |\linebreak x\in L\})$;
\item\label{thecong3} for all $x\in L$ such that $x/\!\!\equiv $ has a minimum and all $a\in [\min (x/\!\!\equiv ))_L\supseteq x/\!\!\equiv $, $[a,\min (x/\!\!\equiv )]=\min (x/\!\!\equiv )$;
\item\label{thecong4} if, for each $x\in L$, there exists an $n_x\in \N ^*$ such that $\rho (x)^{n_{\scriptstyle x}}=\min (x/\!\!\equiv )$, then: $\equiv \ =Cg _L(\{(x,\rho (x))\ |\ x\in L\})=Cg_L(\{(x\wedge y,[x,y])\ |\ x,y\in L\})$ and, for all $a\in x/\!\!\equiv $ and all $n\in \N $ with $n\geq n_{\scriptstyle x}$, $a^n=\min (x/\!\!\equiv )$.\end{enumerate}\label{thecong}\end{proposition}

\begin{proof} We denote by $\sigma =Cg _L(\{(x,\rho (x))\ |\ x\in L\})$, $\sim =Cg_L(\{(x\wedge y,[x,y])\ |\ x,y\in L\})$ and $\approx =Cg_L(\{(x,[x,x])\ |\ x\in L\})$. We will repeatedly use Remark \ref{onthecg}.

\noindent (\ref{thecong1}) For all $a,b,x,y\in L$ such that $a\equiv x$ and $b\equiv y$, so that $\rho (a)=\rho (x)$ and $\rho (b)=\rho (y)$, we have $[a,b]\equiv a\wedge b\equiv \rho (a)\wedge \rho (b)=\rho (x)\wedge \rho (y)\equiv x\wedge y\equiv [x,y]$. For all $(a_i)_{i\in I},(b_i)_{i\in I}\subseteq L$ such that, for all $i\in I$, $a_i\equiv b_i$, so that $\rho (a_i)=\rho (b_i)$, we have $\displaystyle \bigvee _{i\in I}a_i\equiv \bigvee _{i\in I}\rho (a_i)=\bigvee _{i\in I}\rho (b_i)\equiv \bigvee _{i\in I}b_i$.

For each $x\in L$, we have $\rho (\rho (x))=\rho (x)$, so that $\rho (x)\in x/\!\! \equiv $, and each $y\in x/\!\! \equiv $ fulfills $y\leq \rho (y)=\rho (x)$, thus $\rho (x)=\max (x/\!\! \equiv )=\max (\rho (x)/\!\! \equiv )$ since $x/\!\! \equiv \ =\rho (x)/\!\! \equiv $. Therefore $\max (0/\!\! \equiv )=\rho (0)$, thus $0/\!\! \equiv \ =(\rho (0)]_L$, since $0/\!\! \equiv \ \in {\rm Id}(L)$. It also follows that $R(L)=\{\rho (x)\ |\ x\in L\}=\{\max (x/\!\! \equiv )\ |\ x\in L\}$ and, also, $R(L)=\{x\in L\ |\ x=\rho (x)\}=\{x\in L\ |\ x=\max (x/\!\! \equiv )\}$. Thus, for any $x\in L$, $\rho (x)\in [x)_L\cap R(L)$ and, for any $r\in [x)_L\cap R(L)$, we have $r\geq x$, thus $r=\rho (r)\geq \rho (x)$, hence $\rho (x)=\min ([x)_L\cap R(L))$.

\noindent (\ref{thecong2}) By (\ref{thecong1}), $\sim \subseteq \equiv $. Obviously, $\sim \supseteq Cg_L(\{(x\wedge x,[x,x])\ |\ x\in L\})=Cg_L(\{(x,$\linebreak $[x,x])\ |\ x\in L\})$. Each $x\in L$ fulfills $x\equiv \rho (x)$, hence $\sigma \subseteq \equiv $. If $x,y\in L$ fulfill $x\equiv y$, then $x\sigma \rho (x)=\rho (y)\sigma y$, hence $x\sigma y$, therefore $\sigma \subseteq \equiv $. Hence $\equiv \ =\sigma $.

\noindent (\ref{thecong3}) Let $x\in L$ such that $x/\!\!\equiv $ has a minimum and $a\in [\min (x/\!\!\equiv ))_L$. Then $\min (x/\!\!\equiv )\leq a$, $(x,\min (x/\!\!\equiv ))\in \equiv $ and $(\min (x/\!\!\equiv ),[\min (x/\!\!\equiv ),\min (x/\!\!\equiv )])\in \approx \subseteq \equiv $ by (\ref{thecong2}), so that $\min (x/\!\!\equiv )=\min (\min (x/\!\!\equiv )/\!\!\equiv )\leq [\min (x/\!\!\equiv ),\min (x/\!\!\equiv )]\leq [a,\min (x/\!\!\equiv )]\leq \min (x/\!\!\equiv )$, hence $[\min (x/\!\!\equiv ),\min (x/\!\!\equiv )]=[a,\min (x/\!\!\equiv )]=\min (x/\!\!\equiv )$.

\noindent (\ref{thecong4}) Let $a\in L$, arbitrary, and $x\in L$ such that $\min (x/\!\!\equiv )$ exists, so that $x/\!\!\equiv \ =\lfloor \min (x/\!\!\equiv ),\rho (x)\rceil _L$ by (\ref{thecong1}).

For all $n\in \N ^*$, $(a,a^2)\in \, \approx \, \subseteq \, \sim $ and, if $(a,a^n)\in \, \sim $, then $a=a\wedge a\sim a\wedge a^n\sim [a,a^n]=a^{n+1}$. Thus, by (\ref{thecong2}), $(a^n,a)\in \, \sim \, \subseteq \, \equiv $ for all $n\in \N ^*$, hence $(\min (x/\!\!\equiv ),\max (x/\!\!\equiv ))=(\rho (x)^{n_{\scriptstyle x}},\rho (x))\in \, \sim $, therefore $\equiv \, \subseteq \, \sim $, so $\equiv \, =\, \sim $.

If $a\in x/\!\!\equiv $ and $n\in \N $ with $n\geq n_x$, then, by the above, $a^n\, \equiv \, a\, \equiv \, x$, hence, by (\ref{thecong2}) and (\ref{thecong3}), $\min (x/\!\!\equiv )\leq a\leq \rho (x)$, thus $\min (x/\!\!\equiv )\leq a^n\leq a^{n_{\scriptstyle x}}\leq \rho (x)^{n_{\scriptstyle x}}=\min (x/\!\!\equiv )$, hence $a^n=\min (x/\!\!\equiv )$.\end{proof}

\begin{remark} Of course, for all $a\in L$, since $a\leq \rho (a)$, $\rho (a)=0$ implies $a=0$. By Proposition \ref{thecong}, (\ref{thecong1}), $0\in R(L)$ iff $\rho (0)=0$ iff $0/\!\!\equiv \ =\{0\}$ iff $(\forall \, a\in L)\, (\rho (a)=\rho (0)\Leftrightarrow a=0)$ iff $(\forall \, a\in L)\, (\rho (a)=0\Leftrightarrow a=0)$.\label{0rho}\end{remark}

\begin{remark} Note from Proposition \ref{thecong}, (\ref{thecong1}), that each class of $\equiv $ contains exactly one element of $R(L)$, namely the maximum of that class.

$\equiv \cap \ R(L)^2=\Delta _{R(L)}$, because, for all $x,y\in L$, $\rho (x)\equiv \rho (y)$ iff $\rho (x)=\rho (\rho (x))=\rho (\rho (y))=\rho (y)$. Hence, if $R(L)=L$, then $\equiv \ =\Delta _L$. Moreover: $R(L)=L$ iff $x=\rho (x)$ for all $x\in L$ iff $Cg_L(\{(x,\rho (x))\ |\ x\in L\})=\Delta _L$ iff $\equiv \ =\Delta _L$. Obviously, if $\equiv \ =\Delta _L$, then we also have $Cg_L(\{(x\wedge y,[x,y])\ |\ x,y\in L\})=Cg_L(\{(x,[x,x])\ |\ x\in L\})=\Delta _L$, thus $[\cdot ,\cdot ]=\wedge $.

If $L$ is algebraic and ${\rm Smi}(L)\subseteq {\rm Spec}_L$, in particular if $L$ is algebraic and $[\cdot ,\cdot ]=\wedge $ (see Remark \ref{commmeet}) or $L$ has finite length and all $x\in {\rm Smi}(L)={\rm Mi}(L)$ fulfill $[x^{+},x^{+}]\nleq x$ (see Lemma \ref{smimaxspec}, (\ref{smimaxspec1}))), in particular if $L$ has finite length and $[\cdot ,\cdot ]=\wedge $, then $R(L)=L$, thus $\equiv \ =\Delta _L$.

Now recall that: $[x,x]=x$ for all $x\in L$ iff $[x,y]=x\wedge y$ for all $x,y\in L$. Indeed, the converse implication is trivial, while, if $[x,x]=x$ for all $x\in L$, then we have, for all $x,y\in L$: $x\wedge y\geq [x,y]\geq [x\wedge y,x\wedge y]=x\wedge y$, therefore $[x,y]=x\wedge y$. Hence: $Cg_L(\{(x,[x,x])\ |\ x\in L\})=\Delta _L$ iff $Cg_L(\{(x\wedge y,[x,y])\ |\ x,y\in L\})=\Delta _L$.\label{cgrad}\end{remark}

\begin{remark} Obviously, all congruences of a lattice of finite length are complete. Note, also, that a distributive lattice of finite length is a frame.

Let $M$ be a lattice and $\theta \in {\rm Con}(M)$. Then, for any $x,y\in M$ such that $x/\theta $, $y/\theta $, $(x\vee y)/\theta $ and $(x\wedge y)/\theta $ have minima and maxima: $\min ((x\vee y)/\theta )=\min (x/\theta )\vee \min (y/\theta )$, $\max ((x\wedge y)/\theta )=\max (x/\theta )\wedge \max (y/\theta )$ and: $x/\theta \leq y/\theta $ iff $\min (x/\theta )\leq \min (y/\theta )$ iff $\max (x/\theta )\leq \max (y/\theta )$. Indeed, the argument below for the family $(x_i)_{i\in I}\subseteq M$ holds, without further hypotheses, for the finite family $\{x,y\}$, hence the first two equalities. Now, if $x/\theta \leq y/\theta $, that is $(x\vee y)/\theta =x/\theta \vee y/\theta =y/\theta $, then $\min (y/\theta )=\min ((x\vee y)/\theta )=\min (x/\theta )\vee \min (y/\theta )$, hence $\min (x/\theta )\leq \min (y/\theta )$. Conversely, since $(x,\min (x/\theta )),(y,\min (y/\theta ))\in \theta $, if $\min (x/\theta )\leq \min (y/\theta )$, then $x/\theta =\min (x/\theta )/\theta \leq \min (y/\theta )/\theta =y/\theta $. Therefore: $x/\theta \leq y/\theta $ iff $\min (x/\theta )\leq \min (y/\theta )$. Similarly, $x/\theta \leq y/\theta $ iff $\max (x/\theta )\leq \max (y/\theta )$.

If $M$ is complete and $\theta $ preserves arbitrary meets, then each class of $\theta $ has a minimum, because, for all $x\in M$, $\displaystyle \bigwedge (x/\theta )\in x/\theta $, thus $\displaystyle \bigwedge (x/\theta )=\min (x/\theta )$. Dually for joins and maxima of classes. Hence, if $M$ is a complete lattice and $\theta $ is a complete congruence, then all classes of $\theta $ have minima and maxima, so that $x/\theta =\lfloor \min (x/\theta ),\max (x/\theta )\rceil $ for all $x\in M$. 

Note, also, that, if $M$ is complete and $\theta $ preserves arbitrary joins, then, for any family $(x_i)_{i\in I}\subseteq M$, there exists in $M/\theta $ $\displaystyle \bigvee _{i\in I}x_i/\theta =(\bigvee _{i\in I}x_i)/\theta $, thus $M/\theta $ is a complete lattice. Similarly if $\theta $ preserves arbitrary meets.

If $M$ is a complete lattice and $\theta $ is a complete congruence, then, for any nonempty family $(x_i)_{i\in I}\subseteq M$, we have, in the complete lattice $M/\theta $: $\displaystyle \min (\bigvee _{i\in I}x_i/\theta )=\bigvee _{i\in I}\min (x_i/\theta )$ and $\displaystyle \max (\bigwedge _{i\in I}x_i/\theta )=\bigwedge _{i\in I}\max (x_i/\theta )$. Indeed, if we denote by $a_i=\min (x_i/\theta )$ for all $i\in I$ and by $\displaystyle a=\min ((\bigvee _{i\in I}x_i)/\theta )=\min (\bigvee _{i\in I}x_i/\theta )$, then, since $a_i\in x_i/\theta $ for all $i\in I$, we have $\displaystyle \bigvee _{i\in I}a_i\in (\bigvee _{i\in I}x_i)/\theta =\bigvee _{i\in I}x_i/\theta $, hence $\displaystyle \bigvee _{i\in I}a_i\geq a$. For all $k\in I$, $\displaystyle a\wedge a_k\in (\bigvee _{i\in I}x_i)/\theta \wedge x_k/\theta =((\bigvee _{i\in I}x_i)\wedge x_k)/\theta =((x_k\vee \bigvee _{i\in I\setminus \{k\}}x_i)\wedge x_k)/\theta =x_k/\theta $, thus $a\geq a\wedge a_k\geq a_k$, hence $\displaystyle a\geq \bigvee _{i\in I}a_i$. Therefore $\displaystyle a=\bigvee _{i\in I}a_i$, that is $\displaystyle \min (\bigvee _{i\in I}x_i/\theta )=\bigvee _{i\in I}\min (x_i/\theta )$. By duality, it follows that we also have $\displaystyle \max (\bigwedge _{i\in I}x_i/\theta )=\bigwedge _{i\in I}\max (x_i/\theta )$.\label{minmaxcls}\end{remark}

For any lattice $M$ and any $\theta \in {\rm Con}(M)$ such that all classes of $\theta $ have minima, let us denote by $[\cdot ,\cdot ]_{\theta }$ the binary operation on ${\rm Con}(M)$ defined by: $[x,y]_{\theta }=\min ((x\wedge y)/\theta )$ for all $x,y\in M$.

\begin{proposition} If $M$ is a complete lattice and $\theta \in {\rm Con}(M)$ is a complete congruence such that $M/\theta $ is a frame, then:\begin{itemize}
\item $(M,[\cdot ,\cdot ]_{\theta })$ is a commutator lattice;
\item $\theta =Cg_M(\{(x\wedge y,[x,y]_{\theta })\ |\ x,y\in M\})=Cg_M(\{(x,[x,x]_{\theta })\ |\ x\in M\})=\linebreak Cg_M(\{(\min (x/\theta ),\max (x/\theta ))\ |\ x\in M\})=Cg_M(\{(x,\max (x/\theta ))\ |\ x\in M\})$;
\item w.r.t. the commutator operation $[\cdot ,\cdot ]_{\theta }$, ${\rm Spec}_M={\rm Mi}(M)\cap \{\max(x/\theta )\ |\ x\in M\setminus 1/\theta \}$ and $R(M)\subseteq \{\max(x/\theta )\ |\ x\in M\}$;
\item if $M$ has finite length, then $R(M)=\{\max(x/\theta )\ |\ x\in M\}$, $\rho (y)=\max(y/\theta )$ for all $y\in M$, and $\theta =Cg_M(\{(x,\rho (x))\ |\ x\in M\})$.\end{itemize}\label{commofcg}\end{proposition}

\begin{proof} We are using Remark \ref{minmaxcls}.

For each $x\in M$, $x/\theta $ has a minimum and a maximum, so that $x/\theta =\lfloor \min (x/\theta ),\linebreak \max (x/\theta )\rceil _M$. For all $x,y\in M$, $[x,y]_{\theta }=\min ((x\wedge y)/\theta )=[y,x]_{\theta }\leq x\wedge y$ and $[x,y]_{\theta }\in (x\wedge y)/\theta $, thus $[x,y]_{\theta }/\theta =(x\wedge y)/\theta $. Trivially, $[\cdot ,\cdot ]_{\theta }$ distributes over $\displaystyle \bigvee \emptyset =0$. Now let us consider a nonempty family $(x_i)_{i\in I}\subseteq M$ and and $a\in M$. Since $M/\theta $ is a frame, we get: $\displaystyle [\bigvee _{i\in I}x_i,a]_{\theta }=\min (((\bigvee _{i\in I}x_i)\wedge a)/\theta )=\min ((\bigvee _{i\in I}(x_i\wedge a))/\theta )=\bigvee _{i\in I}\min((x_i\wedge a)/\theta )=\bigvee _{i\in I}[x_i,a]_{\theta }$. If $x,y,u,v\in M$ are such that $x\leq u$ and $y\leq v$, then $x\wedge y\leq u\wedge v$, thus $[x,y]_{\theta }=\min ((x\wedge y)/\theta )\leq \min ((u\wedge v)/\theta )=[u,v]_{\theta }$. Therefore $(M,[\cdot ,\cdot ]_{\theta })$ is a commutator lattice.

Let us denote by $\approx =Cg_M(\{(x,[x,x]_{\theta })\ |\ x\in M\})$ and by $\sim =Cg_M(\{(x\wedge y,[x,y]_{\theta })\ |\linebreak x,y\in M\})$. Then $\approx =Cg_M(\{(x\wedge x,[x,x]_{\theta })\ |\ x\in M\})\subseteq \sim \subseteq \theta $ since $(x\wedge y,[x,y]_{\theta })\in \theta $ for all $x,y\in M$. But, for all $x\in L$, $[\max (x/\theta ),\max (x/\theta )]_{\theta }=\min ((\max (x/\theta )\wedge \max (x/\theta ))/\theta )=\min (\max (x/\theta )/\theta )=\min (x/\theta )$ since $(x,\max (x/\theta ))\in \theta $. Therefore $(\max (x/\theta ),\min (x/\theta ))=(\max (x/\theta ),[\max (x/\theta ),\max (x/\theta )]_{\theta })\in \approx $. Hence, for all $y,z\in M$ such that $(y,z)\in \theta $, we have $\min (y/\theta )=\min (z/\theta )\leq y,z\leq \max (y/\theta )=\max (z/\theta )$, and $(\min (y/\theta ),\max (y/\theta ))\in \approx $, hence $(y,z)\in \approx $. Thus $\theta \subseteq \approx $. Therefore $\theta =\sim =\approx =Cg_M(\{(\min (x/\theta ),\max (x/\theta ))\ |\ x\in M\})=Cg_M(\{(x,\max (x/\theta ))\ |\ x\in M\})$, since, as above, all inclusions hold.

By Lemma \ref{smimaxspec}, (\ref{smimaxspec0}), ${\rm Spec}_M=\{x\in {\rm Mi}(M)\setminus \{1\}\ |\ (\forall \, a\in M)\, ([a,a]_{\theta }\leq x\Rightarrow a\leq x)\}$. Thus, if $x\in {\rm Spec}_M$ and $a=\max (x/\theta )$, then $[a,a]_{\theta }=\min (x/\theta )\leq x$, thus $\max (x/\theta )=a\leq x\leq \max (x/\theta )$, hence $x=\max (x/\theta )$. For all $x\in M$ such that $x=\max (x/\theta )$ and all $a\in M$, if $\min (a/\theta )=[a,a]_{\theta }\leq x$, then $a/\theta =\min (a/\theta )/\theta \leq x/\theta $, hence $a\leq \max (a/\theta )\leq \max (x/\theta )=x$, thus, if $x\in {\rm Mi}(M)\setminus \{1\}$, then $x\in {\rm Spec}_M$. Therefore ${\rm Spec}_M=\{x\in {\rm Mi}(M)\setminus \{1\}\ |\ x=\max (x/\theta )\}$, otherwise written ${\rm Spec}_M={\rm Mi}(M)\cap \{\max(x/\theta )\ |\ x\in M\setminus 1/\theta \}$, since, clearly, $\{x\in M\ |\ x=\max (x/\theta )\}=\{\max(x/\theta )\ |\ x\in M\}$ and $1=\max(1/\theta )$.

The set $\{\max (x/\theta )\ |\ x\in M\}$ is closed w.r.t. arbitrary meets, in particular all meets of prime elements of $M$ belong to $\{\max (x/\theta )\ |\ x\in M\}$, that is $R(M)\subseteq \{\max (x/\theta )\ |\ x\in M\}$. $R(M)\supseteq \{1\}\cup {\rm Spec}_M=\{1\}\cup ({\rm Mi}(M)\cap \{\max(x/\theta )\ |\ x\in M\setminus 1/\theta \})$ by the above. Let $x\in M$ such that $x=\max (x/\theta )$, but $x\neq 1$ and $x\notin {\rm Mi}(M)$, so that $x\notin 1/\theta $ and $x=a\wedge b$ for some $a,b\in M$ with $x<a$ and $x<b$. Then $x=\max (x/\theta )=\max ((a\wedge b)/\theta )=\max (a/\theta )\wedge \max (b/\theta )$, so $x$ is a meet of two elements of $\{\max (y/\theta )\ |\ y\in M\}$, both of which are strictly greater than $x$, so that none of them equals $1$, thus none of them belongs to $1/\theta $. If the lattice $M$ has finite length, it follows that $x$ is a finite meet of elements of ${\rm Mi}(M)\cap \{\max (y/\theta )\ |\ y\in M\setminus 1/\theta \}={\rm Spec}_M$, thus $x\in R(M)$. Therefore, if $M$ has finite length, then $R(M)=\{\max (x/\theta )\ |\ x\in M\}$, hence $\rho (x)=\max (x/\theta )$ for all $x\in M$, and thus $\theta =Cg_M(\{(x,\max (x/\theta ))\ |\ x\in M\})=Cg_M(\{(x,\rho (x))\ |\ x\in M\})$.\end{proof}

\begin{proposition} Let $M$ be a lattice, $\langle \cdot ,\cdot \rangle $ be a binary operation on $M$, $\gamma =Cg_M(\{(x\wedge y,\langle x,y\rangle )\ |\ x,y\in M\})$ and $\theta ,\zeta \in {\rm Con}(M)$ such that all classes of $\gamma $, $\theta $ and $\zeta $ have minima. Then:\begin{enumerate}
\item\label{commsofcgs0} for all $x\in M$, $\min (x/(\theta \cap \zeta ))=\min (x/\theta )\vee \min (x/\zeta )$;
\item\label{commsofcgs1} $[\cdot ,\cdot ]_{\theta \cap \zeta }=[\cdot ,\cdot ]_{\theta }\vee [\cdot ,\cdot ]_{\zeta }$;
\item\label{commsofcgs2} if $\zeta \subseteq \theta $, then $[\cdot ,\cdot ]_{\theta }\leq [\cdot ,\cdot ]_{\zeta }$;
\item\label{commsofcgs3} $[\cdot ,\cdot ]_{\gamma }\leq \langle \cdot ,\cdot \rangle $.\end{enumerate}\label{commsofcgs}\end{proposition}

\begin{proof} (\ref{commsofcgs0}) Let $a=\min (x/\theta )$, $b=\min (x/\zeta )$ and $c=\min (x/(\theta \cap \zeta ))=\min (x/\theta \cap x/\zeta )$ since $x/(\theta \cap \zeta )=x/\theta \cap x/\zeta $. Then $a\vee b\in x/\theta \cap x/\zeta $, hence $a\vee b\geq c$. But $c\wedge a\in x/\theta $, hence $c\geq c\wedge a\geq a$ and, analogously, $c\geq b$, thus $c\geq a\vee b$. Therefore $c=a\vee b$.

\noindent (\ref{commsofcgs1}) By (\ref{commsofcgs0}).

\noindent (\ref{commsofcgs2}) By (\ref{commsofcgs1}), or simply noticing that $\zeta \subseteq \theta $ means that, for all $x\in M$, $x/\zeta \subseteq x/\theta $, so that $\min (x/\zeta )\geq \min (x/\theta )$, hence the inequality in the enunciation.

\noindent (\ref{commsofcgs3}) For all $x,y\in M$, $\langle x,y\rangle \in (x\wedge y)/\gamma $, thus $\langle x,y\rangle \geq \min ((x\wedge y)/\gamma )=[x,y]_{\gamma }$.\end{proof}

Recall that a lattice $M$ with $0$ is said to be {\em $0$--regular} iff, for any $\theta ,\zeta \in {\rm Con}(M)$, $0/\theta =0/\zeta $ implies $\theta =\zeta $.

\begin{proposition} If $(L,[\cdot ,\cdot ])$ is a commutator lattice and $\equiv \ =\{(a,b)\in L^2\ |\ \rho (a)=\rho (b)\}$, then:\begin{enumerate}
\item\label{allradical0} $R(L)=L$ iff $\equiv \ =\Delta _L$;
\item\label{allradical1} if $R(L)=L$, then $[\cdot ,\cdot ]=\wedge $ in $L$, in particular $L$ is a frame;
\item\label{allradical2} if $L$ is algebraic, in particular if $L$ has finite length, then: $R(L)=L$ iff $[\cdot ,\cdot ]=\wedge $ in $L$;
\item\label{allradical3} If $L$ is $0$--regular and $\rho (0)=0$, then $R(L)=L$, so $[\cdot ,\cdot ]=\wedge $ in $L$ and $L$ is a frame.\end{enumerate}\label{allradical}\end{proposition}

\begin{proof} Recall from Proposition \ref{thecong}, (\ref{thecong1}), that $\equiv \ \in {\rm Con}(L)$.

\noindent (\ref{allradical0}),(\ref{allradical1}),(\ref{allradical2}) By Remark \ref{cgrad}.

\noindent (\ref{allradical3}) If $L$ is $0$--regular, then, by (\ref{allradical1}) and Remarks \ref{0rho} and \ref{cgrad}: if $\rho (0)=0$, which means that $0/\!\!\equiv \ =\{0\}=0/\Delta _L$, then $\equiv \ =\Delta _L$, so that $R(L)=L$, thus $[\cdot ,\cdot ]=\wedge $ in $L$ and $L$ is a frame.\end{proof}

\begin{remark} Let $(L,[\cdot ,\cdot ])$ be a commutator lattice and consider a $\theta \in {\rm Con}(L)$ that preserves arbitrary joins and the commutator. If we define $[\cdot ,\cdot ]_{(\theta )}:L/\theta \times L/\theta \rightarrow L/\theta $ by $[x/\theta ,y/\theta ]_{(\theta )}=[x,y]/\theta $ for all $x,y\in L$, then it is straightforward that $(L/\theta ,[\cdot ,\cdot ]_{(\theta )})$ is a commutator lattice. In particular, $L/\!\!\equiv $ is a commutator lattice in which $[\cdot ,\cdot ]_{(\equiv )}=\wedge $. See also the proof of Proposition \ref{frame}.\label{quoequiv}\end{remark}

\begin{proposition} If $(L,[\cdot ,\cdot ])$ is a commutator lattice and $\equiv \ =\{(a,b)\in L^2\ |\ \rho (a)=\rho (b)\}$, then $L/\!\!\equiv $ is a frame, thus:\begin{itemize}
\item ${\cal A}nn(L/\!\!\equiv )={\rm PAnn}(L/\!\!\equiv )\subseteq {\rm PId}(L/\!\!\equiv )$;
\item $L/\!\!\equiv $ is Stone iff $L/\!\!\equiv $ is strongly Stone.\end{itemize}\label{frame}\end{proposition}

\begin{proof} For all $x\in L$ and any family $(y_i)_{i\in I}\subseteq L$, $\displaystyle x/\!\!\equiv \wedge \; (\bigvee _{i\in I}y_i/\!\!\equiv )=x/\!\!\equiv \wedge \; (\bigvee _{i\in I}y_i)/\!\!\equiv \ =(x\wedge (\bigvee _{i\in I}y_i))/\!\!\equiv \ =[x,\bigvee _{i\in I}y_i]/\!\!\equiv \ =(\bigvee _{i\in I}[x,y_i])/\!\!\equiv \ =\bigvee _{i\in I}[x,y_i]/\!\!\equiv \ =\bigvee _{i\in I}(x\wedge y_i)/\!\!\equiv \ =\bigvee _{i\in I}(x/\!\!\equiv \wedge \; y_i/\!\!\equiv )$.

This also followed from Remarks \ref{framecommlat} and \ref{quoequiv}.\end{proof}

\begin{lemma} If $(L,[\cdot ,\cdot ])$ is a commutator lattice such that $\rho (0)=0$ and $\equiv \ =\{(a,b)\in L^2\ |\ \rho (a)=\rho (b)\}$, then:\begin{enumerate}
\item\label{suchrho2} for all $a,b\in L$, $a\wedge b=0$ iff $[a,b]=0$;
\item\label{suchrho3} for any $U\subseteq L$, ${\rm Ann}_L(U)=\{x\in L\ |\ (\forall \, u\in U)\, ([x,u]=0)\}$.\end{enumerate}\label{suchrho}\end{lemma}

\begin{proof} (\ref{suchrho2}) By Remarks \ref{0rho} and \ref{onthecg}, if $\rho (0)=0$, then, for all $a,b\in L$: $a\wedge b=0$ iff $a\wedge b\in 0/\!\!\equiv $ iff $[a,b]\in 0/\!\!\equiv $ iff $[a,b]=0$.

\noindent (\ref{suchrho3}) By (\ref{suchrho2}) and the definition of ${\rm Ann}_L(U)$.\end{proof}

\begin{proposition} If $(L,[\cdot ,\cdot ])$ is a commutator lattice such that $\rho (0)=0$ and $\equiv \ =\{(a,b)\in L^2\ |\ \rho (a)=\rho (b)\}$, then:\begin{itemize}
\item for all $U\subseteq L$, ${\rm Ann}_L(U)={\rm Ann}_L((U]_L)={\rm Ann}_L(\bigvee U)$;
\item for any family $(I_k)_{k\in K}\subseteq {\rm Id}(L)$, $\displaystyle {\rm Ann}_L(\bigvee _{k\in K}I_k)=\bigcap _{k\in K}{\rm Ann}_L(I_k)$;
\item ${\cal A}nn(L)={\rm PAnn}(L)\subseteq {\rm PId}(L)$;
\item $L$ is Stone iff $L$ is strongly Stone.\end{itemize}\label{annid}\end{proposition}

\begin{proof} By Propositions \ref{frame} and \ref{quodist} and Proposition \ref{thecong}, (\ref{thecong1}), we have the equalities. By Proposition \ref{quodist}, (\ref{quodist3}), $L$ is Stone iff it is strongly Stone.\end{proof}

\begin{remark} By Proposition \ref{frame}, the second part of condition $(4)_{L/\equiv }$ is trivially fulfilled, so that $(4)_{L/\equiv }$ is equivalent to $(iv)_{L/\equiv }$. By Proposition \ref{annid}, if $\rho (0)=0$, then the second part of condition $(4)_L$ is trivially fulfilled, so that $(4)_L$ is equivalent to $(iv)_L$.\label{(iv)}\end{remark}

\section{Transferring Davey`s Theorem to Commutator Lattices and Related Results}
\label{thetransfer}

Throughout this section, unless mentioned otherwise, $(L,\vee ,\wedge ,[\cdot ,\cdot ],0,1)$ will be a commutator lattice and, as in Section \ref{thecongruence}, we will denote by $\equiv \ =\{(x,y)\in L^2\ |\ \rho (x)=\rho (y)\}\in {\rm Con}(L)$. Recall from Remark \ref{0rho} that $0/\!\! \equiv \ =\{0\}$ iff $\rho (0)=0$.

Let us define $\dotvee :R(L)^2\rightarrow R(L)$ by $\rho (x)\dotvee \rho (y)=\rho (\rho (x)\vee \rho (y))$ for all $x,y\in L$, and, for any family $(x_i)_{i\in I}\subseteq L$, let us denote by $\displaystyle \stackrel{\bullet }{\bigvee _{i\in I}}\rho (x_i)=\rho (\bigvee _{i\in I}\rho (x_i))\in R(L)$.

\begin{remark} Note from the definitions of $\dotvee $ and $\equiv $ and Proposition \ref{thecong}, (\ref{thecong1}), that $\equiv $ preserves $\dotvee $ over arbitrary families of elements of $R(L)$.\end{remark}

\begin{proposition} If $(L,[\cdot ,\cdot ])$ is a commutator lattice and $\dotvee $ is the binary operation defined on $R(L)$ as above, then:\begin{enumerate}
\item\label{framerho1} $(R(L),\dotvee ,\wedge ,\rho (0),1)$ is a frame, isomorphic to $L/\!\! \equiv $;
\item\label{framerho2} in the commutator lattice $(R(L),\dotvee ,\wedge ,\wedge ,\rho (0),1)$, ${\rm Spec}_{R(L)}={\rm Spec}_L$ and $R(R(L))=R(L)$, in particular $\rho (0)\in R(R(L))$;
\item\label{framerho3} in the commutator lattice $
(L/\!\! \equiv ,\vee ,\wedge ,\wedge ,0/\!\!\equiv ,1/\!\!\equiv )$, ${\rm Spec}_{L/\equiv }=\{p/\!\!\equiv \ |\ p\in V(\rho (0))=[\rho (0))_L\cap {\rm Spec}_L\}$ and $R(L/\!\! \equiv )=L/\!\! \equiv $, in particular $\rho (0/\!\!\equiv )=0/\!\!\equiv $.\end{enumerate}\label{framerho}\end{proposition}

\begin{proof} (\ref{framerho1}) $1=\rho (1)\in R(L)$ and, for all $a,b\in L$, $\rho (a)\wedge \rho (b)=\rho (a\wedge b)\in R(L)$, and $\rho (a)\dotvee \rho (b)=\rho (\rho (a)\vee \rho (b))=\rho (a\vee b)$, from which it is straightforward that $(R(L),\dotvee ,\wedge ,\rho (0),1)$ is a bounded lattice.

Let $f:L\rightarrow R(L)$, for all $x\in L$, $f(x)=\rho (x)$. Then $f$ is surjective and, by the above, for all $a,b\in L$, $f(a\wedge b)=f(a)\wedge f(b)$ and $f(a\vee b)=f(a)\dotvee f(b)$, hence $f$ is a surjective lattice morphism. By the definition of $\equiv $, ${\rm Ker}(f)=\ \equiv $. Hence the Main Isomorphism Theorem gives us a lattice isomorphism $h:L/\!\! \equiv \ \rightarrow R(L)$, defined by $h(x/\!\! \equiv )=\rho (x)$ for all $x\in L$. By Proposition \ref{frame}, it follows that $R(L)$ is a frame and $h$ is a frame isomorphism.

\noindent (\ref{framerho2}) Since $(R(L),\dotvee ,\wedge ,\rho (0),1)$ is a frame by (\ref{framerho1}), $(R(L),\dotvee ,\wedge ,\wedge ,\rho (0),1)$ is a commutator lattice. ${\rm Spec}_L\subseteq R(L)=\{\rho (u)\ |\ u\in L\}$, and, for any $a,b,x\in L$: $\rho (a)\wedge \rho (b)\leq \rho (x)$ iff $\rho (a)\leq \rho (x)$ and $\rho (b)\leq \rho (x)$ iff $a\leq \rho (x)$ and $b\leq \rho (x)$ iff $a\wedge b\leq \rho (x)$ iff $\rho ([a,b])=\rho (a\wedge b)\leq \rho (x)$ iff $[a,b]\leq \rho (x)$, hence: $\rho (x)\in {\rm Spec}_{R(L)}$ iff $\rho (x)\in {\rm Spec}_L$, therefore ${\rm Spec}_L={\rm Spec}_{R(L)}$. Thus, in $R(L)$, for any $x\in L$, the radical of $\rho (x)$ is $\displaystyle \bigwedge \{p\in {\rm Spec}_{R(L)}\ |\ \rho (x)\leq p\}=\bigwedge \{p\in {\rm Spec}_L\ |\ \rho (x)\leq p\}=\rho (\rho (x))=\rho (x)$, which means that every element of the commutator lattice $R(L)$ is a radical element, in particular the first element of this lattice, $\rho (0)$, is a radical element.

\noindent (\ref{framerho3}) By (\ref{framerho1}) and (\ref{framerho2}) and the definition of the frame isomorphism $h$.\end{proof}

\begin{remark} By Remarks \ref{inclann} and \ref{0rho}, we have ${\rm Ann}_L(U)/\!\!\equiv \ \subseteq {\rm Ann}_{L/\equiv }(U/\!\!\equiv )$ for any $U\subseteq L$ and, if $\rho (0)=0$, then the properties of Lemma \ref{anntheta} hold for $M$ and $\theta $ replaced by $L$ and $\equiv $, respectively.\end{remark}

For all $x,y\in L$, let us define $\displaystyle x\rightarrow y=\bigvee \{a\in L\ |\ [x,a]\leq y\}$ and $\displaystyle \neg \, x=x\rightarrow 0=\bigvee \{a\in L\ |\ [x,a]=0\}$.

\begin{remark} Let $x,y\in L$. By Remark \ref{maxcomm}, $\displaystyle x\rightarrow y=\max \{a\in L\ |\ [x,a]\leq y\}$ and $\displaystyle \neg \, x=\max \{a\in L\ |\ [x,a]=0\}$, so that, by Lemma \ref{suchrho}, (\ref{suchrho3}), if $\rho (0)=0$, then $\displaystyle \neg \, x=\bigvee _{a\in {\rm Ann}_L(x)}a=\max ({\rm Ann}_L(x))$.\label{negann}\end{remark}

\begin{lemma} If $(L,[\cdot ,\cdot ])$ is a commutator lattice, then, for all $x,y,z\in L$:\begin{enumerate}
\item\label{residuation1} $[x,y]\leq z$ iff $x\leq y\rightarrow z$;
\item\label{residuation2} if $[y,1]=y$, then: $y\rightarrow z=1$ iff $y\leq z$.
\end{enumerate}\label{residuation}\end{lemma}

\begin{proof} (\ref{residuation1}) $y\rightarrow z=\max \{a\in L\ |\ [a,y]\leq z\}$, so both implications hold.

\noindent (\ref{residuation2}) If $[y,1]=y$, then, by (\ref{residuation1}): $y\rightarrow z=1$ iff $1\leq y\rightarrow z$ iff $y=[1,y]\leq z$.\end{proof}

\begin{lemma} If $(L,[\cdot ,\cdot ])$ is a commutator lattice, then, for all $e\in {\cal B}(L)$ and all $a\in L$ such that $[1,e\wedge a]=e\wedge a$, we have $e\wedge a=[e,a]$.\label{meetbool}\end{lemma}

\begin{proof} Since $e\in {\cal B}(L)$, we have $e\vee f=1$ and $e\wedge f=0$ for some $f\in L$. Then $[e,a]\leq e\wedge a=[1,e\wedge a]=[e\vee f,e\wedge a]=[e,e\wedge a]\vee [f,e\wedge a]\leq [e,e\wedge a]\vee (f\wedge e\wedge a)=[e,e\wedge a]\vee 0=[e,e\wedge a]\leq [e,a]$, hence: $e\wedge a=[e,a]$.\end{proof}

\begin{proposition} If $(L,[\cdot ,\cdot ])$ is a commutator lattice such that $[x,1]=x$ for all $x\in L$, then $e\wedge a=[e,a]$ for all $e\in {\cal B}(L)$ and all $a\in L$, in particular $[\cdot ,\cdot ]=\wedge $ in ${\cal B}(L)$, and ${\cal B}(L)$ is a Boolean sublattice of $L$ whose complementation is defined by $\neg \, e=e\rightarrow 0=\max \{a\in L\ |\ [e,a]=0\}=\max ({\rm Ann}_L(e))$ for all $e\in {\cal B}(L)$.\label{boolcenter}\end{proposition}

\begin{proof} By Lemma \ref{meetbool}, $e\wedge a=[e,a]$ for all $e\in {\cal B}(L)$ and all $a\in L$, in particular for all $e,a\in {\cal B}(L)$.

Obviously, $0,1\in {\cal B}(L)$. Now let $x,y\in {\cal B}(L)$, so that $x\vee \overline{x}=y\vee \overline{y}=1$ and $x\vee \overline{x}=y\vee \overline{y}=1$ and $x\wedge \overline{x}=y\wedge \overline{y}=0$ for some $\overline{x},\overline{y}\in {\cal B}(L)$, so that $[x,\overline{x}]=[y,\overline{y}]=0$, as well. Then $[x\vee y,\overline{x}\wedge \overline{y}]=[x,\overline{x}\wedge \overline{y}]\vee [y,\overline{x}\wedge \overline{y}]\leq [x,\overline{x}]\vee [y,\overline{y}]=0$, so $[x\vee y,\overline{x}\wedge \overline{y}]=0$. By Lemma \ref{meetbool}, $[x,x\vee y]=x\wedge (x\vee y)=x$, hence $x\vee y\vee (\overline{x}\wedge \overline{y})=x\vee y\vee [\overline{x},\overline{y}]=[1,x\vee y]\vee [\overline{x},\overline{y}]=[x\vee \overline{x},x\vee y]\vee [\overline{x},\overline{y}]=[x,x\vee y]\vee [\overline{x},x\vee y]\vee [\overline{x},\overline{y}]=x\vee [\overline{x},x\vee y\vee \overline{y}]=x\vee [\overline{x},1]=x\vee \overline{x}=1$, hence $x\vee y\in {\cal B}(L)$ and $\overline{x}\wedge \overline{y}\in {\cal B}(L)$, thus also $x\wedge y\in {\cal B}(L)$, since we can interchange $x$ and $\overline{x}$, respectively $y$ and $\overline{y}$ in the above.

Therefore ${\cal B}(L)$ is a bounded sublattice of $L$ in which the meet coincides to $[\cdot ,\cdot ]$, thus ${\cal B}(L)$ is a bounded distributive lattice, and it is clearly complemented, so ${\cal B}(L)$ is a Boolean sublattice of $L$. Let $e\in {\cal B}(L)$. As in every Boolean algebra, the complement of $e$ in ${\cal B}(L)$ is $\max \{a\in {\cal B}(L)\ |\ e\wedge a=0\}\leq \max \{a\in L\ |\ e\wedge a=0\}=\max \{a\in L\ |\ [e,a]=0\}=\neg \, e$, hence $e\vee \neg \, e=1$. But, by the above, $\neg \, e=\max ({\rm Ann}_L(e))\in {\rm Ann}_L(e)$, thus $e\wedge \neg \, e=0$. Hence $\neg \, e\in {\cal B}(L)$ and $\neg \, e$ is the complement of $e$ in ${\cal B}(L)$.\end{proof}

\begin{remark} If $[1,1]=x<1$, then no $y\in [x)_L$ can be prime, thus $\displaystyle \rho (x)=\bigwedge \emptyset =1=\rho (1)$, hence $1\neq x\in 1/\!\!\equiv $. Therefore $1/\!\!\equiv \ =\{1\}$ implies $[1,1]=1$.\label{1comm}\end{remark}

\begin{proposition} If $(L,[\cdot ,\cdot ])$ is a commutator lattice such that $\rho (0)=0$, then:\begin{enumerate}
\item\label{charstone1} for all $x\in L$, ${\rm Ann}_L(x)=(\neg \, x]_L$;
\item\label{charstone2} $L$ is a Stone lattice iff, for all $x\in L$, $\neg \, x\in {\cal B}(L)$;
\item\label{charstone3} if $[x,1]=x$ for all $x\in L$, then, for all $e\in {\cal B}(L)$, $(e]_L={\rm Ann}_L(\neg \, e)\in {\rm PAnn}(L)$.\end{enumerate}\label{charstone}\end{proposition}

\begin{proof} (\ref{charstone1}) By Proposition \ref{annid} and Remark \ref{negann}.

\noindent (\ref{charstone2}) By (\ref{charstone1}) and the definition of a Stone lattice.

\noindent (\ref{charstone3}) By (\ref{charstone1}) and Proposition \ref{boolcenter}, for all $e\in {\cal B}(L)$, $(e]_L=(\neg \, \neg \, e]_L={\rm Ann}_L(\neg \, e)$\linebreak $\in {\rm PAnn}(L)$.\end{proof}

\begin{proposition} If $(L,[\cdot ,\cdot ])$ is a commutator lattice such that $\rho (0)=0$ and $[x,1]=x$ for all $x\in L$, then:\begin{itemize}
\item ${\cal B}({\rm Id}(L))=\{(e]_L\ |\ e\in {\cal B}(L)\}\subseteq {\rm PAnn}(L)\subseteq {\cal A}nn(L)$;
\item $L$ is a Stone lattice iff ${\rm PAnn}(L)={\cal B}({\rm Id}(L))$;
\item $L$ is a strongly Stone lattice iff ${\cal A}nn(L)={\cal B}({\rm Id}(L))$.\end{itemize}\label{boolid}\end{proposition}

\begin{proof} By Proposition \ref{boolidtheta}, (\ref{boolidtheta2}), and Proposition \ref{charstone}, (\ref{charstone3}).\end{proof}

Let us see, in the following proposition, some side results on compact elements.

\begin{proposition} Let $(L,[\cdot ,\cdot ])$ be a commutator lattice and $\equiv \ =\{(a,b)\in L^2\ |\ \rho (a)=\rho (b)\}$.\begin{enumerate}
\item\label{boolcp1} If $1\in {\rm Cp}(L)$, then $\{x\in {\cal B}(L)\ |\ [x,1]=x\}\subseteq {\rm Cp}(L)$.
\item\label{boolcp2} If $1\in {\rm Cp}(L)$ and $[x,1]=x$ for all $x\in {\cal B}(L)$, then ${\cal B}(L)\subseteq {\rm Cp}(L)$.
\item\label{boolcp0} If $1\in {\rm Cp}(L)$ and $1/\!\! \equiv \ =\{1\}$, then $1/\!\!\equiv \in {\rm Cp}(L/\!\! \equiv )$, ${\cal B}(L/\!\! \equiv )\subseteq {\rm Cp}(L/\!\! \equiv )$ and, in $L/\!\! \equiv $, $V(x/\!\! \equiv )\neq \emptyset $ for all $x\in L\setminus \{1\}$.\end{enumerate}\label{boolcp}\end{proposition}

\begin{proof} (\ref{boolcp1}) Let $x\in {\cal B}(L)$, so that $x\vee y=1$ and $x\wedge y=0$ for some $y\in L$, so we also have $[x,y]=0$. Assume that $[x,1]=x$ and $1\in {\rm Cp}(L)$, and let $\emptyset \neq M\subseteq L$ such that $\displaystyle x\leq \bigvee M$, so that $\displaystyle 1=x\vee y\leq \bigvee M\vee y$, therefore, since $1\in {\rm Cp}(L)$, $\displaystyle 1=x\vee y=\bigvee _{i=1}^nx_i\vee y$ for some $n\in \N ^*$ and some $x_1,\ldots ,x_n\in M$. Then $\displaystyle x=[x,1]=[x,\bigvee _{i=1}^nx_i\vee y]=[x,\bigvee _{i=1}^nx_i]\vee [x,y]=[x,\bigvee _{i=1}^nx_i]\leq \bigvee _{i=1}^nx_i$, hence $\displaystyle x\leq \bigvee _{i=1}^nx_i$.

\noindent (\ref{boolcp2}) By (\ref{boolcp1}).

\noindent (\ref{boolcp0}) Let $\emptyset \neq M\subseteq L$ such that $\displaystyle (\bigvee _{x\in M}x)/\!\! \equiv \ =\bigvee _{x\in M}x/\!\! \equiv \ =1/\!\! \equiv $. If $1/\!\! \equiv \ =\{1\}$, then it follows that $\displaystyle \bigvee _{x\in M}x=1$. If, furthermore, $1\in {\rm Cp}(L)$, then we obtain that $\displaystyle 1=\bigvee _{i=1}^nx_i=1$ for some $n\in \N ^*$ and some $x_1,\ldots ,x_n\in M$, hence $\displaystyle 1/\!\! \equiv \ =(\bigvee _{i=1}^nx_i)/\!\! \equiv \ =\bigvee _{i=1}^nx_i/\!\! \equiv \ $, therefore $1/\!\!\equiv \ =1/\!\! \equiv \ \in {\rm Cp}(L/\!\! \equiv )$.

Since $L/\!\! \equiv $ is a commutator lattice with $[\cdot ,\cdot ]=\wedge $, $L/\!\! \equiv $ fulfills $[x/\!\! \equiv ,1/\!\!\equiv ]=x/\!\! \equiv \wedge \; 1/\!\!\equiv \ =x/\!\! \equiv $ for all $x\in L$. Now apply (\ref{boolcp2}) to obtain that ${\cal B}(L/\!\! \equiv )\subseteq {\rm Cp}(L/\!\! \equiv )$, and Lemma \ref{smimaxspec}, (\ref{smimaxspec4}), to obtain that, in $L/\!\! \equiv $, $V(x/\!\! \equiv )\neq \emptyset $ for all $x\in L\setminus 1/\!\! \equiv \ =L\setminus \{1\}$.\end{proof}

\begin{proposition} If $(L,[\cdot ,\cdot ])$ is a commutator lattice such that $[x,1]=x$ for all $x\in L$ and $\equiv \ =\{(a,b)\in L^2\ |\ \rho (a)=\rho (b)\}$, then:\begin{enumerate}
\item\label{boolmorph12} $p_{\equiv }\mid _{{\cal B}(L)}:{\cal B}(L)\rightarrow {\cal B}(L/\!\! \equiv )$ is a Boolean morphism, which is injective iff $1/\!\! \equiv \cap \ {\cal B}(L)=\{1\}$ iff $0/\!\! \equiv \cap \ {\cal B}(L)=\{0\}$; 
\item\label{boolmorph35} if $\rho (0)=0$ or $1/\!\! \equiv \; =\{1\}$, in particular if $\rho (0)=0$ or $1\in {\rm Cp}(L)$, then the Boolean morphism $p_{\equiv }\mid _{{\cal B}(L)}:{\cal B}(L)\rightarrow {\cal B}(L/\!\! \equiv )$ is injective;
\item\label{boolmorph46} if $\rho (0)=0$ and $1/\!\! \equiv \; =\{1\}$, in particular if $\rho (0)=0$ and $1\in {\rm Cp}(L)$, then:\begin{itemize}
\item for all $x\in L$: $x/\!\! \equiv \in {\cal B}(L/\!\! \equiv )$ iff $x\in {\cal B}(L)$;
\item $p_{\equiv }\mid _{{\cal B}(L)}:{\cal B}(L)\rightarrow {\cal B}(L/\!\! \equiv )$ is a Boolean isomorphism.\end{itemize}\end{enumerate}\label{boolmorph}\end{proposition}

\begin{proof} Assume that $[x,1]=x$ for all $x\in L$, so that ${\cal B}(L)$ is a Boolean sublattice of $L$ by Proposition \ref{boolcenter}.

\noindent (\ref{boolmorph12}) By the above and Remark \ref{booltheta23}.

\noindent (\ref{boolmorph35}) By (\ref{boolmorph12}) and Remark \ref{0rho}, with Lemma \ref{cls1cpct} for the particular case.

\noindent (\ref{boolmorph46}) By Lemma \ref{booltheta}, (\ref{booltheta5}) and Remark \ref{0rho}, with Lemma \ref{cls1cpct} for the particular case.\end{proof}

\begin{remark} By Proposition \ref{boolmorph}, 
(\ref{boolmorph46}), if $1/\!\! \equiv \ =\{1\}$ and $\rho (0)=0$, then $\equiv $ has the BLP (see \cite{cblp}).\end{remark}

\begin{proposition} Let $(L,[\cdot ,\cdot ])$ be a commutator lattice such that $\rho (0)=0$ and $[x,1]=x$ for all $x\in L$, $\equiv \ =\{(a,b)\in L^2\ |\ \rho (a)=\rho (b)\}$, $U\subseteq L$, $a\in L$ and $e\in {\cal B}(L)$. Then:\begin{enumerate}
\item\label{boollambda12} $a/\!\! \equiv \ \leq e/\!\! \equiv $ iff $a\leq e$; $e=\max (e/\!\! \equiv )=\rho (e)$; ${\cal B}(L)\subseteq R(L)$;
\item\label{boollambda0} if $L$ is complemented, then $R(L)=L$, $[\cdot ,\cdot ]=\wedge $ in $L$ and $L$ is a complete Boolean algebra;
\item\label{boollambda45} $(\rho (a)]_L/\!\! \equiv \ ={\rm Ann}_L(U)/\!\! \equiv $ iff $(\rho (a)]_L={\rm Ann}_L(U)$; $(e]_L/\!\! \equiv \ ={\rm Ann}_L(U)/\!\! \equiv $ iff $(e]_L={\rm Ann}_L(U)$.\end{enumerate}\label{boollambda}\end{proposition}

\begin{proof} (\ref{boollambda12}) By Lemma \ref{compactata}, (\ref{justnoticed2}), Proposition \ref{boolcenter}, Proposition \ref{thecong}, (\ref{thecong1}), and Remark \ref{emaxcls}, $e=\max (e/\!\! \equiv )=\rho (e)\in R(L)$, so that ${\cal B}(L)\subseteq R(L)$ and $a/\!\! \equiv \ \leq e/\!\! \equiv $ iff $a\leq e$.

\noindent (\ref{boollambda0}) By (\ref{boollambda12}), if $L$ is complemented, then $L={\cal B}(L)\subseteq R(L)\subseteq L$, so that $L={\cal B}(L)=R(L)$, thus $L$ is a complete Boolean algebra and has $[\cdot ,\cdot ]=\wedge $ by Proposition \ref{boolcenter} and the fact that $L$ is a complete lattice.

\noindent (\ref{boollambda45}) By (\ref{boollambda12}), Lemma \ref{compactata}, (\ref{bmaxth13}), and Proposition \ref{thecong}, (\ref{thecong1}).\end{proof}

\begin{proposition} If $(L,[\cdot ,\cdot ])$ is a commutator lattice such that $\rho (0)=0$ and $\equiv \ =\{(a,b)\in L^2\ |\ \rho (a)=\rho (b)\}$, then, for any cardinality $\kappa $:\begin{enumerate}
\item\label{echiv(i)1} $(1)_{\kappa ,L}$ implies $(1)_{\kappa ,L/\equiv }$ (that is, if $L$ is $\kappa $--Stone, then $L/\!\!\equiv $ is $\kappa $--Stone);
\item\label{echiv(i)2} if $1/\!\!\equiv \  =\{1\}$ and $[x,1]=x$ for all $x\in L$, in particular if $1\in {\rm Cp}(L)$ and $[x,1]=x$ for all $x\in L$, then properties $(1)_{\kappa ,L}$ and $(1)_{\kappa ,L/\equiv }$ are equivalent (that is $L$ is $\kappa $--Stone iff $L/\!\!\equiv $ is $\kappa $--Stone).\end{enumerate}\label{echiv(i)}\end{proposition}

\begin{proof} (\ref{echiv(i)1}) By Remark \ref{0rho} and Proposition \ref{1theta}, (\ref{1theta1}).

\noindent (\ref{echiv(i)2}) By Remark \ref{0rho}, Proposition \ref{boollambda}, (\ref{boollambda12}), and Proposition \ref{1theta}, (\ref{1theta2}), with Lemma \ref{cls1cpct} for the particular case.\end{proof}

\begin{corollary} Let $(L,[\cdot ,\cdot ])$ be a commutator lattice such that $\rho (0)=0$ and $\equiv \ =\{(a,b)\in L^2\ |\ \rho (a)=\rho (b)\}$.\begin{itemize}
\item If $L$ is Stone, respectively strongly Stone, then $L/\!\!\equiv $ is Stone, respectively strongly Stone.
\item if $1/\!\!\equiv \  =\{1\}$ and $[x,1]=x$ for all $x\in L$, in particular if $1\in {\rm Cp}(L)$ and $[x,1]=x$ for all $x\in L$, then $L$ is Stone, respectively strongly Stone, iff $L/\!\!\equiv $ is Stone, respectively strongly Stone.\end{itemize}\label{echivstone}\end{corollary}

\begin{proof} By Proposition \ref{echiv(i)} applied for $\kappa =1$, then for all cardinalities $\kappa $.\end{proof}

\begin{proposition} If $(L,[\cdot ,\cdot ])$ is a commutator lattice such that $\rho (0)=0$ and $\equiv \ =\{(a,b)\in L^2\ |\ \rho (a)=\rho (b)\}$, then, for any cardinality $\kappa $:\begin{enumerate}
\item\label{echiv(ii)1} if $1/\!\!\equiv \ =\{1\}$, in particular if $1\in {\rm Cp}(L)$ and $[1,1]=1$, then $(2)_{\kappa ,L}$ implies $(2)_{\kappa ,L/\equiv }$;
\item\label{echiv(ii)2} if $1/\!\!\equiv \ =\{1\}$ and $[x,1]=x$ for all $x\in L$, in particular if $1\in {\rm Cp}(L)$ and $[x,1]=x$ for all $x\in L$, then properties $(2)_{\kappa ,L}$ and $(2)_{\kappa ,L/\equiv }$ are equivalent.\end{enumerate}\label{echiv(ii)}\end{proposition}

\begin{proof} We get the particular cases from Lemma \ref{cls1cpct}.

\noindent (\ref{echiv(ii)1}) By Remark \ref{0rho} and Proposition \ref{2theta}, (\ref{2theta1}).

\noindent (\ref{echiv(ii)2}) By Remark \ref{0rho}, Proposition \ref{boollambda}, (\ref{boollambda12}), and Proposition \ref{2theta}, (\ref{2theta2}).\end{proof}

\begin{proposition} If $(L,[\cdot ,\cdot ])$ is a commutator lattice such that $\rho (0)=0$ and $\equiv \ =\{(a,b)\in L^2\ |\ \rho (a)=\rho (b)\}$, then, for any cardinality $\kappa $, the properties $(3)_{\kappa ,L}$ and $(3)_{\kappa ,L/\equiv }$ are equivalent.\label{echiv(iii)}\end{proposition}

\begin{proof} By Remark \ref{0rho} and Proposition \ref{3theta}.\end{proof}

\begin{proposition} If $(L,[\cdot ,\cdot ])$ is a commutator lattice such that $\rho (0)=0$ and $\equiv \ =\{(a,b)\in L^2\ |\ \rho (a)=\rho (b)\}$, then, for any cardinality $\kappa $, the properties $(iv)_L$, $(4)_{\kappa ,L}$ and $(4)_{\kappa ,L/\equiv }$ are equivalent.\label{echiv(iv)}\end{proposition}

\begin{proof} By Remarks \ref{0rho} and \ref{(iv)} and Proposition \ref{4theta}.\end{proof}

\begin{proposition} If $(L,[\cdot ,\cdot ])$ is a commutator lattice and $\equiv \ =\{(a,b)\in L^2\ |\ \rho (a)=\rho (b)\}$, then, for any cardinality $\kappa $:\begin{enumerate}
\item\label{echiv(v)1} $(5)_{\kappa ,L}$ implies $(5)_{\kappa ,L/\equiv }$;
\item\label{echiv(v)2} if $\rho (0)=0$ and $1/\!\!\equiv \  =\{1\}$, in particular if $\rho (0)=0$, $1\in {\rm Cp}(L)$ and $[1,1]=1$, then $(5)_{\kappa ,L}$ is equivalent to $(5)_{\kappa ,L/\equiv }$.\end{enumerate}\label{echiv(v)}\end{proposition}

\begin{proof} (\ref{echiv(v)1}) By Proposition \ref{5theta}, (\ref{5theta1}).

\noindent (\ref{echiv(v)2}) By Remark \ref{0rho} and Proposition \ref{5theta}, (\ref{5theta2}), with Lemma \ref{cls1cpct} for the particular case.\end{proof}

\begin{theorem} Let $(L,[\cdot ,\cdot ])$ be a commutator lattice with $\rho (0)=0$, and consider the congruence $\equiv \ =\{(a,b)\in L^2\ |\ \rho (a)=\rho (b)\}$ of $L$. If $1/\!\!\equiv\; =\{1\}$ and $[x,1]=x$ for all $x\in L$, in particular if $1\in {\rm Cp}(L)$ and $[x,1]=x$ for all $x\in L$, then:\begin{enumerate}
\item\label{ourdavey1} for any $h,i,j\in \overline{1,5}$ and any nonzero cardinality $\kappa $, conditions $(iv)_L$, $(h)_{\kappa ,L}$, $(i)_{<\infty ,L}$ and $(j)_L$ are equivalent, thus $L$ satisfies the equivalences from Corollary \ref{clearer}, (\ref{clearer1});
\item\label{ourdavey2} let $m$ be a nonzero cardinality; if, in Definition \ref{defcommlat}, we replace the conditions that $L$ is complete and $[\cdot ,\cdot ]$ is completely distributive w.r.t. the join by $L$ being closed w.r.t. the joins of all families of elements of cardinalities at most $m$ and $[\cdot ,\cdot ]$ being distributive w.r.t. such joins, then we get that: for any $h,i\in \overline{1,5}$ and any nonzero cardinality $\kappa \leq m$, conditions $(h)_{\kappa ,L}$ and $(i)_{<\infty ,L}$ are equivalent, thus $L$ satisfies the equivalences from Corollary \ref{clearer}, (\ref{clearer2}).\end{enumerate}\label{ourdavey}\end{theorem}

\begin{proof} We get the particular case from Lemma \ref{cls1cpct}.

\noindent (\ref{ourdavey1}) By Theorem \ref{davey}, (\ref{davey2}), Proposition \ref{echiv(i)}, (\ref{echiv(i)2}),  Proposition \ref{echiv(ii)}, (\ref{echiv(ii)2}), Propositions \ref{echiv(iii)} and \ref{echiv(iv)}, Proposition \ref{echiv(v)}, (\ref{echiv(v)2}), and Proposition \ref{frame}.

\noindent (\ref{ourdavey2}) By Theorem \ref{transferdavey}, (\ref{transferdavey0}), Proposition \ref{echiv(i)}, (\ref{echiv(i)2}),  Proposition \ref{echiv(ii)}, (\ref{echiv(ii)2}), Propositions \ref{echiv(iii)} and \ref{echiv(iv)}, Proposition \ref{echiv(v)}, (\ref{echiv(v)2}), and the fact that, in this case, $L/\!\!\equiv $ is closed w.r.t. the joins of all families of elements of cardinalities at most $\kappa $ and has the meet distributive w.r.t. the joins of families of elements of cardinalities at most $\kappa $, which follows imediately through an argument analogous to the first proof of Proposition \ref{frame}.\end{proof}

\begin{corollary} Let $(L,[\cdot ,\cdot ])$ be a commutator lattice with $\rho (0)=0$. Then:\begin{itemize}
\item for any nonzero cardinalities $\kappa $ and $\mu $, conditions $(3)_{\kappa ,L}$, $(4)_{\mu ,L}$ and $(iv)_L$ are equivalent;
\item if $\{a\in L\ |\ \rho (a)=1\}=\{1\}$, in particular if $1\in {\rm Cp}(L)$ and $[1,1]=1$, then, for any nonzero cardinalities $\kappa $, $\lambda $ and $\mu $, conditions $(3)_{\kappa ,L}$, $(4)_{\lambda ,L}$, $(iv)_L$ and $(5)_{\mu ,L}$ are equivalent.\end{itemize}

Let $m$ be a nonzero cardinality. If, in Definition \ref{defcommlat}, we replace the conditions that $L$ is complete and $[\cdot ,\cdot ]$ is completely distributive w.r.t. the join by $L$ being closed w.r.t. the joins of all families of elements of cardinalities at most $m$ and $[\cdot ,\cdot ]$ being distributive w.r.t. such joins, then:\begin{itemize}
\item for any nonzero cardinalities $\kappa \leq m$ and $\mu \leq m$, conditions $(3)_{\kappa ,L}$ and $(4)_{\mu ,L}$ are  equivalent;
\item if $\{a\in L\ |\ \rho (a)=1\}=\{1\}$, in particular if $1\in {\rm Cp}(L)$ and $[1,1]=1$, then, for any nonzero cardinalities $\kappa \leq m$, $\lambda \leq m$ and $\mu \leq m$, conditions $(3)_{\kappa ,L}$, $(4)_{\lambda ,L}$, and $(5)_{\mu ,L}$ are equivalent.
\end{itemize}\label{3eq4}\end{corollary}

\begin{proof} By Remark \ref{0rho} and Propositions \ref{echiv345} and \ref{frame}, with Lemma \ref{cls1cpct} for the particular cases.\end{proof}

\section{Transferring Davey`s Theorem to Congruence Lattices and Preservation of the Conditions from Davey`s Theorem by Direct Products and Sublattices}

Throughout this section, unless mentioned otherwise, $A$ will be a member of a congruence\linebreak --modular variety ${\cal V}$. Following \cite{retic}, we use these notations for the radical of a congruence of $A$ in the commutator lattice $({\rm Con}(A),\vee ,\cap ,[\cdot ,\cdot ]_A,\Delta _A,\nabla _A)$ and the lattice congruence $\equiv $ associated to the same commutator lattice: $\rho _A(\alpha )=\bigcap ({\rm Spec}(A)\cap [\alpha )_{{\rm Con}(A)})$ for all $\alpha \in {\rm Con}(A)$ and $\equiv _A=\{(\theta ,\zeta )\in {\rm Con}(A)^2\ |\ \rho _A(\theta )=\rho _A(\zeta )\}$.

Recall from the end of Section \ref{theth} that $A$ is semiprime, that is $\rho _A(\Delta _A)=\Delta _A$, if the commutator $[\cdot ,\cdot ]_A$ of $A$ equals the intersection, in particular if ${\cal V}$ is congruence--distributive. Recall, also, that $\nabla _A$ is a compact congruence of $A$ if ${\cal V}$ is semi--degenerate, and that $[\theta ,\nabla _A]_A=\theta $ for all $\theta \in {\rm Con}(A)$ if ${\cal V}$ is semi--degenerate or the commutator of $A$ equals the intersection, in particular if ${\cal V}$ is congruence--distributive. Of course, $\nabla _A$ is a compact congruence of $A$ in other particular cases such as the case when ${\rm Con}(A)$ is compact, in particular when ${\rm Con}(A)$ has finite height, in particular when ${\rm Con}(A)$ is finite, in particular when $A$ is finite.

\begin{remark} By Lemma \ref{cls1cpct} and the above, the class of $\nabla _A$ w.r.t. $\equiv _A$ is a singleton if ${\cal V}$ is semi--degenerate, or if, for instance, ${\rm Con}(A)$ is compact and ${\cal V}$ is congruence--distributive.\end{remark}

$\equiv _A$ satisfies the properties from Proposition \ref{thecong}, in particular, by (\ref{thecong1}), $\equiv _A$ is a lattice congruence of ${\rm Con}(A)$ that preserves arbitrary joins and the commutator of $A$ and satisfies $[\alpha ,\beta ]_A\equiv _A\alpha \cap \beta $ for all $\alpha ,\beta \in {\rm Con}(A)$. Moreover, since the meet in ${\rm Con}(A)$ is the intersection of congruences, the surjectivity of the map $p_{\equiv _A}:{\rm Con}(A)\rightarrow {\rm Con}(A)/\!\!\equiv _A$ ensures us that:

\begin{proposition} If $A$ is a member of a congruence--modular variety and $\equiv _A=\{(\theta ,\zeta )\in {\rm Con}(A)^2\ |\ \rho _A(\theta )=\rho _A(\zeta )\}$, then $\equiv _A$ is a complete congruence of the complete lattice ${\rm Con}(A)$, so all its classes are intervals.\label{cpltcg}\end{proposition}

Again by Proposition \ref{thecong}, the radical congruences of $A$ are the maxima of the classes of $\equiv _A$ and, for each radical congruence $\theta $ of $A$, $\min (\theta /\!\!\equiv _A)=\min \{\alpha \in {\rm Con}(A)\ |\ \rho _A(\alpha )=\theta \}$; also, for all $\beta \in \theta /\!\!\equiv _A$, we have $[\beta ,\min (\theta /\!\!\equiv _A)]_A=\min (\theta /\!\!\equiv _A)$. By Proposition \ref{frame}, ${\rm Con}(A)/\!\!\equiv _A$ is a frame, so Proposition \ref{allradical}, (\ref{allradical2}), gives us:

\begin{corollary} If $A$ is a member of a congruence--modular variety, then all congruences of $A$ are radical iff the commutator of $A$ equals the intersection of congruences.\end{corollary}

Note that, if $A$ is semiprime, then, for $\theta =\; \equiv _A$, the annihilators in ${\rm Con}(A)$ satisfy the properties from Lemmas \ref{anntheta} and \ref{panntheta}. Also, if $A$ is semiprime, then the properties from Lemmas \ref{suchrho} and \ref{compactata} and Proposition \ref{annid} hold, in particular:

\begin{corollary} If $A$ is a semiprime member of a congruence--modular variety, in particular if $A$ belongs to a congruence--distributive variety, then:\begin{itemize}
\item ${\cal A}nn({\rm Con}(A))={\rm PAnn}({\rm Con}(A))\subseteq {\rm PId}({\rm Con}(A))$;
\item ${\rm Con}(A)$ is Stone iff ${\rm Con}(A)$ is strongly Stone.\end{itemize}\end{corollary}

So, if $A$ is semiprime, then all annihilators in ${\rm Con}(A)$ have maxima, so that ${\rm Con}(A)$ is Stone iff these maxima are complemented; see also Proposition \ref{charstone}, (\ref{charstone2}). 

By Lemma \ref{meetbool} and Proposition \ref{boolcenter}, if ${\cal V}$ is semi--degenerate or congruence\linebreak --distributive, then $[\alpha ,\theta ]_A=\alpha \cap \theta $ for all $\alpha \in {\rm Con}(A)$ and all $\theta \in {\cal B}({\rm Con}(A))$ and ${\cal B}({\rm Con}(A))$ is a Boolean sublattice of ${\rm Con}(A)$ in which the complementation is defined by $\neg \, \theta =\max ({\rm Ann}_{{\rm Con}(A)}(\theta ))$ for all $\theta \in {\cal B}({\rm Con}(A))$. By Proposition \ref{boolcp}, if ${\cal V}$ is semi--degenerate, then all complemented congruences of $A$ are compact.

\begin{remark} By Proposition \ref{boolid} and Proposition \ref{boolmorph}, (\ref{boolmorph46}), if $A$ is se\-mi\-prime and $\nabla _A/\!\!\equiv _A=\{\nabla _A\}$, in particular if $A$ is semiprime and ${\cal V}$ is semi--degenerate, then the complemented elements of ${\rm Id}({\rm Con}(A))$ are the principal ideals of ${\rm Con}(A)$ generated by complemented congruences of $A$, ${\cal B}({\rm Con}(A))/\!\!\equiv _A={\cal B}({\rm Con}(A)/\!\!\equiv _A)$, that is $\equiv _A$ has the BLP, and, for all $\theta \in {\rm Con}(A)$, we have: $\theta \in {\cal B}({\rm Con}(A))$ iff $\theta /\!\!\equiv _A\in {\cal B}({\rm Con}(A)/\!\!\equiv _A)$.\label{boolidcg}\end{remark}

By Proposition \ref{boolmorph}, if ${\cal V}$ is semi--degenerate or congruence--distributive, then\linebreak $p_{\equiv _A}\mid _{{\cal B}({\rm Con}(A))}:{\cal B}({\rm Con}(A))\rightarrow {\cal B}({\rm Con}(A)/\!\!\equiv _A)$ is an injective Boolean morphism, which is an isomorphism if $A$ is semiprime and $\nabla _A/\!\!\equiv _A=\{\nabla _A\}$, in particular if ${\cal V}$ is both semi--degenerate and congruence--distributive.

If $A$ is semiprime, $\nabla _A/\!\! \equiv _A\ =\{\nabla _A\}$ and $[\alpha ,\nabla _A]_A=\alpha $ for all $\alpha \in {\rm Con}(A)$, in particular if $A$ is semiprime and ${\cal V}$ is semi--degenerate or, for instance, ${\cal V}$ is congruence--distributive and ${\rm Con}(A)$ is compact, then ${\rm Con}(A)$ satisfies the properties from Proposition \ref{boollambda}, in particular all complemented congruences of $A$ are radical.

\begin{corollary} Let $A$ be a member of a congruence--modular variety ${\cal V}$, $\equiv _A=\{(\theta ,\zeta )\in {\rm Con}(A)^2\ |\ \rho _A(\theta )=\rho _A(\zeta )\}$ and $\kappa $ be an arbitrary cardinality.

Then $(5)_{\kappa ,{\rm Con}(A)}$ implies $(5)_{\kappa ,{\rm Con}(A)/\equiv _A}$.

If $A$ is semiprime, in particular if the commutator of $A$ equals the intersection, in particular if ${\cal V}$ is congruence--distributive, then:\begin{itemize}
\item if ${\rm Con}(A)$ is Stone (equivalently, strongly Stone), then ${\rm Con}(A)/\!\!\equiv _A$ is Stone (equivalently, strongly Stone); in particular, if $\kappa $ is nonzero, then, for any cardinality $\lambda $, $(1)_{\kappa ,{\rm Con}(A)}$ implies $(1)_{\lambda ,{\rm Con}(A)/\equiv _A}$;
\item $(3)_{\kappa ,{\rm Con}(A)}$ is equivalent to $(3)_{\kappa ,{\rm Con}(A)/\equiv _A}$;
\item $(iv)_{{\rm Con}(A)}$, $(4)_{\kappa ,{\rm Con}(A)}$ and $(4)_{\kappa ,{\rm Con}(A)/\equiv _A}$ are equivalent.\end{itemize}

If $A$ is semiprime and $\{\theta \in {\rm Con}(A)\ |\ \rho _A(\theta )=\nabla _A\}=\{\nabla _A\}$, in particular if $A$ is semiprime, $\nabla _A\in {\rm Cp}({\rm Con}(A))$ and $[\nabla _A,\nabla _A]_A=\nabla _A$, in particular if $A$ is semiprime and ${\cal V}$ is semi--degenerate, in particular if ${\cal V}$ is congruence--distributive and semi--degenerate, then:\begin{itemize}
\item $(2)_{\kappa ,{\rm Con}(A)}$ implies $(2)_{\kappa ,{\rm Con}(A)/\equiv _A}$;
\item $(5)_{\kappa ,{\rm Con}(A)}$ is equivalent to $(5)_{\kappa ,{\rm Con}(A)/\equiv _A}$.\end{itemize}

If $A$ is semiprime, $\{\theta \in {\rm Con}(A)\ |\ \rho _A(\theta )=\nabla _A\}=\{\nabla _A\}$ and $[\alpha ,\nabla _A]_A=\alpha $ for all $\alpha \in {\rm Con}(A)$, in particular if $A$ is semiprime and ${\cal V}$ is semi--degenerate or ${\cal V}$ is congruence--distributive and $\nabla _A\in {\rm Cp}({\rm Con}(A))$, in particular if ${\cal V}$ is congruence--distributive and semi--degenerate, then:\begin{itemize}
\item ${\rm Con}(A)$ is Stone iff ${\rm Con}(A)$ is strongly Stone iff ${\rm Con}(A)/\!\!\equiv _A$ is Stone iff ${\rm Con}(A)/\!\!\equiv _A$ is strongly Stone;
\item for each $i\in \overline{1,5}$, $(i)_{\kappa ,{\rm Con}(A)}$ is equivalent to $(i)_{\kappa ,{\rm Con}(A)/\equiv _A}$.
\end{itemize}\label{cglatdavey}\end{corollary}

\begin{proof} By Propositions \ref{echiv(i)}, \ref{echiv(ii)}, \ref{echiv(iii)}, \ref{echiv(iv)} and \ref{echiv(v)} and Corollary \ref{echivstone}.\end{proof}

\begin{corollary} Let $A$ be a member of a congruence--modular variety ${\cal V}$. If $A$ is se\-mi\-prime, $\{\theta \in {\rm Con}(A)\ |\ \rho _A(\theta )=\nabla _A\}=\{\nabla _A\}$ and $[\alpha ,\nabla _A]_A=\alpha $ for all $\alpha \in {\rm Con}(A)$, in particular if $A$ is semiprime and ${\cal V}$ is semi--degenerate or ${\cal V}$ is congruence--distributive and $\nabla _A\in {\rm Cp}({\rm Con}(A))$, in particular if ${\cal V}$ is congruence--distributive and semi--degenerate, then, for any $i,j\in \overline{1,5}$ and any nonzero cardinalities $\kappa $ and $\mu $, conditions $(iv)_{{\rm Con}(A)}$, $(i)_{\kappa ,{\rm Con}(A)}$ and $(j)_{\mu ,{\rm Con}(A)}$ are equivalent, and thus each of them is equivalent to\linebreak $(i)_{<\infty ,{\rm Con}(A)}$ and to $(i)_{{\rm Con}(A)}$ and ${\rm Con}(A)$ satisfies the equivalences from Corollary \ref{clearer}, (\ref{clearer1}).\label{transfercglat}\end{corollary}

\begin{corollary} Let $A$ be a member of a congruence--modular variety ${\cal V}$. If $A$ is se\-mi\-prime, in particular if the commutator of $A$ equals the intersection, in particular if ${\cal V}$ is con\-gru\-ence--distributive, then:\begin{itemize}
\item for any nonzero cardinalities $\kappa $ and $\mu $, conditions $(3)_{\kappa ,{\rm Con}(A)}$, $(4)_{\mu ,{\rm Con}(A)}$ and\linebreak $(iv)_{{\rm Con}(A)}$ are equivalent;
\item if $\{\theta \in {\rm Con}(A)\ |\ \rho _A(\theta )=\nabla _A\}=\{\nabla _A\}$, in particular if $\nabla _A\in {\rm Cp}({\rm Con}(A))$ and $[\nabla _A,\nabla _A]_A=\nabla _A$, then, for any nonzero cardinalities $\kappa $, $\lambda $ and $\mu $, conditions $(3)_{\kappa ,{\rm Con}(A)}$, $(4)_{\lambda ,{\rm Con}(A)}$, $(iv)_{{\rm Con}(A)}$ and $(5)_{\mu ,{\rm Con}(A)}$ are equivalent.\end{itemize}\end{corollary}

\begin{proof} By Corollary \ref{3eq4}.\end{proof}

\begin{remark} Note that, for any nonempty family $(L_i)_{i\in I}$ of bounded lattices, $\displaystyle {\cal B}(\prod _{i\in I}L_i)\linebreak =\prod _{i\in I}{\cal B}(L_i)$ and, for all $\displaystyle (a_i)_{i\in I}\in \prod _{i\in I}L_i$, $\displaystyle {\rm Ann}_{\prod _{i\in I}L_i}((a_i)_{i\in I})=\prod _{i\in I}{\rm Ann}_{L_i}(a_i)$ and\linebreak  $\displaystyle ((a_i)_{i\in I}]_{\prod _{i\in I}L_i}=\prod _{i\in I}(a_i]_{L_i}$. Moreover, if $\displaystyle pr_j:\prod _{i\in I}L_i\rightarrow L_j$ is the canonical projection for each $j\in I$, then, for all $\displaystyle U\subseteq \prod _{i\in I}L_i$, $\displaystyle {\rm Ann}_{\prod _{i\in I}L_i}(U)=\prod _{i\in I}{\rm Ann}_{L_i}(pr_i(U))$ and $\displaystyle (U]_{\prod _{i\in I}L_i}=\prod _{i\in I}(pr_i(U)]_{L_i}$. Hence, for any cardinality $\kappa $ and each $h\in \overline{1,5}$, $(h)_{\kappa ,\prod _{i\in I}L_i}$ is satisfied iff $(h)_{\kappa ,L_i}$ is satisfied for all $i\in I$.\label{presprod}\end{remark}

\begin{corollary} Let $A$ and $B$ be members of a congruence--modular variety ${\cal V}$ and $\kappa $ be a cardinality.\begin{itemize}
\item If the direct product $A\times B$ has no skew congruences, in particular if ${\cal V}$ is semi--degenerate, then: $A\times B$ is semiprime iff $A$ and $B$ are semiprime.
\item If the direct product $A\times B$ has no skew congruences and $A$ and $B$ are semiprime, then: ${\rm Con}(A\times B)$ is Stone iff ${\rm Con}(A)$ and ${\rm Con}(B)$ are Stone.
\item If the direct product $A\times B$ has no skew congruences, $A$ and $B$ are semiprime and ${\cal B}({\rm Con}(A))$ and ${\cal B}({\rm Con}(B))$ are closed w.r.t. the intersection, in particular if they are (Boolean) sublattices of ${\rm Con}(A)$ and ${\rm Con}(B)$, respectively, then: ${\rm Con}(A\times B)$ is $\kappa $--Stone iff ${\rm Con}(A)$ and ${\rm Con}(B)$ are $\kappa $--Stone.
\item If ${\cal V}$ is semi--degenerate and $A$ and $B$ are semiprime, then: ${\rm Con}(A\times B)$ is $\kappa $--Stone iff ${\rm Con}(A)$ and ${\rm Con}(B)$ are $\kappa $--Stone.
\item If ${\cal V}$ is congruence--distributive, then: ${\rm Con}(A\times B)$ is $\kappa $--Stone iff ${\rm Con}(A)$ and ${\rm Con}(B)$ are $\kappa $--Stone.\end{itemize}\label{prodcg}\end{corollary}

\begin{proof} According to \cite[Theorem 5.17, p. 48]{ouwe} that, for all $\alpha ,\theta \in {\rm Con}(A)$ and all $\beta ,\zeta \in {\rm Con}(B)$, we have $[\alpha \times \beta ,\theta \times \zeta ]_{A\times B}=[\alpha ,\theta ]_A\times [\beta ,\zeta ]_B$, hence, if the direct product $A\times B$ has no skew congruences, in particular if ${\cal V}$ is congruence--distributive or semi--degenerate (\cite[Theorem 8.5, p. 85]{fremck},\cite[Lemma 5.2]{agl}), then ${\rm Spec}(A\times B)=\{\phi \times \nabla _B,\nabla _A\times \psi \ |\ \phi \in {\rm Spec}(A),\psi \in {\rm Spec}(B)\}$, so that $R({\rm Con}(A\times B))=R({\rm Con}(A))\times R({\rm Con}(B))$ and hence $A\times B$ is semiprime iff $A$ and $B$ are semiprime (see also \cite{retic}).

Since, for any bounded lattice $L$, if ${\cal B}(L)$ is closed w.r.t. the meet, in particular if ${\cal B}(L)$ is a (Boolean) sublattice of $L$, then $1_{1,L}$ is equivalent to $1_{<\infty ,L}$, and, if $\kappa $ is an infinite cardinality, then, for any $U\subseteq A$ and any $V\subseteq B$, we have: $|U|\leq \kappa $ and $|V|\leq \kappa $ iff $|U\times V|\leq \kappa $, by Remark \ref{presprod} we get the statements in the enunciation.\end{proof}

\begin{remark} Let $M$ be a bounded sublattice of a lattice $L$ and $U,V\subseteq L$. Then it is straightforward that ${\rm Ann}_L(U)\cap M\subseteq {\rm Ann}_M(U\cap M)$, $(U\cap M]_L\cap M=(U\cap M]_M$ and, if $L$ is distributive, then $(U]_L\cap (V]_L=(U\cap V]_L$.\end{remark}

\begin{lemma} Let $L$ be a bounded distributive lattice and $M$ a bounded sublattice of $L$.\begin{enumerate}
\item\label{sublat} If $U\subseteq M$ is such that ${\rm Ann}_L(U)\vee {\rm Ann}_L({\rm Ann}_L(U))=L$, then ${\rm Ann}_M(U)\vee \linebreak {\rm Ann}_M({\rm Ann}_M(U))=M$.
\item\label{(v)sublat} for any cardinality $\kappa $, $(5)_{\kappa ,L}$ implies $(5)_{\kappa ,M}$.\end{enumerate}\label{allsublat}\end{lemma}

\begin{proof} (\ref{sublat}) If $U$ is as in the hypothesis, then ${\rm Ann}_M(U)\vee {\rm Ann}_M({\rm Ann}_M(U))\supseteq ({\rm Ann}_L(U)\cap M)\vee ({\rm Ann}_M({\rm Ann}_L(U)\cap M)\supseteq ({\rm Ann}_L(U)\cap M)\vee ({\rm Ann}_L({\rm Ann}_L(U))\cap M)\supseteq ({\rm Ann}_L(U)\cap M)\cup ({\rm Ann}_L({\rm Ann}_L(U))\cap M)=({\rm Ann}_L(U)\cup {\rm Ann}_L({\rm Ann}_L(U))\cap M$, thus ${\rm Ann}_M(U)\vee {\rm Ann}_M({\rm Ann}_M(U))\supseteq ({\rm Ann}_L(U)\cup {\rm Ann}_L({\rm Ann}_L(U))]_M\cap (M]_M=({\rm Ann}_L(U)\cup {\rm Ann}_L(\linebreak {\rm Ann}_L(U))]_L\cap M\cap M=({\rm Ann}_L(U)\vee {\rm Ann}_L({\rm Ann}_L(U)))\cap M=L\cap M=M$.

\noindent (\ref{(v)sublat}) Assume that $(5)_{\kappa ,L}$ is fulfilled, and let $U\subseteq M\subseteq L$ with $|U|\leq \kappa $, so that ${\rm Ann}_L(U)\vee {\rm Ann}_L({\rm Ann}_L(U))=L$, hence ${\rm Ann}_M(U)\vee {\rm Ann}_M({\rm Ann}_M(U))=M$ by (\ref{sublat}).\end{proof}

Let us assume that the set ${\rm Cp}({\rm Con}(A))$ of the compact congruences of $A$ contains $\nabla _A$ and is closed w.r.t. the commutator of $A$. In \cite{retic}, under these hypotheses we have constructed the {\em reticulation} ${\cal L}(A)$ of $A$, which, by definition, is a bounded distributive lattice whose prime spectrum of ideals (or filters, but our construction in \cite{retic} fulfills this property for ideals) is homeomorphic to the prime spectrum of congruences of $A$, w.r.t. the Stone topologies. ${\cal L}(A)$ is unique modulo a lattice isomorphism and, by our construction from \cite{retic}: ${\cal L}(A)={\rm Cp}({\rm Con}(A))/\!\!\equiv _A$, which is a bounded sublattice of ${\rm Con}(A)/\!\!\equiv _A$.

\begin{proposition} Let $A$ be a member of a congruence--modular variety ${\cal V}$ such that $\nabla _A\in {\rm Cp}({\rm Con}(A))$ and ${\rm Cp}({\rm Con}(A))$ is closed w.r.t. the commutator of $A$. Then, for any cardinality $\kappa $, $(5)_{\kappa ,{\rm Con}(A)}$ implies $(5)_{\kappa ,{\cal L}(A)}$.\label{cglatvsretic}\end{proposition}

\begin{proof} By Corollary \ref{cglatdavey} and Lemma \ref{allsublat}, (\ref{(v)sublat}).\end{proof}

\begin{corollary} Let $A$ be a member of a congruence--modular variety ${\cal V}$ such that $\nabla _A\in {\rm Cp}({\rm Con}(A))$ and ${\rm Cp}({\rm Con}(A))$ is closed w.r.t. the commutator of $A$. If $A$ is semiprime and ${\rm Con}(A)$ is Stone, then ${\cal L}(A)$ is strongly Stone.\end{corollary}

\begin{proof} By Proposition \ref{cglatvsretic}, Corollary \ref{cglatdavey}, the distributivity of ${\cal L}(A)$ and Theorem \ref{davey}, (\ref{davey3}).\end{proof}

\section{Transferring Davey`s Theorem to Commutative Unitary Rings}
\label{resultselem}

Let us see how we can to obtain versions of Davey's Theorem for the elements of semiprime algebras from congruence--modular varieties by transferring results such as Corollary \ref{transfercglat} from their congruence lattices. We exemplify here for semiprime commutative unitary rings.

Let $(T,\vee ,\wedge ,\odot ,\rightarrow ,0,1)$ be a {\em residuated lattice} (otherwise called a {\em commutative integral bounded lattice--ordered monoid}), which means that $(T,\vee ,\wedge ,0,1)$ is a bounded lattice, $(T,\odot ,1)$ is a commutative monoid and $\rightarrow $ is a binary operation on $T$ which fulfills the {\em law of residuation}: for all $a,b,c\in T$, $a\leq b\rightarrow c$ iff $a\odot b\leq c$. See more about residuated lattices in \cite{gal}, \cite{ior}, \cite{pic}. Residuated lattices form a semi--degenerate congruence--distributive variety, hence they are semiprime and thus their congruence lattices satisfy Theorem \ref{davey}, (\ref{davey1}), and even the equivalences from Corollary \ref{transfercglat}. But they also fulfill a theorem of this form for elements, which can be expressed in the following way, since we notice that the bounded lattice of the filters of $T$ is a bounded sublattice of that of the filters of the underlying bounded lattice of $T$ and that, for each $e\in {\cal B}(T)$:

\begin{theorem}{\rm \cite[Theorem $5.2.6$]{eu},\cite[Theorem $3.13$]{eu7}} If $S$ is the dual of the underlying bounded lattice of a residuated lattice, then conditions $(1)_{m,S}$, $(2)_{m,S}$, $(3)_{m,S}$, $(4)_{m,S}$ and $(5)_{m,S}$ are equivalent.\label{elemreslat}\end{theorem}

In \cite{eu,eu7}, we have proven Theorem \ref{elemreslat} by transferring the dual of Theorem \ref{davey} from bounded distributive lattices to residuated lattices through the reticulation functor for residuated lattices.

\begin{remark} Note from Lemma \ref{residuation} that, if $(L,\vee ,\wedge ,[\cdot ,\cdot ],0,1)$ is a commutator lattice in which the operation $[\cdot ,\cdot ]$ is associative and satisfies $[a,1]=a$ for all $a\in L$, then $L$ is a complete residuated lattice with the residuation $\rightarrow $ defined above Remark \ref{negann}.\end{remark}

For instance, rings form a semi--degenerate congruence--modular variety with associative commutators, so that, for any commutative unitary ring $R$, $({\rm Con}(R),\vee ,\cap ,[\cdot ,\cdot ]_R,\rightarrow ,\Delta _R,\nabla _R)$ is a complete residuated lattice.

Throughout the rest of this section, unless mentioned otherwise, $(R,+,\cdot ,0,1)$ will be a commutative unitary ring.

We denote by $({\rm Id}(R),\vee =+,\cap ,[\cdot ,\cdot ]=\cdot ,\{0\},R)$ the commutator lattice of the ideals of $R$ and by $\iota\gamma _R:{\rm Id}(R)\rightarrow {\rm Con}(R)$ the canonical lattice isomorphism: for all $I\in {\rm Id}(R)$, $\iota\gamma _R(I)=\{(x,y)\in I^2\ |\ x-y\in I\}$. We denote by ${\rm Spec}_{\rm Id}(R)={\rm Spec}_{{\rm Id}(R)}$ the set of the prime ideals of $R$ (w.r.t. to the commutator operation given by the multiplication of ideals). Recall that $\iota\gamma _R$ preserves the commutator operation, that is $[\iota\gamma _R(I),\iota\gamma _R(J)]_R=\iota\gamma _R(I\cdot J)$ for all $I,J\in {\rm Id}(R)$, from which it is easy to deduce that $\iota\gamma _R({\rm Spec}_{\rm Id}(R))={\rm Spec}(R)$ and thus $\iota\gamma _R(R({\rm Id}(R)))=R({\rm Con}(R))$. If we denote, for each $I\in {\rm Id}(R)$, by $\sqrt{I}=\bigcap \{P\in {\rm Spec}_{\rm Id}(R)\ |\ I\subseteq P\}$ the {\em radical} of $I$, then note that $R$ is semiprime iff $\{0\}\in R({\rm Id}(R))$ iff $\sqrt{\{0\}}=\{0\}$.

For every $U\subseteq R$, $\langle U\rangle _R$ shall be the ideal of $R$ generated by $U$, so, for each $x\in R$, $\langle \{x\}\rangle _R=xR$. Let ${\rm PId}(R)$ be the set of the principal ideals of $R$ and note that ${\rm Cp}({\rm Id}(R))$ is the set of the finitely generated ideals of $R$. It is straightforward that, for all $x,a,b\in R$, $\iota\gamma _R(xR)=Cg_R(x,0)$ and $Cg_R(a,b)=Cg_R(a-b,0)$, hence $\iota\gamma _R({\rm PId}(R))={\rm PCon}(R)$ and thus $\iota\gamma _R({\rm Cp}({\rm Id}(R)))={\rm Cp}({\rm Con}(R))$. Notice that, for any $k,n\in \N ^*$ and any $x_1,\ldots ,x_k,y_1,\ldots ,y_n\in R$, $\langle \{x_1,\ldots ,x_k\}\rangle _R\cdot \langle \{y_1,\ldots ,y_k\}\rangle _R=(x_1R+\ldots +x_kR)\cdot (y_1R+\ldots +y_nR)=\langle \{x_iy_j\ |\ i\in \overline{1,k},j\in \overline{1,n}\}\rangle _R$, so ${\rm Cp}({\rm Id}(R))$ is closed w.r.t. $\cdot $, thus ${\rm Cp}({\rm Con}(R))$ is closed w.r.t. $[\cdot ,\cdot ]_R$. Let $R^*$ be the reticulation of $R$, as constructed in \cite{bell,bell2} (see also \cite{joy,sim}): $R^*={\rm Id}(R)/\!\!\sim _R$, where $\sim _R$ is the complete lattice congruence of ${\rm Id}(R)$ defined by: $\sim _R=\{(I,J)\in ({\rm Id}(R))^2\ |\ \sqrt{I}=\sqrt{J}\}$ (see also Proposition \ref{cpltcg}); by the above, $R^*$ is isomorphic to ${\rm Con}(R)/\!\!\equiv _R$. Regarding the results from \cite{bell} we are using, note that, since $R$ is commutative, it follows that $R$ is quasicommutative, thus, by \cite[Theorem $3$]{bell}, $R$ fulfills condition $(*)$ from \cite{bell}.

\begin{remark} By \cite[Lemma, p. 1861]{bell}, for all $I\in {\rm Id}(R)$, there exists a $K\in {\rm Cp}({\rm Id}(R))$ such that $K\subseteq I$ and $\sqrt{K}=\sqrt{I}$, hence $R^*={\rm Id}(R)/\!\!\sim _R={\rm Cp}({\rm Id}(R))/\!\!\sim _R$, therefore ${\rm Con}(R)/\!\!\equiv _R\ ={\rm Cp}({\rm Con}(R))/\!\!\equiv _R\ ={\cal L}(R)$, thus, as expected by the uniqueness of the reticulation, $R^*$ is isomorphic to ${\cal L}(R)$.\label{hereswhy}\end{remark}

The fact that the variety of commutative unitary rings is semi--degenerate and congruence--modular and Corollary \ref{transfercglat}, along with the fact that the lattices ${\rm Con}(R)$ and ${\rm Id}(R)$ are isomorphic, give us:

\begin{corollary} If $R$ is a semiprime commutative unitary ring, then, for any $i,j\in \overline{1,5}$ and any nonzero cardinalities $\kappa $ and $\mu $, conditions $(iv)_{{\rm Id}(R)}$, $(i)_{\kappa ,{\rm Id}(R)}$ and $(j)_{\mu ,{\rm Id}(R)}$ are equivalent, in particular ${\rm Id}(R)$ is a Stone lattice iff it is a strongly Stone lattice.\label{idrdavey}\end{corollary}

Let us see that, similarly to what happens in residuated lattices, commutative unitary rings also fulfill an analogue of Davey`s Theorem for elements instead of congruences.

Let $\kappa $ be an arbitrary cardinality.

For any $a\in R$ and any $U\subseteq R$, ${\rm Ann}_R(a)$ and ${\rm Ann}_R(U)$ will denote the {\em annihilator} of $a$ and that of $U$, respectively: ${\rm Ann}_R(a)=\{x\in R\ |\ xa=0\}$ and $\displaystyle {\rm Ann}_R(U)=\bigcap _{u\in U}{\rm Ann}_R(u)$. As in the case of bounded lattices, let us denote by ${\cal A}nn(R)=\{{\rm Ann}_R(U)\ |\ U\subseteq R\}$, ${\cal A}nn_{<\infty }(R)=\{{\rm Ann}_R(U)\ |\ U\subseteq R,|U|<\aleph _0\}$, ${\cal A}nn_{\kappa }(R)=\{{\rm Ann}_R(U)\ |\ U\subseteq R,|U|\leq \kappa \}$, ${\rm PAnn}(R)=\{{\rm Ann}_R(a)\ |\ a\in R\}={\cal A}nn_1(R)$, ${\rm 2Ann}(R)=\{{\rm Ann}_R({\rm Ann}_R(U))\ |\ U\subseteq R\}$, ${\rm 2Ann}_{<\infty }(R)=\{{\rm Ann}_R({\rm Ann}_R(U))\ |\ U\subseteq R,|U|<\aleph _0\}$, ${\rm 2Ann}_{\kappa }(R)=\{{\rm Ann}_R({\rm Ann}_R(U))\ |\linebreak U\subseteq R,|U|\leq \kappa \}$ and ${\rm P2Ann}(R)=\{{\rm Ann}_R({\rm Ann})_R(a))\ |\ a\in R\}={\rm 2Ann}_1(R)$. It is well known and straightforward that ${\cal A}nn(R)\subseteq {\rm Id}(R)$.

$E(R)$ will denote the set of the idempotent elements of $R$. Recall that $(E(R),\vee ,\wedge =\cdot ,\neg \, ,0,1)$ is a Boolean algebra, where, for every $e,f\in E(R)$, $\neg \, e=1-e$ and $e\vee f=\neg \, (\neg \, e\wedge \neg \, f)=1-(1-e)\cdot (1-f)$.

$R$ is called a {\em Baer ring} iff, for any $a\in R$, there exists an $e\in E(R)$ such that ${\rm Ann}_R(a)=eR$. By analogy to the case of bounded lattices, we shall call $R$ a {\em strongly Baer ring}, respectively a {\em $\kappa $--Baer ring} iff, for any $U\subseteq R$, respectively any $U\subseteq R$ with $|U|\leq \kappa $, there exists an $e\in E(R)$ such that ${\rm Ann}_R(U)=eR$.

Let us consider the following conditions on $R$, where $\kappa $ is an arbitrary cardinality:

\begin{flushleft}\begin{tabular}{ll}
$(1^{\circ })_{\kappa ,R}$ & $R$ is a $\kappa $--Baer ring;\\ 
$(1^{\circ })_{<\infty ,R}$ & ${\cal A}nn_{<\infty }(R)\subseteq \{eR\ |\ e\in E(R)\}$;\\ 
$(1^{\circ })_R$ & $R$ is a strongly Baer ring;\end{tabular}

\begin{tabular}{ll}
$(2^{\circ })_{\kappa ,R}$ & $R$ is a Baer ring and $E(R)$ is a $\kappa $--complete Boolean algebra;\\ 
$(2^{\circ })_{<\infty ,R}$ & $R$ is a Baer ring and $E(R)$ is a Boolean algebra;\\ 
$(2^{\circ })_R$ & $R$ is a Baer ring and $E(R)$ is a complete Boolean algebra;\end{tabular}

\begin{tabular}{ll}
$(3^{\circ })_{\kappa ,R}$ & ${\rm 2Ann}(R)$ is a $\kappa $--complete Boolean sublattice of ${\rm Id}(R)$ such that\\ 
& $I\mapsto {\rm Ann}_R({\rm Ann}_R(I))$ is a lattice morphism from ${\rm Id}(R)$ to ${\rm 2Ann}(R)$;\\ 
$(3^{\circ })_{<\infty ,R}$ & ${\rm 2Ann}(R)$ is a Boolean sublattice of ${\rm Id}(R)$ such that\\ 
& $I\mapsto {\rm Ann}_R({\rm Ann}_R(I))$ is a lattice morphism from ${\rm Id}(R)$ to ${\rm 2Ann}(R)$;\\ 
$(3^{\circ })_R$ & ${\rm 2Ann}(R)$ is a complete Boolean sublattice of ${\rm Id}(R)$ such that\\ 
& $I\mapsto {\rm Ann}_R({\rm Ann}_R(I))$ is a lattice morphism from ${\rm Id}(R)$ to ${\rm 2Ann}(R)$;\end{tabular}

\begin{tabular}{ll}
$(4^{\circ })_{\kappa ,R}$ & for all $I,J\in {\rm Id}(R)$, ${\rm Ann}_R(I\cap J)={\rm Ann}_R(I)\vee {\rm Ann}_R(J)$, and\\ 
& ${\rm 2Ann}_{\kappa }(R)\subseteq {\cal A}nn_{<\infty }(R)$;\\ 
$(4^{\circ })_{<\infty ,R}$ & for all $I,J\in {\rm Id}(R)$, ${\rm Ann}_R(I\cap J)={\rm Ann}_R(I)\vee {\rm Ann}_R(J)$, and\\ 
& ${\rm 2Ann}_{<\infty }(R)\subseteq {\cal A}nn_{<\infty }(R)$;\\ 
$(4^{\circ })_R$ & for all $I,J\in {\rm Id}(R)$, ${\rm Ann}_R(I\cap J)={\rm Ann}_R(I)\vee {\rm Ann}_R(J)$, and\\ 
& ${\rm 2Ann}(R)\subseteq {\cal A}nn_{<\infty }(R)$;\\ 
$(iv^{\circ })_R$ & for all $I,J\in {\rm Id}(R)$, ${\rm Ann}_R(I\cap J)={\rm Ann}_R(I)\vee {\rm Ann}_R(J)$;\end{tabular}

\begin{tabular}{ll}
$(5^{\circ })_{\kappa ,R}$ & for each $U\subseteq R$ with $|U|\leq \kappa $, ${\rm Ann}_R(U)\vee {\rm Ann}_R({\rm Ann}_R(U))=R$;\\ 
$(5^{\circ })_{<\infty ,R}$ & for each finite $U\subseteq R$, ${\rm Ann}_R(U)\vee {\rm Ann}_R({\rm Ann}_R(U))=R$;\\ 
$(5^{\circ })_R$ & for each $U\subseteq R$, ${\rm Ann}_R(U)\vee {\rm Ann}_R({\rm Ann}_R(U))=R$.\end{tabular}\end{flushleft}

Obviously, conditions $(iv^{\circ })_R$, $(h)_{\kappa ,R}$, $(i)_{<\infty ,R}$ and $(j)_R$ satisfy the properties stated after Remark \ref{frameannpid} for conditions $(iv^{\circ })_L$, $(h)_{\kappa ,L}$, $(i)_{<\infty ,L}$ and $(j)_L$, where $h,i,j\in \overline{1,5}$.

\begin{remark}If $n\in \N ^*$, $u_1,\ldots ,u_n\in R$ and, for each $i\in \overline{1,n}$, ${\rm Ann}_R(u_i)=e_iR$ for some $e_i\in E(R)$, then $\displaystyle {\rm Ann}_R(\{u_1,\ldots ,\displaystyle u_n\})=\bigcap _{i=1}^n{\rm Ann}_R(u_i)=\bigcap _{i=1}^ne_iR=e_1R\cdot \ldots \cdot e_nR=e_1\ldots e_nR$, with $e_1\ldots e_n=e_1\wedge \ldots \wedge e_n\in E(R)$, hence $(1^{\circ })_{1,R}$ implies $(1^{\circ })_{<\infty ,R}$. Therefore $(1^{\circ })_{1,R}$ is equivalent to $(1^{\circ })_{<\infty ,R}$, that is $R$ is a Baer ring iff $R$ satisfies $(1^{\circ })_{<\infty ,R}$.

Hence, if ${\cal A}nn(R)={\cal A}nn_{<\infty }(R)$, so that $(1^{\circ })_{<\infty ,R}$ is equivalent to $(1^{\circ })_R$, then $(1^{\circ })_{1,R}$ is equivalent to $(1^{\circ })_R$, that is $R$ is a Baer ring iff $R$ is a strongly Baer ring.\label{finitelybaer}\end{remark}

\begin{remark} Proposition \ref{boolcenter} and the fact that the lattices ${\rm Con}(R)$ and ${\rm Id}(R)$ are isomorphic ensure us that ${\cal B}({\rm Id}(R))$ is a Boolean sublattice of ${\rm Id}(R)$.

If $R$ is semiprime, then, by {\rm \cite[Lemma, p. 1863]{bell}}, the map $e\mapsto eR/\!\!\sim _R$ from $E(R)$ to ${\cal B}(R^*)$ is a Boolean isomorphism, so, by Proposition \ref{boolmorph}, (\ref{boolmorph46}), it follows that the map $e\mapsto eR$ from $E(R)$ to ${\cal B}({\rm Id}(R))$ is a Boolean isomorphism.\label{boolrstar}\end{remark}

\begin{lemma} If $R$ is a semiprime commutative unitary ring, then:\begin{enumerate}
\item\label{bidrann1} ${\cal B}({\rm Id}(R))=\{eR\ |\ e\in E(R)\}$ and the map $e\mapsto eR$ from $E(R)$ to ${\cal B}({\rm Id}(R))$ is a Boolean isomorphism;
\item\label{bidrann2} if $U\subseteq R$, then $U\cap {\rm Ann}_R(U)\subseteq \{0\}$; if $I\in {\rm Id}(R)$, then $I\cap {\rm Ann}_R(I)=\{0\}$;
\item\label{bidrann4} if $U\subseteq R$ such that ${\rm Ann}_R(U)\vee {\rm Ann}_R({\rm Ann}_R(U))=R$, then ${\rm Ann}_R(U)=eR$ for some $e\in E(R)$.\end{enumerate}\label{bidrann}\end{lemma}

\begin{proof} (\ref{bidrann1}) By Remark \ref{boolrstar}.

\noindent (\ref{bidrann2}) If $U\subseteq R$ and $x\in U\cap {\rm Ann}_R(U)$, then $x\cdot x=0$, so that $x=0$ since a semiprime commutative unitary ring has no nonzero nilpotents \cite[p.125,126]{kist}. Now, if $I\in {\rm Id}(R)$, then $0\in I\cap {\rm Ann}_R(I)$.

\noindent (\ref{bidrann4}) By (\ref{bidrann2}), it follows that ${\rm Ann}_R(U)\in {\cal B}({\rm Id}(R))$, having ${\rm Ann}_R({\rm Ann}_R(U))$ as a complement, so that ${\rm Ann}_R(U)=eR$ for some $e\in E(R)$ by (\ref{bidrann1}).\end{proof}

\begin{lemma} If $R$ is a commutative unitary ring, then:\begin{itemize}
\item for any $U\subseteq R$, ${\rm Ann}_R(U)={\rm Ann}_R(\langle U\rangle _R)$;
\item if all ideals of $R$ are finitely generated, then ${\rm 2Ann}(R)={\rm 2Ann}_{<\infty }(R)\subseteq {\cal A}nn(R)={\cal A}nn_{<\infty }(R)$;
\item for any $I\in {\rm Id}(R)$, ${\rm Ann}_{{\rm Id}(R)}(I)=({\rm Ann}_R(I)]_{{\rm Id}(R)}$ and ${\rm Ann}_{{\rm Id}(R)}({\rm Ann}_{{\rm Id}(R)}(I))=({\rm Ann}_R({\rm Ann}_R(I))]_{{\rm Id}(R)}$.
\end{itemize}\label{annr}\end{lemma}

\begin{proof} Let $U\subseteq R$, arbitrary. Since $U\subseteq \langle U\rangle _R$, we have ${\rm Ann}_R(\langle U\rangle _R)\subseteq {\rm Ann}_R(U)$. The converse inclusion holds, as well, since, given any $a\in \langle U\rangle _R$ and any $x\in {\rm Ann}_R(U)$, we have $a=a_1\cdot u_1+\ldots +a_n\cdot u_n$ for some $n\in \N ^*$, $a_1,\ldots ,a_n\in R$ and $u_1,\ldots ,u_n\in U$, so that $x\cdot u_1=\ldots =x\cdot u_n=0$, therefore $x\cdot a=0$, so $x\in {\rm Ann}_R(\langle U\rangle _R)$.

Thus, in the particular case when all ideals of $R$ are finitely generated, so that there exists a finite $F\subseteq R$ such that $\langle U\rangle _R=\langle F\rangle _R$, then ${\rm Ann}_R(U)={\rm Ann}_R(\langle U\rangle _R)={\rm Ann}_R(\langle F\rangle _R)={\rm Ann}_R(F)$, hence ${\cal A}nn(R)={\cal A}nn_{<\infty }(R)$.

Let $J\in {\rm Id}(R)$. Then: $J\in ({\rm Ann}_R(I)]_{{\rm Id}(R)}$ iff $J\subseteq {\rm Ann}_R(I)$ iff $x\in {\rm Ann}_R(I)$ for all $x\in J$ iff $x\cdot y=0$ for all $x\in J$ and all $y\in I$ iff $J\cdot I=\{0\}$ iff $J\in {\rm Ann}_{{\rm Id}(R)}(I)$. Hence ${\rm Ann}_{{\rm Id}(R)}(I)=({\rm Ann}_R(I)]_{{\rm Id}(R)}$, therefore ${\rm Ann}_{{\rm Id}(R)}({\rm Ann}_{{\rm Id}(R)}(I))={\rm Ann}_{{\rm Id}(R)}(({\rm Ann}_R(I)]_{{\rm Id}(R)})={\rm Ann}_{{\rm Id}(R)}({\rm Ann}_R(I))=({\rm Ann}_R({\rm Ann}_R(I))]_{{\rm Id}(R)}$.\end{proof}

\begin{lemma}If $R$ is a semiprime commutative unitary ring, then, for any $U\subseteq R$, there exists a finite subset $S\subseteq \langle U\rangle _R$ such that ${\rm Ann}_R(U)={\rm Ann}_R(S)$, so ${\rm 2Ann}(R)={\rm 2Ann}_{<\infty }(R)\subseteq {\cal A}nn(R)={\cal A}nn_{<\infty }(R)$.\label{annfin}\end{lemma}

\begin{proof} By Remark \ref{hereswhy} and Lemmas \ref{annr} and \ref{anntheta}, for an appropriate finite subset $S\subseteq \langle U\rangle _R$, we have $\langle U\rangle _R/_{\textstyle \sim _R}=\langle S\rangle _R/_{\textstyle \sim _R}$, thus ${\rm Ann}_{{\rm Id}(R)}(\langle U\rangle _R)/\!\!\sim _R={\rm Ann}_{R^*}(\langle U\rangle _R/\!\!\sim _R)={\rm Ann}_{R^*}(\langle S\rangle _R/_{\textstyle \sim _R})={\rm Ann}_{{\rm Id}(R)}(\langle S\rangle _R)/\!\!\sim _R$, hence $({\rm Ann}_R(U)]_{{\rm Id}(R)}=({\rm Ann}_R(\langle U\rangle _R)]_{{\rm Id}(R)}={\rm Ann}_{{\rm Id}(R)}(\langle U\rangle _R)={\rm Ann}_{{\rm Id}(R)}(\langle S\rangle _R)=({\rm Ann}_R(\langle S\rangle _R)]_{{\rm Id}(R)}=({\rm Ann}_R(S)]_{{\rm Id}(R)}$, thus\linebreak ${\rm Ann}_R(U)={\rm Ann}_R(S)$.\end{proof}

\begin{proposition} Let $R$ be a commutative unitary ring.\begin{enumerate}
\item\label{1r1} If all ideals of $R$ are finitely generated, then $R$ is a Baer ring iff $R$ is a strongly Baer ring.
\item\label{1r2} If $R$ is semiprime, then: $R$ is a Baer ring iff $R$ is a strongly Baer ring iff ${\rm Id}(R)$ is a Stone lattice iff ${\rm Id}(R)$ is a strongly Stone lattice.\end{enumerate}\label{1r}\end{proposition}

\begin{proof} (\ref{1r1}) By Remark \ref{finitelybaer} and Lemma \ref{annr}.

\noindent (\ref{1r2}) By Remark \ref{finitelybaer} and Lemma \ref{annfin}, $R$ is Baer iff $R$ is strongly Baer.

For any $U\subseteq R$, we have $({\rm Ann}_R(U)]_{{\rm Id}(R)}=({\rm Ann}_R(\langle U\rangle _R)]_{{\rm Id}(R)}={\rm Ann}_{{\rm Id}(R)}(\langle U\rangle _R)$ by Lemma \ref{annr}, so that, for any $e\in R$, ${\rm Ann}_R(U)=eR$ iff ${\rm Ann}_{{\rm Id}(R)}(\langle U\rangle _R)=(eR]_{{\rm Id}(R)}$. According to Lemma \ref{bidrann}, (\ref{bidrann1}), $e\in E(R)$ iff $eR\in {\cal B}({\rm Id}(R))$. Hence $(1)_{1,{\rm Id}(R)}$ is equivalent to $(1^{\circ })_R$, that is ${\rm Id}(R)$ is a Stone lattice iff $R$ is a strongly Baer ring.

Finally, by Corollary \ref{idrdavey}, ${\rm Id}(R)$ is a Stone lattice iff ${\rm Id}(R)$ is a strongly Stone lattice.\end{proof}

See also \cite[Theorem $8$]{bell} and \cite[Theorem $2.6$]{sim}, according to which, if $R$ is semiprime, then $R$ is a Baer ring iff $R^*$ is a Stone lattice, which, by Corollary \ref{cglatdavey} and the fact that the lattices ${\rm Con}(R)$ and ${\rm Id}(R)$ are isomorphic, is equivalent to ${\rm Id}(R)$ being a Stone lattice.

\begin{proposition}For any semiprime commutative unitary ring  $R$ and any cardinality $\kappa $, conditions $(2^{\circ })_{\kappa ,R}$ and $(2)_{\kappa ,{\rm Id}(R)}$ are equivalent.\label{2r}\end{proposition}

\begin{proof} By Lemma \ref{bidrann}, (\ref{bidrann1}), and Proposition \ref{1r}, (\ref{1r2}).\end{proof}

For the next lemma, recall that ${\rm 2Ann}(R)\subseteq {\cal A}nn(R)\subseteq {\rm Id}(R)$,\linebreak ${\rm PAnn}({\rm Id}(R))\subseteq {\rm Id}({\rm Id}(R))$ and ${\rm P2Ann}({\rm Id}(R))\subseteq {\rm Id}({\rm Id}(R))$, and we will be referring to these sets of annihilators as subposets of ${\rm Id}(R)$, respectively ${\rm Id}({\rm Id}(R))$.

\begin{lemma} For any commutative unitary ring $R$, the map $x\mapsto (x]_{{\rm Id}(R)}$ from ${\cal A}nn(R)$ to ${\rm PAnn}({\rm Id}(R))$, as well as from ${\rm 2Ann}(R)$ to ${\rm P2Ann}({\rm Id}(R))$, is an order isomorphism.\label{annrannidr}\end{lemma}

\begin{proof} By Lemma \ref{annr}, ${\cal A}nn(R)=\{{\rm Ann}_R(I)\ |\ I\in {\rm Id}(R)\}$ and ${\rm 2Ann}(R)=\{{\rm Ann}_R(\linebreak {\rm Ann}_R(I))\ |\ I\in {\rm Id}(R)\}$, hence these maps are completely defined. By the same lemma, these maps are well defined and surjective. Clearly, they are injective, thus bijective, and both these maps and their inverses are order--preserving.\end{proof}

\begin{proposition} For any commutative unitary ring $R$ and any cardinality $\kappa $, the properties $(3^{\circ })_{\kappa ,R}$ and $(3)_{\kappa ,{\rm Id}(R)}$ are equivalent.\label{3r}\end{proposition}

\begin{proof} By Lemma \ref{annrannidr} and the fact that, by Lemma \ref{annr}, the map $x\mapsto (x]_{{\rm Id}(R)}$ from ${\rm 2Ann}(R)\subseteq {\rm Id}(R)$ to ${\rm P2Ann}({\rm Id}(R))$ composed with the map $I\mapsto {\rm Ann}_R({\rm Ann}_R(I))$ from ${\rm Id}(R)$ to ${\rm 2Ann}(R)$ equals the map $I\mapsto {\rm Ann}_{{\rm Id}(R)}({\rm Ann}_{{\rm Id}(R)}(I))$ from ${\rm Id}(R)$ to ${\rm P2Ann}({\rm Id}(R))\subseteq {\rm Id}({\rm Id}(R))$.\end{proof}

\begin{proposition} Let $R$ be a commutative unitary ring.  Then:\begin{enumerate}
\item\label{ivr} $(iv^{\circ })_R$ is equivalent to $(iv)_{{\rm Id}(R)}$;
\item\label{4r1} if $R$ has all ideals finitely generated, then conditions $(iv)_{{\rm Id}(R)}$, $(iv^{\circ })_R$ and $(4^{\circ })_R$ are equivalent;
\item\label{4r0} if $R$ is semiprime, then conditions $(iv)_{{\rm Id}(R)}$, $(4)_{{\rm Id}(R)}$, $(iv^{\circ })_R$ and $(4^{\circ })_R$ are equivalent.\end{enumerate}\label{4r}\end{proposition}

\begin{proof} (\ref{ivr}) Let $I,J\in {\rm Id}(R)$. By Lemma \ref{annr}, ${\rm Ann}_{{\rm Id}(R)}(I\cap J)=({\rm Ann}_R(I\cap J)]_{{\rm Id}(R)}$ and ${\rm Ann}_{{\rm Id}(R)}(I)\vee {\rm Ann}_{{\rm Id}(R)}(J)=({\rm Ann}_R(I)]_{{\rm Id}(R)}\vee ({\rm Ann}_R(J)]_{{\rm Id}(R)}=({\rm Ann}_R(I)\vee {\rm Ann}_R(J)]_{{\rm Id}(R)}$, hence: ${\rm Ann}_R(I\cap J)={\rm Ann}_R(I)\vee {\rm Ann}_R(J)$ iff ${\rm Ann}_{{\rm Id}(R)}(I\cap J)={\rm Ann}_{{\rm Id}(R)}(I)\vee {\rm Ann}_{{\rm Id}(R)}(J)$.

\noindent (\ref{4r1}),(\ref{4r0}) By Lemmas \ref{annr} and \ref{annfin}, if $R$ has all ideals principal or it is semiprime, then the second part of condition $(4^{\circ })_R$ is trivially satisfied, so that $(4^{\circ })_R$ is equivalent to $(iv^{\circ })_R$.

By (\ref{ivr}), $(iv^{\circ })_R$ is equivalent to $(iv)_{{\rm Id}(R)}$.

Finally, by Corollary \ref{idrdavey}, if $R$ is semiprime, then $(iv)_{{\rm Id}(R)}$ is equivalent to $(4)_{{\rm Id}(R)}$.\end{proof}

\begin{proposition} Let $R$ be a commutative unitary ring. Then:\begin{itemize}
\item $(5)_{1,{\rm Id}(R)}$ is equivalent to $(5^{\circ })_R$;
\item if all ideals of $R$ are finitely generated, then $(5)_{1,{\rm Id}(R)}$, $(5^{\circ })_R$ and $(5^{\circ })_{<\infty ,R}$ are equivalent;
\item if $R$ is semiprime, then $(5)_{1,{\rm Id}(R)}$, $(5)_{{\rm Id}(R)}$, $(5^{\circ })_R$ and $(5^{\circ })_{<\infty ,R}$ are equivalent.\end{itemize}\label{5r}\end{proposition}

\begin{proof} By Lemma \ref{annr}, for any $U\subseteq R$, we have ${\rm Ann}_{{\rm Id}(R)}(\langle U\rangle _R)\vee {\rm Ann}_{{\rm Id}(R)}({\rm Ann}_{{\rm Id}(R)}(\linebreak \langle U\rangle _R))=({\rm Ann}_R(\langle U\rangle _R)]_{{\rm Id}(R)}\vee ({\rm Ann}_R({\rm Ann}_R(\langle U\rangle _R))]_{{\rm Id}(R)}=({\rm Ann}_R(U)]_{{\rm Id}(R)}\vee ({\rm Ann}_R(\linebreak {\rm Ann}_R(U))]_{{\rm Id}(R)}=({\rm Ann}_R(U)\vee ({\rm Ann}_R({\rm Ann}_R(U))]_{{\rm Id}(R)}\in {\rm PId}({\rm Id}(R))$ since ${\rm Ann}_R(U)\vee ({\rm Ann}_R({\rm Ann}_R(U)\in {\rm Id}(R)$, hence ${\rm Ann}_{{\rm Id}(R)}(\langle U\rangle _R)\vee {\rm Ann}_{{\rm Id}(R)}({\rm Ann}_{{\rm Id}(R)}(\langle U\rangle _R))={\rm Id}(R)=(R]_{{\rm Id}(R)}$ iff ${\rm Ann}_R(U)\vee ({\rm Ann}_R({\rm Ann}_R(U))=R$. Therefore $(5)_{1,{\rm Id}(R)}$ is equivalent to $(5^{\circ })_R$.

Clearly, if ${\cal A}nn(R)={\cal A}nn_{<\infty }(R)$, in particular if $R$ has all ideals principal or $R$ is semiprime, then $(5^{\circ })_R$ is equivalent to $(5^{\circ })_{<\infty ,R}$.

By Corollary \ref{idrdavey}, if $R$ is semiprime, then $(5)_{1,{\rm Id}(R)}$ is equivalent to $(5)_{{\rm Id}(R)}$.\end{proof}

\begin{theorem} If $R$ is a semiprime commutative unitary ring, then, for any nonzero cardinality $\kappa $ and any $h,i,j\in \overline{1,5}$, $(iv^{\circ })_R$, $(h^{\circ })_{\kappa ,R}$, $(i^{\circ })_{<\infty ,R}$ and $(j^{\circ })_R$ are equivalent.\label{myringsdavey}\end{theorem}

\begin{proof} By Corollary \ref{idrdavey} and Propositions \ref{1r}, \ref{2r}, \ref{3r}, \ref{4r} and \ref{5r}.\end{proof}

\begin{remark} Let $S$ be a commutative unitary ring. Since the variety of commutative unitary rings is semi--degenerate and thus it has no skew congruences, it follows that ${\rm Id}(R\times S)={\rm Id}(R)\times {\rm Id}(S)$, hence, if $R$ and $S$ are semiprime, then, for any cardinality $\kappa $, the ring $R\times S$ is $\kappa $--Baer iff $R$ and $S$ are $\kappa $--Baer, according to Corollary \ref{prodcg}.\end{remark}

If we eliminate from Theorem \ref{myringsdavey} the trivial implications, along with those that immediately follow from Lemma \ref{annfin}, then we obtain the following:

\begin{corollary} If $R$ is a semiprime commutative unitary ring, then the following are equivalent:\begin{itemize}
\item $R$ is a Baer ring;
\item $R$ is a strongly Baer ring and $E(R)$ is a complete Boolean algebra;
\item ${\rm 2Ann}(R)$ is a Boolean sublattice of ${\rm Id}(R)$ such that $I\mapsto {\rm Ann}_R({\rm Ann}_R(I))$ is a lattice morphism from ${\rm Id}(R)$ to ${\rm 2Ann}(R)$;
\item ${\rm 2Ann}(R)$ is a complete Boolean sublattice of ${\rm Id}(R)$ such that $I\mapsto {\rm Ann}_R({\rm Ann}_R(I))$ is a lattice morphism from ${\rm Id}(R)$ to ${\rm 2Ann}(R)$;
\item for all $I,J\in {\rm Id}(R)$, ${\rm Ann}_R(I\cap J)={\rm Ann}_R(I)\vee {\rm Ann}_R(J)$;
\item for any $U\subseteq R$, ${\rm Ann}_R(U)\vee {\rm Ann}_R({\rm Ann}_R(U))=R$.\end{itemize}\end{corollary}

Propositions \ref{1r}, \ref{2r}, \ref{3r}, \ref{4r} and \ref{5r}, along with Theorem \ref{davey}, (\ref{davey0}), also give us:

\begin{corollary} Let $R$ be a commutative unitary ring and $m$ be a nonzero cardinality such that the intersection in ${\rm Id}(R)$ is distributive w.r.t. the joins of families of cardinality at most $m$. Then, for any nonzero cardinalities $\kappa \leq m$, $\lambda \leq m$ and $\mu \leq m$ and any infinite cardinality $\iota \leq m$:\begin{itemize}
\item conditions $(2^{\circ })_{\kappa ,R}$, $(3^{\circ })_{\lambda ,R}$ and $(5^{\circ })_R$ are equivalent;
\item if all ideals of $R$ are finitely generated, then conditions $(2^{\circ })_{\kappa ,R}$, $(3^{\circ })_{\lambda ,R}$, $(iv^{\circ })_R$, $(4^{\circ })_{\mu ,R}$, $(4^{\circ })_R$, $(5^{\circ })_{<\infty ,R}$, $(5^{\circ })_{\iota ,R}$ and $(5^{\circ })_R$ are equivalent.\end{itemize}\end{corollary}

\section{Conclusions}
\label{conclusions}

Determining what kinds of complete algebraic modular lattices are congruence lattices of semiprime algebras from semi--degenerate congruence--modular varieties may be of interest, since it will follow that the equivalences in Corollary \ref{transfercglat} hold for all those kinds of lattices. 

Another theme for future research is studying further extensions of Davey`s Theorem to different kinds of lattices, as well as finding more classes of algebras in which, given an appropriate setting (regarding definitions for annihilators and a Boolean center), Davey`s Theorem holds not only for congruences, but also for elements, as in the case of bounded distributive lattices, commutator lattices, residuated lattices and commutative unitary rings.

\section*{Acknowledgements}

This work was supported by the research grant ``Clonoids: a Unifying Approach to Equational Logic and Clones`` of the Austrian Science Fund FWF (P29931).

I thank Erhard Aichinger for very hepful discussions on the problems tackled in this paper. 

\end{document}